\documentclass[11pt]{article}
\usepackage{benstyle}
\usepackage{tcolorbox}

\newcommand{\Lt}[2]{L^{2}(#1,#2)}
\newcommand{\Ltr}{\Lt{\Omega}{\rho}}
\newcommand{\Ltt}{\Lt{\Omega}{\tau}}

\newcommand{\CNN}{\mathbb{C}^{N\times N}}

\newcommand{\CMN}{\mathbb{C}^{M\times N}}

\newcommand{\CN}{\mathbb{C}^{N}}
\newcommand{\CM}{\mathbb{C}^{M}}

\newcommand*{\bb}{\bm{b}} 
\newcommand*{\bc}{\bm{c}}

\newcommand*{\bx}{\bm{x}} 
\newcommand*{\by}{\bm{y}} 
\newcommand*{\bz}{\bm{z}}

\newcommand{\smM}{\sum_{m=1}^{M}}
\newcommand{\siN}{\sum_{i=1}^{N}}

\newcommand{\brac}[1]{\left\lbrace #1\right\rbrace }

\newcommand{\Nikw}{\mathcal{N}(P,\tau,w)}

\newcommand{\bG}{\bm{G}}
\newcommand{\bI}{\bm{I}}

\newcommand{\bX}{\bm{X}}

\newcommand{\lmi}[1]{\lambda_{\min}\left( #1\right)}
\newcommand{\lma}[1]{\lambda_{\max}\left(#1\right)}

\newcommand{\hop}[1]{\mathbb{E}(#1)}

\newcommand{\mP}[1]{\mathbb{P}\left(#1\right)}

\newcommand{\nmt}[1]{\nm{#1}_{\tau,w}}

\title{Near-optimal sampling strategies for multivariate function approximation on general domains}

\author{Ben Adcock and Juan M.\ Cardenas \\ Department of Mathematics \\ Simon Fraser University \\ Canada}

\begin{document}

\newpage

\setcounter{page}{1}

\maketitle

\begin{abstract}
In this paper, we address the problem of approximating a multivariate function defined on a general domain in $d$ dimensions from sample points.  We consider weighted least-squares approximation in an arbitrary finite-dimensional space $P$ from independent random samples taken according to a suitable measure. In general, least-squares approximations can be inaccurate and ill-conditioned when the number of sample points $M$ is close to $N = \dim(P)$. To counteract this, we introduce a novel method for sampling in general domains which leads to provably accurate and well-conditioned approximations.  The resulting sampling measure is discrete, and therefore straightforward to sample from.  Our main result shows near-optimal sample complexity for this procedure; specifically, $M = \mathcal{O}(N \log(N))$ samples suffice for a well-conditioned and accurate approximation.  Numerical experiments on polynomial approximation in general domains confirm the benefits of this method over standard sampling.
\end{abstract}

\section{Introduction}

In this paper, we consider the problem of approximating a multivariate function $f : \Omega \rightarrow \bbC$ of $d \geq 1$ variables whose domain $\Omega \subseteq \bbR^d$ may be irregular.  This problem arises in many applications in computational science and engineering, and presents two main challenges.  First, the well known \textit{curse of dimensionality}, and second, the potential \textit{irregularity} of the domain $\Omega$.  While there has been significant progress made towards mitigating the former (see \S \ref{ss:relation}), the majority of this work has focused on the case of tensor-product domains, for instance, the unit hypercube $\Omega = [-1,1]^d$.  Far less attention has been paid to the case of irregular domains.

Recently, in \cite{adcock2018approximating} the first author developed a framework for polynomial approximation of smooth functions in general domains in $d$ dimensions.  The approach, known as \textit{polynomial frame approximation}, is based on regularized least-squares approximation using orthonormal polynomials on a bounding hypercube and random sampling from the restriction of the orthogonality measure to $\Omega$.

For certain domains and polynomial spaces, this procedure has provable bounds on the \textit{sample complexity}; that is, the scaling between the dimension of the approximation space $N$ and the number of pointwise samples $M$ which is sufficient to guarantee a well conditioned and accurate approximation.  While these bounds are independent of the dimension $d$, and therefore ameliorate the curse of dimensionality {in the number of function evaluations}, the best known bounds are quadratic in $N$, i.e.\ $M = \ord{N^2 \log(N)}$, and are known to hold only for domains possessing the so-called \textit{$\lambda$-rectangle property} and polynomial spaces corresponding to \textit{lower sets} of multi-indices.
The reason for this can be traced to the choice of measure from which the sample points are drawn.  In the case of Legendre polynomials, for instance, this is simply the uniform measure over $\Omega$, which is known to be a relatively poor distribution for polynomial approximation.  

In recent works \cite{ArrasEtAlAdaptive,MiglioratiCohenOptimal,MiglioratiAdaptive}, it has been shown how construct a sampling measure depending on the space $P$ which leads to the near-optimal scaling $M = \ord{N \log(N)}$, where $N$ is the dimension of $P$.  Note that $P$ can be an arbitrary finite-dimensional subspace in this setup; it need not be a space of polynomials.  Unfortunately, the practical implementation of this approach requires two ingredients: first, an orthonormal basis for $P$, and second, a tensorial structure for the corresponding basis functions.  The latter is used in order to efficiently sample from the constructed measure.  Orthonormal polynomials on hypercubes typically exhibit both these qualities; for instance, the Legendre polynomials on $[-1,1]^d$ are simply the tensor-products of the univariate Legendre polynomials on $[-1,1]$, thus both tensorial and easy to construct.  However, for irregular domains, neither property holds in general.

In this work, we combine the ideas of \cite{adcock2018approximating} and \cite{MiglioratiCohenOptimal}, as well as those of \cite{MiglioratiAdaptive}, to construct a weighted least-squares approximation on general domains with the near-optimal sample complexity $M = \ord{N \log(N)}$. {Note that our main results guarantee, under this scaling, an error bound relating the $L^2$-norm error to a best approximation term in a certain weighted sup-norm (this is not quite optimal -- see Remark \ref{r:L2normerr}).}   
Our method is based on three steps.  First, using the results of \cite{adcock2018approximating} we generate a fine grid of $K \gg N$ points over the domain $\Omega$. {Throughout, we assume this step is computationally feasible; see \S \ref{s:conclusion} for further discussion on this topic.}  Second, starting from a nonorthogonal basis for the approximation space -- for example, as in \cite{adcock2018approximating,BADHframespart}, the restriction of an orthonormal basis on a bounding box to $\Omega$ -- we construct an orthonormal basis with respect to the corresponding discrete measure supported on the grid {(note that the approach in \cite{adcock2018approximating,BADHframespart} does not seek to orthogonalize the original basis).}  Third, we use the ideas of \cite{MiglioratiCohenOptimal,MiglioratiAdaptive} to generate a near-optimal sampling measure.  Unlike in these works, the resulting sampling measure is discrete, supported on the grid of $K$ points, and therefore straightforward to sample randomly from.  Following ideas from \cite{MiglioratiAdaptive}, we present two versions of our approach.  The first method (Method 1) considers a fixed approximation space $P$, while the second (Method 2) considers a sequence of nested spaces $P_1 \subset P_2 \subset P_3 \subset \ldots$.  The second method has the benefit of being \textit{adaptive}: all the sample points used to compute the approximation in the space $P_i$ are recycled when computing the approximation in $P_{i+1}$.

To demonstrate the effectiveness of these two methods, we present numerical experiments showing polynomial approximation on general domains in arbitrary dimensions.  For many domains in various dimensions, the new sampling methods achieve better accuracy and stability than when drawing samples from the uniform measure, as was done in \cite{adcock2018approximating}.  Furthermore, the adaptive method (Method 2) leads to no deterioration in accuracy or stability over Method 1.

\subsection{Related work}\label{ss:relation}

Motivated by applications in uncertainty quantification and parametric PDEs, least-squares polynomial approximation of high-dimensional, smooth functions has received significant attention over the last ten years.  The majority of works have focused on tensor-product domains.  Besides \cite{adcock2018approximating}, mentioned above, very few works have considered the question of general domains $\Omega$.  {Note that \cite{adcock2018approximating}, based on ideas of \cite{BADHframespart,BADHFramesPart2}, uses Singular Value Decomposition (SVD) followed by thresholding of the singular values to address the nonorthogonality of the approximation system. Conversely, in this paper, we use QR decomposition to explicitly orthogonalize the basis.  The resulting orthonormal basis is used to construct the near-optimal sampling measure and then to compute the final approximation.}

Since it is often a critical constraint in practice, a major focus of previous work on least-squares polynomial approximation has been quantifying the sample complexity when the samples are drawn randomly from the orthogonality measure of the polynomial basis employed.  See \cite{DavenportEtAlLeastSquares,MiglioratiThesis,MiglioratiJAT,MiglioratiEtAlFoCM} and references therein.  Unfortunately, these methods tend to have superlinear sample complexity.  Approaches at designing sampling measures which lead to log-linear sample complexity have been considered in, for instance, \cite{NarayanJakemanZhouChristoffelLS}, the aforementioned works \cite{MiglioratiCohenOptimal,MiglioratiAdaptive} and \cite{ArrasEtAlAdaptive}.

We remark in passing that least-squares approximation, as we consider in this paper, is but one approach for polynomial approximation in high dimensions.  A related, but distinct, line of work uses compressed sensing techniques for this problem.  See \cite{AdcockCSFunInterp,ABBCorrecting,BASBCWMatheon,ChkifaDownwardsCS,DoostanOwhadiSparse,HamptonDoostanCSPCE,NarayanZhouCCP} and references therein.  This approach is quite powerful, since, unlike least squares, it does not require one to specify \textit{a priori} the approximation space $P$.  However, it is not yet known how to perform (provably) optimal sampling in the compressed sensing setting.

{
Finally, we note that a similar approach to Method 1 of this paper, with corresponding theoretical analysis, has also been developed simultaneously by Migliorati \cite{MiglioratiIrregular}.
}

\subsection{Outline}

The outline of the remainder of this paper is as follows.  We first summarize our two methods, Method 1 and Method 2, in \S \ref{s:summary}.  In \S \ref{s:theory} we present the main theoretical analysis of these methods.  Proofs of the results presented in this section are given in \S \ref{sec:proofs}.  We conclude in \S \ref{s:numexamp} with numerical examples.

\section{Summary of the methods}\label{s:summary}

We now present our two main methods: Nonadaptive sampling (Method 1) and Adaptive sampling (Method 2).

\subsection{Method 1}

\begin{tcolorbox}[floatplacement=ht,float,title=Method 1.\ Nonadaptive sampling for general domains]
\noindent \textbf{Inputs:} Domain $\Omega$, probability measure $\rho$ over $\Omega$, function $f \in L^\infty(\Omega)$.
\\  Finite-dimensional subspace $P \subset L^2(\Omega,\rho)$ of dimension $N$.
\\ Basis $\{ \psi_1,\ldots,\psi_N \}$ for $P$ (not necessarily orthogonal).
\\ Number of sample points $M \geq N$, fine grid size $K \geq N$.
\\
\\
\noindent \textbf{Step 1:} Draw $K$ points $Z = \{ \bm{z}_i \}^{K}_{i=1}$ independently from $\rho$.  
\\
\noindent \textbf{Step 2:}
Construct the $K \times N$ matrix $\bm{B} = \{ \psi_j(\bm{z}_i) / \sqrt{K} \}^{K,N}_{i,j=1}$ and {check whether or not $\bm{B}$ has full rank ($\mathrm{rank}(\bm{B}) = N$). If not, go back to Step 1. Else proceed to Step 3. }
%
\\
\noindent {\textbf{Step 3:}} Compute the reduced QR decomposition $\bm{B} = \bm{Q} \bm{R}$, where $\bm{Q} = \{ q_{ij} \} \in \bbC^{K \times N}$ and $\bm{R} \in \bbC^{N \times N}$.  Define the probability distribution $\pi = \{\pi_i \}^{K}_{i=1}$ on $\{1,\ldots,K\}$ by
\bes{
\pi_i = \frac1N \sum^{N}_{j=1} |q_{ij} |^2,\quad i = 1,\ldots,K.
}
\noindent {\textbf{Step 4:}} Draw $M$ integers $i_1,\ldots,i_M$ independently from $\pi$, define
\bes{
\bm{A} = \left \{ \frac{q_{i_j,k}}{\sqrt{M \pi_{i_j}}}  \right \}^{M,N}_{j,k=1} \in \bbC^{M \times N},\quad \bm{b} = \left \{ \frac{f(\bm{z}_{i_j}) }{\sqrt{M K \pi_{i_j}}} \right \}^{M}_{j= 1} \in \bbC^M,
}
and compute $\bm{c} = \argmin{\bm{x} \in \bbC^N} \nm{\bm{A} \bm{x} - \bm{b}}_{2}$.
\\
\\
\textbf{Output:} The approximation $\tilde{f}(\bm{y}) = \sum^{N}_{i=1} c_i \phi_i(\bm{y})$, where $\phi_i(\bm{y}) = \sum^{i}_{j=1} (R^{-*})_{ij} \psi_j(\bm{y})$.
\end{tcolorbox}

In Method 1 (shown in the box below) we compute an approximation $\tilde{f} \in P$ to a function $f$ from a fixed subspace $P$ using the set of samples $\{ f(\bm{y}_i) \}^{M}_{i=1}$.  The sample points $\bm{y}_1,\ldots,\bm{y}_M$ are drawn independently and identically from the grid $Z$ according to probability distribution $\pi$.  In other words, $\bm{y}_i \sim \mu$, where $\mu$ is the discrete \textit{sampling measure}
\bes{
\D \mu(\bm{y}) = \sum^{K}_{i=1} \pi_i \delta(\bm{y} - \bm{z}_i),\qquad \bm{y} \in \Omega.
}
Observe that this measure is precisely
\bes{
\D \mu(\bm{y}) = \sum^{K}_{i=1} \left ( \frac1N \sum^{N}_{i=1} |\phi_i(\bm{y})|^2 \right )^{-1} \delta(\bm{y} - \bm{z}_i).
}
The function $\sum^{N}_{i=1} |\phi_i(\bm{y})|^2$ is the {reciprocal of the \textit{Christoffel function} of $P$ \cite{NevaiFreud}}, which was previously identified in \cite{MiglioratiCohenOptimal} as a suitable measure from which to obtain optimal sampling.

{Note that the reduced QR decomposition in Step 3 of Method 1 (and likewise for Method 2) refers to the decomposition $\bm{B} = \bm{Q} \bm{R}$, where $\bm{Q}$ is $K \times N$ with orthonormal columns, i.e.\ $\bm{Q}^* \bm{Q} = \bm{I}$ and $\bm{R}$ is $N \times N$ and upper triangular.}

Several remarks are in order.  First, both Method 1 and Method 2 (described next) assume it is possible to draw samples from the probability measure $\rho$.  This can be achieved via, for example, rejection sampling, as we do in our experiments later.  However, this may not be feasible in all settings, depending on the problem at hand.  {Second, in \S \ref{s:Zsize} we derive guarantees on $K$ which ensure $\bm{B}$ has full rank with high probability (see Step 2). Naturally, in practice, if $\bm{B}$ fails to be full rank then, rather than throwing away the current points, one may instead prefer to increase $K$ and draw additional points until full rankness is achieved.  See \cite{MiglioratiIrregular} for further discussion on this issue. Third,} note that if $\tilde{f}$ is only sought on the fine grid, then the computation of the functions $\phi_i$ in the final stage is unnecessary.  Since
\bes{
\{ \tilde{f}(\bm{z_i}) \}^{K}_{i=1} = \sqrt{K} Q \bm{c},
}
evaluating $\tilde{f}$ on this grid involves only a simple matrix-vector multiplication.  Fourth, we remark in passing that the scalings of the rows of $\bm{A}$ and $\bm{b}$ are to ensure good conditioning of $\bm{A}$, under suitable conditions on $M$ and $N$.  See Theorem \ref{t:Method1} below.  Finally, we note the computational cost {is (assuming the $\psi_j$'s are cheap to evaluate) dominated by the cost of computing the QR decomposition of $\bm{B}$, which involves an offline cost of $\ordu{K N^2}$ flops, and solving the least-squares problem to obtain $\bm{c}$, which involves an online cost of $\ordu{M N^2}$ flops.}

Two main questions we investigate in this paper are how large $M$ needs to be in comparison to $N$ and how large $K$ needs to be in comparison to $N$.  Note that the former pertains to the sample complexity of the method.  As we show later, with the probability distribution defined in Step 3, the \textit{log-linear} scaling $M \asymp N \log(N)$ is sufficient for a well conditioned approximation which also accurately approximates $f$ over the fine grid $Z$.  {Besides better sample complexity, this also compares favourably with the method of \cite{adcock2018approximating} in terms of the online computational cost, which is $\ord{N^3 \log(N)}$ for any domain, whereas the cost in \cite{adcock2018approximating}  is $\ord{N^4 \log(N)}$ and only for certain domains (i.e.\ those for which the sample complexity is provably quadratic in $N$).} 

To ensure the approximation is also accurate over $\Omega$, we need $K$ to be sufficiently large in relation to $N$.  Currently, we have no complete answer for general domains and spaces $P$.  However, if $\rho$ is the uniform measure and $P$ is a polynomial subspace based on a so-called \textit{lower set} of multi-indices (as is typical in practice), then we show that $K \asymp N^2 \lambda^{-1} \log(N)$ is sufficient, provided the domain $\Omega$ has the so-called \textit{$\lambda$-rectangle property}.  We discuss this further in \S \ref{s:Zsize}.

\subsection{Method 2}

Method 1 has the limitation that if the subspace $P$ is augmented to a larger space $\tilde{P} \supset P$, the existing sample points $\bm{y}_1,\ldots,\bm{y}_M$ are not sampled from the appropriate distribution for $\tilde{P}$.  In Method 2, following ideas of \cite{MiglioratiAdaptive}, we consider an adaptive procedure in which all samples are recycled as the $P$ is increased.

\begin{tcolorbox}[floatplacement=!th,float,title=Method 2.\ Adaptive sampling for general domains]
\noindent \textbf{Inputs:} Domain $\Omega$, probability measure $\rho$ over $\Omega$, function $f \in L^\infty(\Omega)$.
\\ Subspaces $P_1 \subset P_2 \subset \ldots \subset P_{r} \subset L^2(\Omega,\rho)$ of dimensions $N_1 < N_2 < \ldots < N_{r} < \infty$.
\\ A set of functions $\{ \psi_1,\ldots,\psi_{N_r} \}$ such that $\{ \psi_1,\ldots,\psi_{N_t} \}$ is a basis of $P_{N_t}$ for each $t = 1,\ldots,r$.
\\ Sampling ratios $k_1 \leq k_2 \leq \ldots \leq k_r$ with $k_t \in \bbN$, fine grid size $K \geq N_r$.
\\
\\
\noindent \textbf{Step 1:} Draw $K$ points $\{ \bm{z}_i \}^{K}_{i=1}$ independently from $\rho$.
\\
\noindent \textbf{Step 2:}
Construct the $K \times N_r$ matrix $\bm{B} = \{ \psi_j(\bm{z}_i) / \sqrt{K} \}^{K,N_r}_{i,j=1}$ and {check whether or not $\bm{B}$ has full rank ($\mathrm{rank}(\bm{B}) = N_r$). If not, go back to Step 1. Else set $N_0 = 0$, $k_0 = 0$, $M_{0} = 0$, $t = 1$ and proceed to Step 3. }
\\
\noindent {\textbf{Step 3:}} Construct the $K \times N_t$ matrix $\bm{B} = \{ \psi_j(\bm{z}_i) / \sqrt{K} \}^{K,N_t}_{i,j=1}$ and compute its reduced QR decomposition $\bm{B} = \bm{Q} \bm{R}$, where $\bm{Q} = \{ q_{ij} \} \in \bbC^{K \times N_t}$ and $\bm{R} \in \bbC^{N_t \times N_t}$.
\\
\noindent {\textbf{Step 4:}} For each $l = N_{t-1}+1,\ldots,N_t$ define probability distributions $\pi^{(l)} = \{\pi^{(l)}_i \}^{K}_{i=1}$ on $\{1,\ldots,K\}$ by
\bes{
\pi^{(l)}_i = |q_{i l} |^2,\quad i = 1,\ldots,K.
}
\noindent {\textbf{Step 5:}} Set $M_t = k_t N_t$.  For each $l = 1,\ldots,N_{t-1}$ draw $k_t - k_{t-1}$ integers independently from $\pi^{(l)}$, and for each $l = N_{t-1}+1,\ldots,N_t$ draw $k_t$ integers independently from $\pi^{(l)}$.  This gives $M_t - M_{t-1}$ new integers, and $M_t$ integers $i_1,\ldots,i_{M_t}$ in total.
\\
\noindent {\textbf{Step 6:}} 
Define
\bes{
\bm{A} = \left \{ \frac{q_{i_j,k}}{\sqrt{\frac{M_t}{N_t} \sum^{N_t}_{l=1} \pi^{(l)}_{i_j}} }  \right \}^{M_t,N_t}_{j,k=1} \in \bbC^{M_t \times N_t},\quad \bm{b} = \left \{ \frac{f(\bm{z}_{i_j}) }{ \sqrt{\frac{M_t K}{N_t} \sum^{N_t}_{l=1} \pi^{(l)}_{i_j}}   } \right \}^{M_t}_{j= 1} \in \bbC^{M_t}.
}
and compute $\bm{c}^{(t)} = \argmin{\bm{x} \in \bbC^{N_t}} \nm{\bm{A} \bm{x} - \bm{b}}_{2}$.
\\
\noindent {\textbf{Step 7:}}  If $t < r$ increment $t$ by one and repeat {Steps 3--7}.
\\
\\
\textbf{Output:} The approximations
$\tilde{f}^{(t)}(\bm{y}) = \sum^{N_t}_{i=1} c^{(t)}_i \phi_i(\bm{y})$, $t = 1,\ldots,r$, where $\phi_i(\bm{y}) = \sum^{i}_{j=1} (R^{-*})_{ij} \psi_j(\bm{y})$.
\end{tcolorbox}

In this method, the sample complexity for subspace $P_t$ is $M_t = k_t N_t$.  As we show later, a suitable choice of $k_t$ to ensure a sequence of well conditioned and accurate approximations is $k_t \asymp \log(N_t)$.  We note also that the QR decomposition computed in {Step 3} need not be done from scratch at each step.  One can use standard methods to update the decomposition according to the new columns added at each step.  See, for example, \cite[Chpt.\ 24]{lawson1995solving}.

\subsection{Main theoretical results}

Having presented our two methods, we now summarize their stability and accuracy, and in particular, the conditions on $M$ and $K$.  This is the topic of the following two theorems.  To this end, we now define the \textit{Nikolskii constant} $\cN(P,\rho)$ as the smallest possible constant such that
\be{
\label{NikrhoIntro}
\nm{p}_{L^{\infty}(\Omega)} \leq \cN(P,\rho) \nm{p}_{\Ltr}, \qquad \forall p \in P.
}
See \S \ref{s:Zsize} for further information.

\thm{
\label{t:Method1}
Consider the setup of Method 1.  Suppose that {$ \gamma,\delta \in (0,1)$} and
\eas{
M &\geq N \log(4N/\gamma) \left ( (1+\delta) \log(1+\delta) - \delta \right )^{-1},
\\
K &\geq (\cN(P,\rho))^2 \log(2 N/\gamma) \left ( (1-\delta) \log(1-\delta) + \delta \right )^{-1},
}
where $\cN(P,\rho)$ is as in \R{NikrhoIntro}.  Then the following holds with probability at least $1-\gamma$:
\begin{enumerate}
\item[(i)] the matrix $\bm{B}$ is full rank,
\item[(ii)] the condition number of the matrix $\bm{A}$ satisfies $\kappa(\bm{A}) \leq \sqrt{\frac{1+\delta}{1-\delta}}$, 
\item[(iii)] for any $f \in L^{\infty}(\Omega)$ the approximation $\tilde{f}$ is unique and satisfies
\bes{
\nmu{f - \tilde{f}}_{L^2(\Omega,\rho)} \leq \inf_{p \in P} \left \{ \nm{f - p}_{L^2(\Omega,\rho)} + \frac{1}{1-\delta} \tnm{f - p}_{Z,\pi} \right \},
}
where $\tnm{g}_{Z,\pi} = \max_{i=1,\ldots,K} \left \{ \frac{|g(\bm{z}_i)|}{\sqrt{K \pi_i}} \right \}$.
\end{enumerate}
}

\thm{
\label{t:Method2}
Consider the setup of Method 2.  Suppose that {$ \gamma,\delta \in (0,1)$, $\gamma_1,\ldots,\gamma_r \in (0,1)$} with $\sum^{r}_{t=1} \gamma_t = \gamma$ and
\eas{
k_t &\geq \log(4N/\gamma_t) \left ( (1+\delta) \log(1+\delta) - \delta \right )^{-1},\qquad t = 1,\ldots,r,
\\
K &\geq (\cN(P_r,\rho))^2 \log(2 N_r/\gamma) \left ( (1-\delta) \log(1-\delta) + \delta \right )^{-1},
}
where $\cN(P_r,\rho)$ is as in \R{NikrhoIntro} with $P = P_r$.  Then the following holds with probability at least $1-\gamma$.  For every $t = 1,\ldots,r$,
\begin{enumerate}
\item[(i)]  the matrix $\bm{B}$ of {Step 3} is full rank,
\item[(ii)] the condition number of the matrix $\bm{A}$ in {Step 6} satisfies $\kappa(\bm{A}) \leq \sqrt{\frac{1+\delta}{1-\delta}}$, 
\item[(iii)] for any $f \in L^{\infty}(\Omega)$ the approximation $\tilde{f}^{(t)}$ is unique and satisfies
\bes{
\nmu{f - \tilde{f}^{(t)}}_{L^2(\Omega,\rho)} \leq \inf_{p \in P_t} \left \{ \nm{f - p}_{L^2(\Omega,\rho)} + \frac{1}{1-\delta} \tnm{f - p}_{Z,\pi,t} \right \},
}
where $\tnm{g}_{Z,\pi,t} = \max_{i=1,\ldots,K} \left \{ \frac{|g(\bm{z}_i)|}{\sqrt{\frac{K}{N_t} \sum^{N_t}_{l=1} \pi^{(l)}_i}} \right \}$.
\end{enumerate}
}

\section{Theoretical analysis}\label{s:theory}

We now present our main analysis.  To do so, we introduce a more general framework than that considered above, in which we consider three distinct quantities: an \textit{error measure} $\rho$, an \textit{orthogonality measure} $\tau$, and \textit{sampling measures} $\mu_1,\ldots,\mu_M$.  Both Methods 1 and Methods 2 correspond to specific cases of this framework in which the orthogonality measure $\tau$ is a discrete measure over the grid $Z$.  The difference between the two methods lies with the choices of the sampling measures $\mu_i$.  Note that this framework is general, and includes those of \cite{MiglioratiCohenOptimal,MiglioratiAdaptive} as special cases.  Specifically, they correspond to the choice $\tau = \rho$ and specific choices of the sampling measures $\mu_i$ (see later).  The flexibility gained by allowing a distinct orthogonality measure $\tau$ is what leads to Methods 1 and 2.

\subsection{General setup}

Consider the space $\Ltr$ of square-integrable functions over a domain $\Omega \subseteq \bbR^d$ with respect to a probability measure $\rho$.  We refer to $\rho$ as the \textit{error measure}: it gives the norm in which we measure the error of our approximation.
Next, we define a second measure $\tau$, the \textit{orthogonality measure}, over $\Omega$.  We assume $\tau$ is a probability measure, $\int_{\Omega} \D \tau = 1$.  This is the measure which we shall subsequently use to construct an orthonormal basis of the approximation space.  Specifically, let $P \subset L^\infty(\Omega)$ be the \textit{approximation space} of dimension $\dim(P) = N < \infty$.  We write
\bes{
P = \spn \{ \phi_1,\ldots,\phi_N \},
}
where $\{ \phi_i \}^{N}_{i=1}$ is the corresponding orthonormal basis for $P$ in $L^2(\Omega,\tau)$.
Finally, we define \textit{sampling measures} $\mu_1,\ldots,\mu_M$ over $\supp(\tau)$, the support of the measure $\tau$.  These are also probability measures.  The $i^{\rth}$ such measure $\mu_i$ is the measure from which the $i^{\rth}$ sample will be drawn.  When later $\tau$ is taken as a discrete measure, this means that the sampling measures will also be discrete measures.  We also assume that there exists a function $w$ that is positive and defined everywhere on $\supp(\tau)$ and satisfies
\be{
\label{Measure}
\frac1M \sum^{M}_{i=1} \D \mu_i(\bm{y}) = \frac{1}{w(\bm{y})} \D \tau(\bm{y}),\qquad \forall \bm{y} \in \supp(\tau).
}
Note that this implies that $\int_{\Omega} w^{-1} \D \tau = 1$.  

We are now ready to define our approximation.  Let $M \geq N$ and draw $M$ points $\bm{y}_1,\ldots,\bm{y}_M$ independently, with $\bm{y}_i$ drawn according to the $i^{\rth}$ sampling measure $\mu_i$.  We then define the weighted least-squares approximation of a function $f \in L^{\infty}(\Omega)$ as
\be{
\label{GenLSprob}
\tilde{f} \in \argmin{p\in P} \left \{ \frac{1}{M}\sum^{M}_{i=1} w(\by_i)\left|f(\by_i)-p(\by_i) \right|^{2} \right \}.
}
Write $\tilde{f} = \sum^{N}_{i=1} c_i \phi_i$.  Then this is equivalent to the algebraic weighted least-squares problem
\bes{
\bm{c} = (c_i)^{N}_{i=1} \in  \argmin{\bx \in \CN}\nm{\bm{A}\bx-\bb}_{2}{},
}
where
\be{
\label{Adef}
\bm{A}=\brac{\frac{1}{\sqrt{M}}\sqrt{w(\by_i)}\phi_j(\by_i)}^{M,N}_{i,j = 1}\in\CMN,\qquad \bb=\brac{\frac{1}{\sqrt{M}}\sqrt{w(\by_i)}f(\by_i)}^{M}_{i=1}\in \CM .
}
For convenience, we now also define the discrete semi-inner product and semi-norm
\be{
\ip{f}{g}_{\Upsilon,w} = \frac{1}{M}\sum^{M}_{i=1} w(\by_i)f(\by_i)\overline{g(\by_i)},\qquad \nm{f}_{\Upsilon,w} = \sqrt{\ip{f}{f}_{\Upsilon,w} },
}
where $\Upsilon = \{ \bm{y}_1,\ldots,\bm{y}_{M}\}$.
Note that $\tilde{f}$ is can be expressed equivalently as
\be{
\tilde{f} \in \argmin{p \in P}\nm{f-p}_{\Upsilon,w}.
}
Given a domain $\Omega$, an error measure $\rho$ and an approximation space $P$, we are free to choose $\tau$ and the $\mu_i$.  This raises the following question: \textit{how should one choose the orthogonality measure $\tau$ and the sampling measures $\mu_1,\ldots,\mu_M$?}
There are two constraints to keep in mind.  First, we wish to take as few samples $M$ as possible.  Second, we need probability measures $\mu_i$ from  which it is not computationally intensive to draw samples.
We address the former determining how the error of the weighted least-squares approximation $\tilde{f}$ depends on these quantities, and in particular, how the number of samples $M$ influences the error bound.  For the second, as noted, we construct $\tau$ as a discrete measure over a suitable grid.

\subsection{Error and sample complexity estimates}

We first define the constant
\be{
\label{Cdef}
\cC := \sup \brac{ \frac{\nm{p}_{\Ltt}}{\nm{p}_{\Upsilon,w}} : p \in P,\ p|_{\supp(\tau)} \neq 0 }.
}
Notice that $\cC < \infty$ if and only if $\nm{\cdot}_{\Upsilon,w}$ is a norm on $P \subset L^2(\Omega,\tau)$, which in turn is a necessary and sufficient condition for the least-squares problem \R{GenLSprob} to have a unique solution.

This constant relates the orthogonality measure $\tau$ to the sampling measures $\mu_i$.  We also need a constant relating the error measure $\rho$ to $\tau$.  We define
\be{
\label{Ddef}
\cD := \sup \brac{\frac{\nm{p}_{\Ltr}}{\nm{p}_{L^2(\Omega,\tau)}} : p \in P,\ p|_{\supp(\rho)} \neq 0 }.
}
As above, notice that $\cD < \infty$ if and only if $\nm{\cdot}_{L^2(\Omega,\tau)}$ is a norm on $P \subset L^2(\Omega,\rho)$.

\begin{theorem}\label{Thm_errorbound}
Suppose that the constant $\cC$ defined in \R{Cdef} satisfies $\cC < \infty$ and let $f \in L^{\infty}(\Omega)$.  Then the approximation $\tilde{f}$ is unique and satisfies
\bes{
\nmu{f-\tilde{f}}_{\Ltt} \leq (1 + \cC) \inf_{\substack{p \in P}} \nm{f-p}_{\tau,w}.
}
If in addition the constant $\cD$ defined in \R{Ddef} satisfies $\cD < \infty$ then
\bes{
\nmu{f-\tilde{f}}_{\Ltr} \leq \inf_{p \in P} \left \{ \nm{f - p}_{\Ltr} + \cC \cD \nm{f - p}_{\tau,w} \right \},
}
where $\nm{g}_{\tau,w} = \sup_{\bm{y} \in \supp(\tau)} \sqrt{w(\bm{y})} | g(\bm{y}) |$.
\end{theorem}

See \S \ref{sec:proofs} for the proof.  This result states that firstly the approximation error in $L^2(\Omega,\tau)$ is determined by the constant $\cC$, which relates the $L^2$-norm with respect to $\tau$ to the discrete $L^2$-norm over the sample points, and the best approximation error measured in the $\nm{\cdot}_{\tau,w}$ norm, a weighted sup-norm over the support of $\tau$.  Secondly, the error in $L^2(\Omega,\rho)$ is determined by the same factors multiplied by the additional constant $\cD$, which relates the $L^2$-norms over $\rho$ and $\tau$.  

\rem{
\label{r:L2normerr}
{The above estimate, which bounds the $L^2$-norm error in terms of a weighted sup-norm, is not quite optimal.}  It is possible to obtain estimates (in expectation) involving solely $L^2$-norms by slightly modifying the least-squares estimator $\tilde{f}$.  For succinctness we shall not do this.  See, for instance, \cite{DavenportEtAlLeastSquares,MiglioratiAdaptive,MiglioratiIrregular}. 
}

We now move on to the question of optimal sampling.  As can be seen in the previous theorem, the samples influence the size of the constant $\cC$.  In the following theorem, we determine a sufficient condition on the sampling measures $\mu_i$ which guarantees that $\cC \lesssim 1$.

We now make the following standard assumption about the subspace $P \subset L^2(\Omega,\tau)$:
\be{
\label{Pass}
\mbox{For any $\bm{y} \in \supp(\tau)$ there exists a $p \in P$ with $p(\bm{y}) \neq 0$.}
}
Note that this implies that the function $\sum^{N}_{i=1} | \phi_i(\bm{y}) |^2 > 0$ on $\supp(\tau)$, for any orthonormal basis $\{ \phi_i \}^{N}_{i=1}$ of $P$ with respect to $\tau$.  In particular, the function 
\be{
\label{woptimalchoice}
w(\by)=\left( \frac{1}{N}\sum_{i=1}^{N}\left|\phi_i(\by)\right|^2\right)^{-1},\qquad \bm{y} \in \supp(\tau),
}
is positive and defined everywhere on $\mathrm{supp}(\tau)$.  Notice also that
\[ \int_{\Omega}w^{-1}(\by)d\tau(\by)
= \int_{\Omega}\frac{1}{N}\siN\left|\phi_i(\by)\right|^2 d\tau(\by)
= 1.
\]  
This function is also independent of the orthonormal basis used.  Indeed, $\sum_{i=1}^{N}\left|\phi_i(\by)\right|^2$ is {the reciprocal of the \textit{Christoffel function} \cite{NevaiFreud}} of the subspace $P \subset L^2(\Omega,\tau)$ .

\thm{
\label{Thm_optimalsampling}
Let {$ \gamma,\delta \in (0,1)$}, $P \subset L^2(\Omega,\tau)$ be such that \R{Pass} holds and $\{ \phi_i \}^{N}_{i=1}$ of $P$ with respect to the orthogonality measure $\tau$.  Let $w$ be as in \R{woptimalchoice} and suppose that $\mu_1,\ldots,\mu_M$ are probability measures satisfying \R{Measure} for this choice of $w$.  If
\be{
\label{MAW1}
M \geq N\log(N/\gamma) \left ((1-\delta )\log(1-\delta )+\delta \right )^{-1},
}
then, with probability at least {$1-\gamma$}, the constant $\cC$ defined in \R{Cdef} satisfies $\cC \leq \frac{1}{\sqrt{1-\delta}}$.  Moreover, if $M$ satisfies the slightly stricter condition
\be{
\label{MAW2}
M \geq N \log(2N/\gamma) \left ((1+\delta )\log(1+\delta )-\delta \right )^{-1}, 
}
then, with probability at least $1-\gamma$, the constant $\cC$ satisfies $\cC \leq \frac{1}{\sqrt{1-\delta}}$ and the condition number of the matrix $\bm{A}$ defined by \R{Adef} satisfies $\kappa(\bm{A}) \leq \sqrt{\frac{1+\delta}{1-\delta}}$.
}

Note that the condition {\R{MAW2} is stricter than the condition \R{MAW1}}.    We defer the proof of this theorem to \S \ref{sec:proofs}.

\subsection{Choice of sampling measure}

Theorem \ref{Thm_optimalsampling} implies that $M \gtrsim N \log(N)$ samples are sufficient for a small constant $\cC$, provided the sampling measures $\mu_i$ are such that
\be{
\label{meassatisfy}
\frac1M \sum^{M}_{i=1} \D \mu_i(\bm{y}) =  \frac{1}{N}\sum_{i=1}^{N}\left|\phi_i(\by)\right|^2 \D \tau(\bm{y}),\qquad \forall \bm{y} \in \supp(\tau).
}
This in turn implies that the least-square estimator $\tilde{f}$ is a quasi-best approximation over $L^2(\Omega,\tau)$ and, provided $\cD \lesssim 1$, also a quasi-best approximation over $L^2(\Omega,\rho)$.  We discuss the constant $\cD$ in the next section.

Before doing so, let us consider the case $\tau = \rho$, so that $\cD = 1$.  Then one choice of sampling measure that satisfies \R{meassatisfy} is simply $\mu_1 = \ldots = \mu_M = \mu$, where
\bes{
\D \mu(\bm{y}) = \frac{1}{N}\sum_{i=1}^{N}\left|\phi_i(\by)\right|^2 \D \rho(\bm{y}),\quad \bm{y} \in \Omega.
}
This is the optimal sampling measure introduced in \cite{MiglioratiCohenOptimal}.  
As noted, a disadvantage of this measure is that it is \textit{nonadaptive}.  If the $N$ increases, the measure $\mu$ changes, and one has to discard the existing samples (in practice, one can recycle at least some of these samples -- see \cite{ArrasEtAlAdaptive}).
An alternative approach, which avoids this problem, is the following.  First, fix $k \in \bbN$ and let $M = k N$.  Then, let
\bes{
\D \mu_i(\bm{y}) =  | \phi_j(\bm{y}) |^2 \D \rho(\bm{y}),\quad \bm{y} \in \Omega,\qquad (j-1) k < i \leq j k,\ j = 1,\ldots,N,
}
so that the first $k$ points are drawn from the measure $ | \phi_1 |^2 \D \rho$, the next $k$ points are drawn from the measure $| \phi_2 |^2 \D \rho$ and so forth.  
Observe that this choice of measures satisfies \R{meassatisfy}:
\bes{
\frac1M \D \mu_i(\bm{y}) = \frac1M \sum^{N}_{j=1} k |\phi_j(\bm{y}) |^2 \D \rho(\bm{y}) = \frac1N \sum^{N}_{j=1} |\phi_j(\bm{y}) |^2 \D \rho(\bm{y}).
}
This approach was introduced in \cite{MiglioratiAdaptive}.  It is clearly \textit{adaptive}, since if $N$ is incremented by one, we need only  sample an additional $k$ points from the new measure $| \phi_{N+1} |^2 \D \rho$.

\subsection{Discrete orthogonality measures}\label{ss:discorth}

Both of the above approaches require a known orthonormal basis for $P$ and the ability to sample from the corresponding measures in a computationally efficient manner.  Neither is typically the case when $\rho$ is a continuous measure on a general domain.
To avoid this issue, we now reintroduce the orthogonality measure $\tau$.  We construct this as a discrete measure based on a grid $Z = \{ \bm{z}_i \}^{K}_{i=1} \subset \Omega$, where the $\bm{z}_i$ are independently and identically drawn from the error measure $\rho$.  It is worth noting that sampling from $\rho$ may not be trivial in practice.  In our experiments, we use rejection sampling.  We shall not dwell on this issue any further, since it is domain (and therefore application) specific {(see also \S \ref{s:conclusion})}.  We also note that the use of a random grid here is simply to allow one to bound the constant $\cD$.  Deterministic grids are also permitted within this framework, although designing a good grid with $\cD \lesssim 1$ provably may be nontrivial.

Given such a grid $Z$, we define
\bes{
\D \tau(\bm{y}) = \frac{1}{K} \sum^{K}_{i=1} \delta(\bm{y}- \bm{z}_i).
}
We first describe the construction of the orthonormal basis $\{\phi_1,\ldots,\phi_N \}$ for $P \subset L^2(\Omega,\tau)$.  First, let $\{ \psi_1,\ldots,\psi_N \}$ be a basis for $P$ in $L^2(\Omega,\rho)$, and set
\bes{
\bm{B} = \left \{ \psi_j(\bm{z}_i) / \sqrt{K} \right \}^{K,N}_{i,j=1} \in \bbC^{K \times N}.
}
We assume henceforth that $\bm{B}$ is full rank, $\mathrm{rank}(\bm{B}) = N$.  Note that this is equivalent to the condition $\cD < \infty$.  Indeed,
\bes{
\bm{B} \bm{c} = 0 \quad \Leftrightarrow \quad \frac1K \sum^{K}_{i=1} | p(\bm{z}_i) |^2 = 0,\ p = \sum^{N}_{i=1} c_i \psi_i \quad \Leftrightarrow \quad \nm{p}_{L^2(\Omega,\tau)} = 0.
}
In Proposition \ref{p:Ksize} we give a sufficient condition on $K$ for this to occur.  Now let $\bm{B}$ have reduced QR decomposition $\bm{B} = \bm{Q} \bm{R}$, where $\bm{Q} = \{ q_{ij} \} \in \bbC^{K \times N}$ and $\bm{R} \in \bbC^{N \times N}$.  Then it follows straightforwardly that the functions $\phi_i$ are given by
\bes{
\phi_i(\bm{y}) = \sum^{i}_{j=1} (R^{-*})_{ij} \psi_j(\bm{y}),\quad i = 1,\ldots,N.
}

\subsection{Derivation of Methods 1 and 2}

We now complete the derivation of Methods 1 and 2.  For both methods, we first notice that the function $w(\bm{y})$ defined by \R{woptimalchoice} satisfies
\be{
\label{wMethods12}
\frac{1}{w(\bm{z}_i)} = \frac{K}{N} \sum^{N}_{j=1} | q_{ij} |^2,\quad i = 1,\ldots,K.
}
We next consider each method separately:

\subsubsection*{Method 1}
We let $\mu_1 = \ldots = \mu_M = \mu$, where
\bes{
\D \mu(\bm{y}) = \frac{1}{N}\sum_{i=1}^{N}\left|\phi_i(\by)\right|^2 \D \tau(\bm{y}) =\sum^{K}_{i=1} \frac{1}{K w(\bm{z}_i)} \delta(\bm{y} - \bm{z}_i) \D \bm{y},
}
and $w(\bm{y})$ is as in \R{woptimalchoice}.  Let $\pi = \{ \pi_i \}^{K}_{i=1}$ be the probability distribution on $\{1,\ldots,K \}$ with
\be{
\label{piiMeth1}
\pi_i = \frac{1}{K w(\bm{z}_i) } = \frac{1}{N} \sum^{N}_{j=1} | q_{ij} |^2,\qquad i = 1,\ldots,K.
}
Then, random sampling $\bm{y} \sim \mu$ is effected by randomly choosing an integer $i \in \{1,\ldots,K\}$ according to $\pi$ and then setting $\bm{y} = \bm{z}_i$.  Let $i_1,\ldots,i_M$ be $M$ integers drawn independently from $\{1,\ldots,K\}$ according to $\pi$ and $\bm{y}_1,\ldots,\bm{y}_M$ be the sample points.  Observe that 
\eas{
\bm{A} &= \left \{ \frac{1}{\sqrt{M}} \sqrt{w(\bm{y}_j)} \phi_k(\bm{y}_j)  \right \}^{M,N}_{j,k=1}= \left \{ \frac{q_{i_j,k}}{\sqrt{M \pi_{i_j}} }  \right \}^{M,N}_{j,k=1},
\\
\bm{b} &= \left \{ \frac{1}{\sqrt{M}} \sqrt{w(\bm{y}_i)} f(\bm{y}_j) \right \}^{M}_{j=1} = \left \{ \frac{f(\bm{z}_{i_j}) }{\sqrt{M K \pi_{i_j}}} \right \}^{M}_{j=1}.
}
This completes the derivation of Method 1.

\subsubsection*{Method 2}

In this case, we fix $M = k N$ for some $k \in \bbN$ and define, for $(l-1) k < i \leq l k$ and $l = 1,\ldots,N$, the sampling measures
\bes{
\D \mu_i(\bm{y}) = | \phi_l(\bm{y}) |^2 \D \tau(\bm{y}) = \frac{1}{K} \sum^{K}_{i=1} | \phi_l(\bm{z}_i) |^2 \delta(\bm{y}-\bm{z}_i) \D \bm{y}.
}
For each $l$, we define the probability distribution $\pi^{(l)} = \{\pi^{(l)}_i \}^{K}_{i=1}$ on $\{1,\ldots,K\}$ as
\bes{
\pi^{(l)}_{i} = \frac{1}{K} | \phi_l(\bm{z}_i) |^2 = | q_{il} |^2,\quad i = 1,\ldots,K.
}
Thus, drawing a sample from the $\mu_i$, $(l-1) k < i \leq l k$, is equivalent to $\bm{y} = \bm{z}_i$, where $i \sim \pi^{(l)}$.  Let $i_1,\ldots,i_M$ be the $M$ integers drawn according to the $\pi^{(l)}$.  Then we have
\eas{
\bm{A} &= \left \{ \frac{1}{\sqrt{M}} \sqrt{w(\bm{y}_j)} \phi_k(\bm{y}_j)  \right \}^{M,N}_{j,k=1} = \left \{ \frac{q_{i_j,k}}{\sqrt{\frac{M}{N} \sum^{N}_{l=1} \pi^{(l)}_{i_j} }} \right \}^{M,N}_{j,k=1},
\\
\bm{b} &= \left \{ \frac{1}{\sqrt{M}} \sqrt{w(\bm{y}_i)} f(\bm{y}_j) \right \}^{M}_{j=1} = \left \{ \frac{f(\bm{z}_{i_j})}{\sqrt{\frac{M K}{N} \sum^{N}_{l=1} \pi^{(l)}_{i_j} }} \right \}^{M}_{j=1}.
}
Up to the small modifications needed to make the method adaptive, this completes the derivation of Method 2.

\subsection{The size of the grid $Z$}\label{s:Zsize}

The size $K$ of $Z$ influences the magnitude of the constant $\cD$.  We now estimate this term.  For this, we use the following \textit{Nikolskii-type} inequality for the space $P \subset L^2(\Omega,\rho)$.  We let $\cN(P,\rho)$ be the smallest positive number such that 
\be{
\label{Nikrho}
\nm{p}_{L^{\infty}(\Omega)} \leq \cN(P,\rho) \nm{p}_{\Ltr}, \qquad \forall p \in P.
}

\prop{
\label{p:Ksize}
Let {$ \gamma,\delta \in (0,1)$} and $Z = \{ \bm{z}_i \}^{K}_{i=1}$ where the $\bm{z}_i$ are drawn independently and identically according to the measure $\rho$ on $\Omega$.  If 
\bes{
K \geq (\cN(P,\rho))^2 ((1-\delta )\log(1-\delta )+\delta )^{-1}\log(N/\gamma),
}
where $N=\dim(P)$, then with probability at least $1-\gamma$ the constant $\cD$ satisfies $\cD \leq \frac{1}{\sqrt{1-\delta}}$.
}

See \cite[Thm.\ 6.2]{adcock2018approximating}.  This reduces the question of how large to choose $K$ to that of determining the Nikolskii constant $\cN(P,\rho)$ for a measure $\rho$ over a domain $\Omega$.  As discussed in \cite{adcock2018approximating}, there are no generic results on this for arbitrary domains and measures.  However, in certain cases, one can show that $(\cN(P,\rho))^2$ is at most quadratic in $N$, the dimension of $P$:

\defn{
[$\lambda$-rectangle property]
A compact domain $\Omega$ has the $\lambda$-rectangle property for some $0<\lambda<1 $ if it can be written as a (possibly overlapping and uncountable) union $\Omega = \bigcup_{R \in \cR} R$ of hyperrectangles $R$ satisfying $\inf_{R \in \cR} \mathrm{Vol}(R) = \lambda \mathrm{Vol}(\Omega)$.
}

See \cite[Defn.\ 6.5]{adcock2018approximating}.  The following is \cite[Thm.\ 6.6]{adcock2018approximating}:

\prop{
\label{p:Klambdarect}
Suppose that $\Omega$ has the $\lambda$-rectangle property and let $P$ be the polynomial space $P = \spn \{ \bm{y} \mapsto \bm{y}^{\bm{n}} : \bm{n} \in \Lambda \}$,
where $\Lambda \subset \bbN^d_0$, $| \Lambda | = N$ is a lower set of multi-indices\footnote{That is, if $\bm{n} \in \Lambda$ and $\bm{n}' \leq \bm{n}$ then $\bm{n}' \in \Lambda$.}.  Let $\rho$ be the uniform probability measure on $\Omega$.  Then $(\cN(P,\rho))^2 \leq N^2 / \lambda$.
}

We remark in passing that most standard polynomial spaces correspond to lower sets, e.g.\ tensor product, total degree, hyperbolic cross, and so forth.

\rem{
Unfortunately, while many irregular domains have the the $\lambda$-rectangle property, some simple domains such a balls and simplicies do not \cite{adcock2018approximating}.
Various of results on the Nikolskii constant (or more generally, the Christoffel function) are known for certain irregular domains, although typically only for \textit{total degree} polynomial spaces, i.e.\ those for which $\Lambda = \Lambda_n = \{ \bm{n} = (n_1,\ldots,n_d) : n_1+\ldots + n_d \leq n \}$.  See, for example, \cite{XuOrthPoly} for results when $\Omega$ is a ball or simplex, \cite{PrymakChristoffel} for planar domains with piecewise smooth boundaries, \cite{KrooChristoffelStar} for convex and starlike domains and \cite{DaiNikolskii} when $\Omega$ is the surface of the sphere. 
It is an open problem to determine the Nikolskii constant for more general domains and subspaces $P$.
}

\section{Proofs of the main results}\label{sec:proofs}

We now prove the main results.  The proofs are based on similar ideas to those found in previous works on least-squares approximation.  See, for instance, \cite{adcock2018approximating,DavenportEtAlLeastSquares,MiglioratiCohenOptimal,MiglioratiAdaptive}.

\prf{[Proof of Theorem \ref{Thm_errorbound}]
Fix $p \in P$.  Then
\be{
\label{redphone}
\nmu{f-\tilde{f}}_{\Ltt} \leq  \nm{f-p}_{\Ltt}+ \nmu{\tilde{f}-p}_{\Ltt} .
}
We bound the first term using \R{Measure} and the fact that the $\mu_i$ are probability measures:
\eas{
\nm{f-p}_{\Ltt}^{2} = \frac1M \sum^{M}_{i=1} \int_{\Omega}|f(\by)-p(\by)|^{2}w(\by) \D \mu_i(\bm{y}) \leq \sup_{\by \in \supp(\tau)}w(\by)|f(\by)-p(\by)|^2 .
}
Hence $\nm{f-p}_{\Ltt}\leq \nmt{f-p}$.  For the second term, we first observe that $\ip{\tilde{f}}{p}_{\Upsilon,w} = \ip{f}{p}_{\Upsilon,w}$, $\forall p \in P$,
since $\tilde{f}$ is a discrete least-squares approximation, and therefore satisfies the normal equations.  In particular,
\bes{
\nmu{\tilde{f} - p}^2_{\Upsilon,w} = \ip{f - p}{\tilde{f} - p}_{\Upsilon,w} \leq \nm{f - p}_{\Upsilon,w} \nmu{\tilde{f} - p}_{\Upsilon,w},
}
and therefore $\nmu{\tilde{f} - p}_{\Upsilon,w} \leq \nm{f - p}_{\Upsilon,w}$.  Hence, by the definition of $\cC$,
\bes{
\nmu{\tilde{f} - p}_{\Ltt} \leq \cC \nmu{\tilde{f} - p}_{\Upsilon,w} \leq \cC \nmu{f - p}_{\Upsilon,w}.
}
Furthermore, we have 
\bes{
\nmu{f-p}_{\Upsilon,w}^2 = \frac{1}{M}\sum^{M}_{i=1} w(\by_i)|f(\by_i)-p(\by_i)|^2 
\leq \sup_{\substack{\by \in \supp(\tau)}} w(\by)|f(\by)-p(\by)|^2  
= \nmt{f-p}^2.
}
Combining this with the previous estimate gives $\nmu{\tilde{f} - p}_{\Ltt} \leq \cC \nmu{f - p}_{\tau,w}$.  Substituting this into \R{redphone} completes the proof of the first result.

We now consider the second result.  We have
\eas{
\nmu{f-\tilde{f}}_{\Ltr} \leq  \nmu{f-p}_{\Ltr}+\nmu{p-\tilde{f}}_{\Ltr} & \leq  \nm{f-p}_{\Ltr}+ \cD \nmu{p - \tilde{f}}_{\Ltt}
\\
& \leq \nm{f-p}_{\Ltr}+ \cC \cD \nmu{p - \tilde{f}}_{\Upsilon,w}.
}
Thus, using the earlier arguments, we deduce that
\bes{
\nmu{f-\tilde{f}}_{\Ltr} \leq \nmu{f-p}_{\Ltr} + \cC \cD \nmu{f - p}_{\tau,w},
}
as required.
}

We now prove Theorem \ref{Thm_optimalsampling}.  For this, we use the following \textit{weighted} Nikolskii-type inequality for the space $P \subset L^2(\Omega,\tau)$.  For the moment, consider an arbitrary positive function $w$ defined everywhere on $\Omega$, and let $\Nikw$ be the smallest positive number such that
\begin{equation}
\sup_{\bz \in \supp(\tau)}\sqrt{w(\bz)}|p(\bz)| \leq \Nikw \nm{p}_{\Ltt}, \qquad \forall p \in P. \label{inNik}
\end{equation}
Note that the earlier Nikolskii inequality \R{Nikrho} is a special case of this weighted inequality, corresponding to $\tau = \rho$ and $w \equiv 1$.  At this stage, it is also useful to note the relation between the Nikolskii constant and the Christoffel function of the subspace $P \subset L^2(\Omega,\tau)$.  Specifically, it is straightforward to show that
\be{
\label{Nikwequiv}
\Nikw = \sup_{\bm{y} \in \supp(\tau)} \sqrt{w(\bm{y}) \sum^{N}_{i=1} |\phi_i(\bm{y})|^{2}}.
}
{
Considering this expression, it becomes clear why $w(\bm{y})$ is taken in \R{woptimalchoice} as proportional to the Christoffel function, since this yields $\Nikw = \sqrt{N}$. The proof of Theorem \ref{Thm_optimalsampling} below relies on this observation.
}


\begin{theorem}\label{Thsam}
Let {$ \gamma,\delta \in (0,1)$} and $\mu_1,\ldots,\mu_M$ be probability measures satisfying \eqref{Measure} for some positive function $w$ defined almost everywhere on $\Omega$.  If 
\bes{
M \geq (\Nikw)^{2}((1-\delta )\log(1-\delta )+\delta )^{-1}\log(N/\gamma),
}
where $N=\dim(P)$, then with probability at least $1-\gamma$ the constant $\cC$ defined in \R{Cdef} satisfies $\cC \leq \frac{1}{\sqrt{1-\delta}}$.  Moreover, if $M$ satisfies the slightly stricter condition
\bes{
M \geq (\Nikw)^{2} ((1+\delta) \log(1+\delta)-\delta)^{-1} \log(2N/\gamma),
}
then, with probability at least $1-\gamma$, the constant $\cC$ satisfies $\cC \leq \frac{1}{\sqrt{1-\delta}}$ and the condition number of the matrix $\bm{A}$ defined by \R{Adef} satisfies $\kappa(\bm{A}) \leq \sqrt{\frac{1+\delta}{1-\delta}}$.
\end{theorem}

To prove this result, we first recall the \textit{Matrix Chernoff} inequality (see \cite[Thm.\ 1.1]{TroppUserFriendly}):

\begin{theorem}{(Matrix Chernoff)}
Consider a finite sequence $\brac{\bX_{k}}$ of independent, random, self-adjoint matrices with dimension $d$. Assume that each random matrix satisfies
\[ \bX_{k} \succeq 0 \quad\text{and}\quad \lma{\bX_{k}}\leq R \quad\text{almost surely.}\]
Define 
\[ \mu_{\min}:= \lmi{\sum_{k}\hop{\bX_{k}}} \quad\text{and}\quad 
\mu_{\max}:=\lma{\sum_{k}\hop{\bX_{k}}}. \]
Then
\eas{
 \mP{\lmi{\sum_{k}\bX_{k}}\leq (1-\delta)\mu_{\min}} &\leq d \left[ \frac{e^{-\delta}}{(1-\delta)^{1-\delta}}\right]^{\mu_{\min}/R} ,\qquad \forall \delta \in [0,1],
\\
\mP{\lma{\sum_{k}\bX_{k}} \geq (1+\delta)\mu_{\max}} & \leq d \left[ \frac{e^{\delta}}{(1+\delta)^{1+\delta}}\right]^{\mu_{\max}/R},\qquad \forall \delta \geq 0.
}
\end{theorem}

\begin{proof}[Proof of Theorem \ref{Thsam}]
Let $\brac{\phi_{1},...,\phi_{N}}$ be an orthonormal basis of $P$ with resect to $\tau$, $p \in P$, $p \neq 0$ be arbitrary and write $p=\siN c_i \phi_i$, so that 
\[ 
\nm{p}^{2}_{\Ltt}=\int_{\Omega}\left|\siN c_i\phi_i(\by)\right|^{2}d\tau(\by) = \siN|c_i|^{2} =\nm{\bc}_{2}^{2},\qquad \bc = (c_i)_{i=1}^N.
\]
Notice that $\nm{p}^2_{\Upsilon,w} = \frac1M \sum^{M}_{i=1} w(\bm{y}_i) |p(\bm{y}_i) |^2 = \bc^{*}\bG \bc$,
where $\bG  = \bm{A}^* \bm{A} \in \CNN$ is the self-adjoint matrix with entries $G_{j,k} = \ip{\phi_j}{\phi_k}_{\Upsilon,w}$.
It follow that
\bes{
\cC = \sup \left \{ \frac{\nm{\bm{c}}}{\sqrt{\bm{c}^* \bm{G} \bm{c}}} : \bm{c} \in \bbC^N,\ \bm{c} \neq \bm{0} \right \} = \frac{1}{\sqrt{\lambda_{\min}(\bm{G})}},
}
where $\lmi{\bG}$ is the minimal eigenvalue of $\bG$.
Write
\[ 
\bG = \sum^{M}_{i=1} \bX_{i}, \qquad \bX_{i}=\left \{ \frac{1}{M}w(\by_{i})\phi_j(\by_{i})\overline{\phi_k(\by_{i})} \right \}^{N}_{j,k = 1}.
\]
By construction, these matrices are independent and non-negative definite.  Also, 
\begin{eqnarray*}
\left ( \hop{\bX_{i}} \right )_{j,k}
 =  \int_{\Omega}\phi_j(\by)\overline{\phi_k(\by)}w(\by)\frac{1}{M}d\mu_i(\by),
\end{eqnarray*}
which gives
\[ 
\left ( \sum_{i=1}^{M}\hop{\bX_{i}} \right )_{j,k}= \int_{\Omega}\phi_j(\by)\overline{\phi_k(\by)}w(\by)\frac{1}{M}\sum_{i=1}^{M}d\mu_i(\by)=
\int_{\Omega}\phi_j(\by)\overline{\phi_k(\by)}d\tau (\by) = \delta_{j,k}.	
\]
Hence $\sum^{M}_{i=1} \bbE(\bm{X}_i) = \bm{I}$ is the identity matrix.
Moreover, for any $\bc \in \CN$ we have 
\[
\bc^{*}\bX_{i}\bc = \frac{1}{M}\left|\sum^{N}_{j=1} c_j\sqrt{w(\by_{i})}\phi_j(\by_{i})\right|^{2} 
\leq 
\frac{(\Nikw)^{2}}{M} \nm{ \siN c_j \phi_j }_{\Ltt}^{2}
=\frac{(\Nikw)^{2}}{M}\nm{\bc}_{2}^{2}.
\]
Since these matrices are self adjoint and nonnegative definite, we deduce that
\bes{
\lambda_{\max}(\bm{X}_i) \leq \frac{(\Nikw)^{2}}{M}.
}
We now apply the Matrix Chernoff bound with $d = N$, $R=(\Nikw)^{2}/{M}$ and 
\[ 
\mu_{\min}=\lmi{\smM\hop{\bX_{m}}}=\lmi{\bI}=1,
\]
to get
\eas{ 
\mP{\cC \geq \frac{1}{\sqrt{1-\delta}}} = \mP{\lmi{\smM \bX_{m}}\leq (1-\delta)} &\leq N \left[ \frac{e^{-\delta}}{(1-\delta)^{1-\delta}} \right]^{1/R},
\\
& = N \exp \left( 
-\frac{(1-\delta)\log(1-\delta)+\delta}{M^{-1}(\Nikw)^{2}}
\right).
}
The condition on $M$ implies that $\mP{\cC \geq \frac{1}{\sqrt{1-\delta}}} \leq \gamma$, which gives the first result.

For the second result, we note that $\kappa(\bm{A} ) = \sqrt{\lambda_{\max}(\bG)/\lambda_{\min}(\bG)}$.  Hence, by the Matrix Chernoff bound with $d$, $N$ and $\mu_{\min}$ as above and $\mu_{\max} = 1$, we have
\eas{
\bbP \left ( \kappa(\bm{A}) \geq \sqrt{\frac{1+\delta}{1-\delta}} \right ) &\leq \bbP \left ( \lambda_{\min}(\bm{G}) \leq (1-\delta) \right )+ \bbP \left ( \lambda_{\max}(\bm{G}) \geq (1+\delta) \right )
\\
\leq &N \left ( \exp \left( 
-\frac{(1-\delta)\log(1-\delta)+\delta}{M^{-1}(\Nikw)^{2}}
\right) + \exp \left( 
-\frac{(1+\delta)\log(1+\delta)-\delta}{M^{-1}(\Nikw)^{2}}
\right) \right ).
}
Note that $(1+\delta) \log(1+\delta) - \delta \leq (1-\delta ) \log(1-\delta) + \delta$ for $0 < \delta < 1$.  Hence
\bes{
\bbP \left ( \kappa(\bm{A}) \geq  \sqrt{\frac{1+\delta}{1-\delta}} \right )  \leq 2 N \exp \left( -\frac{(1+\delta)\log(1+\delta)-\delta}{M^{-1}(\Nikw)^{2}}\right) \leq  \gamma,
}
where in the last step we use the condition on $M$.  This completes the proof.
\end{proof}

\begin{proof}[Proof of Theorem \ref{Thm_optimalsampling}]
The result follows from \R{Nikwequiv} the definition of $w$ \R{woptimalchoice}.  Indeed, we have $\cN(P,\tau,w) = \sqrt{N}$ for this choice of $w$.  Hence Theorem \ref{Thsam} gives the result.
\end{proof}

We conclude this section with the proofs of Theorems \ref{t:Method1} and \ref{t:Method2}:

\prf{[Proof of Theorem \ref{t:Method1}]
Theorem \ref{Thm_optimalsampling} and the condition on $M$ imply that $\cC \leq 1/\sqrt{1-\delta}$ and $\kappa(\bm{A}) \leq \sqrt{1+\delta} / \sqrt{1-\delta}$ with probability at least $1-\gamma / 2$, and Proposition \ref{p:Ksize} and the condition on $K$ imply that $\cD \leq 1/\sqrt{1-\delta}$ with probability at least $1-\gamma / 2$.  Hence $\cC \leq 1/\sqrt{1-\delta}$, $\cD \leq 1/\sqrt{1-\delta}$ and $\kappa(\bm{A}) \leq \sqrt{1+\delta} / \sqrt{1-\delta}$ with probability at least $1-\gamma$.  The condition on $\cD$ implies that $\bm{B}$ is full rank (see \S \ref{ss:discorth}).  Next, observe that \R{wMethods12} and \R{piiMeth1} give
\bes{
\nm{g}_{\tau,w} = \sup_{\bm{y} \in \supp(\tau)} \sqrt{w(\bm{y})} | g(\bm{y}) | = \max_{i=1,\ldots,K} \sqrt{w(\bm{z}_i)} | g(\bm{z}_i) | = \max_{i=1,\ldots,K} \left \{ \frac{|g(\bm{z}_i) |}{\sqrt{K \pi_i}} \right \} = \tnm{g}_{Z,\pi}.
}
The result now follows from Theorem \ref{Thm_errorbound}.
}

\prf{[Proof of Theorem \ref{t:Method2}]
For each $t$, let $\cC = \cC_t$ and $\cD = \cD_t$ be the corresponding constants and write $\bm{A} = \bm{A}_t$, $\bm{B} = \bm{B}_t$ for the matrices defined in Step 7 and Step 4 of Method 2 respectively.  Define the following events:
\eas{
E &: \cD_r \leq \frac{1}{\sqrt{1-\delta}},
\\
F_t &: \cC_t \leq \frac{1}{\sqrt{1-\delta}}\ \mbox{and}\ \kappa(\bm{A}_t) \leq \sqrt{\frac{1+\delta}{1-\delta}},\quad t = 1,\ldots,r,
\\
G &= E \cap F_1 \cap \ldots \cap F_r.
}
Suppose first that event $G$ occurs.  Notice that $\cD_1 \leq \cD_2 \leq \ldots \leq \cD_r \leq 1/\sqrt{1-\delta}$ since the $P_t$ are nested subspaces.  Hence $\bm{B}_t$ is full rank for every $t$, which gives (i).  Also, as in the proof of the previous theorem, the events $F_t$ imply (ii) and the events $E$ and $F_t$ imply (iii).

It remains to show $\bbP(G) \geq 1 - \gamma$.  By the union bound $\bbP(G^c) \leq \bbP(E^c) + \bbP(F^c_1) + \ldots + \bbP(F^c_r)$.
Proposition \ref{p:Ksize} and the condition on $K$ give that $\bbP(E^{c}) \leq \gamma/2$.  Moreover, since $M_t = k_t N_t$, Theorem \ref{Thm_optimalsampling} and the condition on $k_t$ give that $\bbP(F^c_t) \leq \gamma_t/2$.  Hence $\bbP(G^c) \leq \gamma/2 + \sum^{r}_{t=1} \gamma_t / 2 = \gamma$, as required.
}

\section{Numerical examples}\label{s:numexamp}
To conclude this paper, we demonstrate Methods 1 and 2 on several examples.  Throughout, we consider the approximation of smooth functions using polynomials.  In particular, we choose
\bes{
P = P^{\mathrm{HC}}_n = \spn \left\{ \bm{y} \mapsto \bm{y}^{\bm{n}} : \bm{n} \in \Lambda^{\mathrm{HC}}_n \right \},
}
where $\Lambda^{\mathrm{HC}}_n $ is the \textit{hyperbolic cross} index set of index $n$:
\bes{
\Lambda^{\mathrm{HC}}_n = \left \{ \bm{n} = (n_1,\ldots,n_d) \in \bbN^d_0 : \prod^{d}_{k=1} {(n_k+1)}\leq n+1 \right \}.
}
We take $\Omega$ to be a compact domain contained in the unit hypercube $[-1,1]^d$, and $\rho$ to be the uniform measure on $\Omega$.  Sampling from $\rho$ is performed by rejection sampling.  The initial basis $\{ \psi_1,\ldots,\psi_N \}$ for $P$ is constructed by taking the restrictions to $\Omega$ of the orthonormal Legendre polynomials on $[-1,1]^d$ with indices belonging to $\Lambda$.  This approach is based on \cite{adcock2018approximating}. {Note, however, that \cite{adcock2018approximating} does not seek to orthogonalize this basis, unlike the approach developed in this paper.}

In our experiments we first generate a grid of size $K = 20000$ points, and then compute the approximation for values $1 \leq N_1 < \ldots < N_r \leq 1000$.  {In the first series of experiments, we compute the relative approximation error
\be{
\label{Kgriderr}
E_{\tau}(f) = \frac{\nmu{f - \tilde{f}}_{L^2(\Omega,\tau)}}{\nmu{f}_{L^2(\Omega,\tau)}},
}
over this grid,} as well as the constant $\cC$. {Later, in Fig.\ \ref{f:Tgriderr}, we consider an error computed over a distinct grid of points.}  The {computed values of $E_{\tau}(f)$ and $\cC$} values are averaged over {$50$ trials}, as follows. For Method 1, for each $N$ we take {$50$} independent draws of the corresponding $M$ sample points $\{ \bm{y}_i \}^{M}_{i=1}$ and take the {mean} values of $\cC$ and the error.  For Method 2, we perform {$50$} independent experiments sweeping, as described in the method, over $N_1,\ldots,N_r$ and then take the {mean} values.

{To examine the benefit of the new sampling procedure,} we also consider uniform random sampling over the fine grid of $K$ points {(this type of sampling is similar to that used \cite{adcock2018approximating}, although not identical, since \cite{adcock2018approximating} considers random sampling with respect to the continuous uniform measure on $\Omega$)}.  This methods is referred to as `Uniform' in our experiments.  For functions, we use the following:
{
\eas{
f_1(\bm{y}) &= \exp \left ( - \sum^{d}_{i=1} y_i / d\right ),\qquad f_2(\bm{y}) = \frac{1}{\sum^{d}_{i=1} \sqrt{|y_i|}},
\\
f_3(\bm{y}) &= \prod^{d}_{i=1} \frac{d/4}{d/4 + (y_i + (-1)^{i+1} / (i+1) )^2},\qquad f_4(\bm{y}) = \frac{1}{y^2_1+y^2_2}.
}
}
Note that $f_3$ is known as the `Genz product peak' function.

\begin{figure}
\begin{center}
{\small
\begin{tabular}{ccc}
 \includegraphics[scale=0.38]{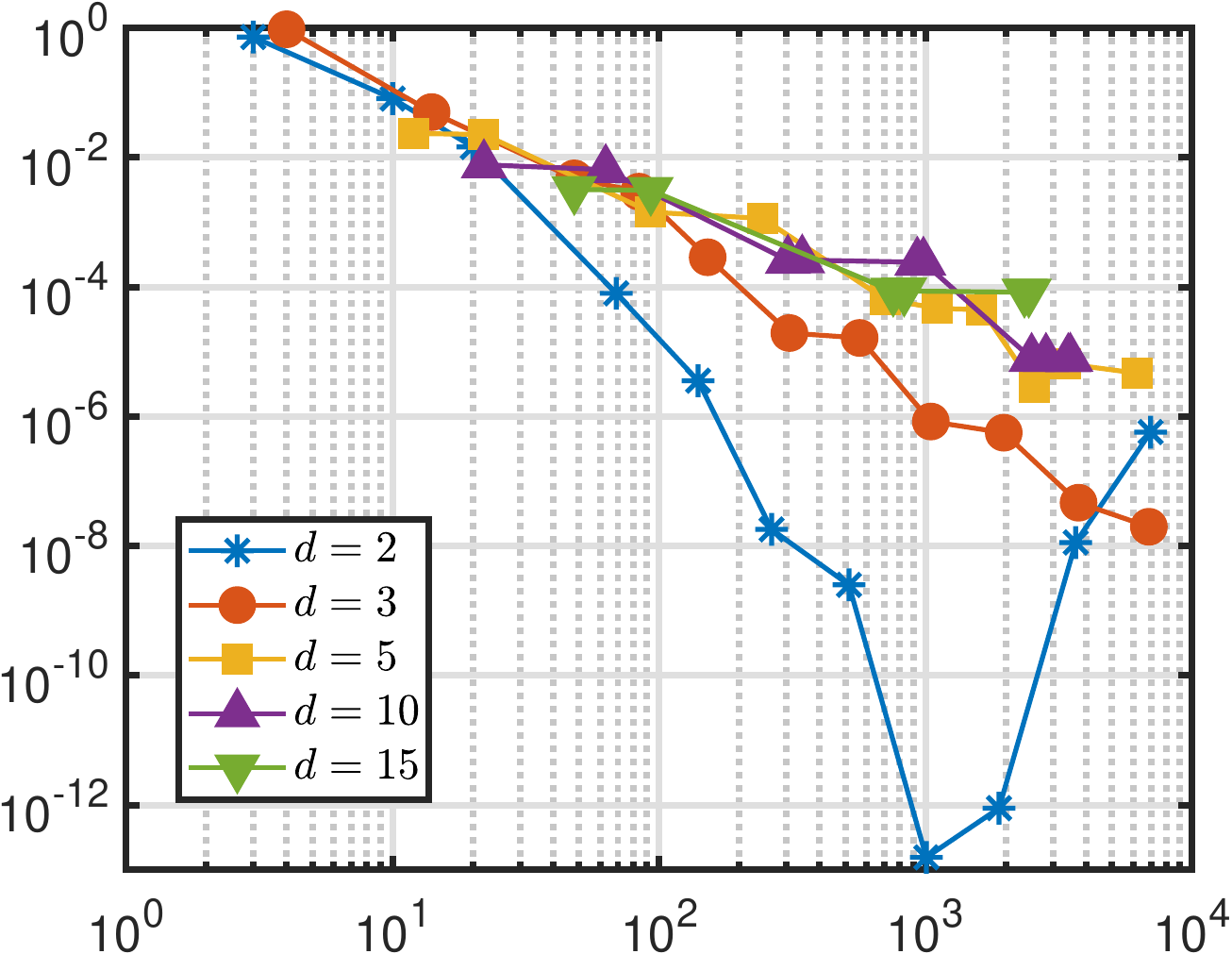} &
\includegraphics[scale=0.38]{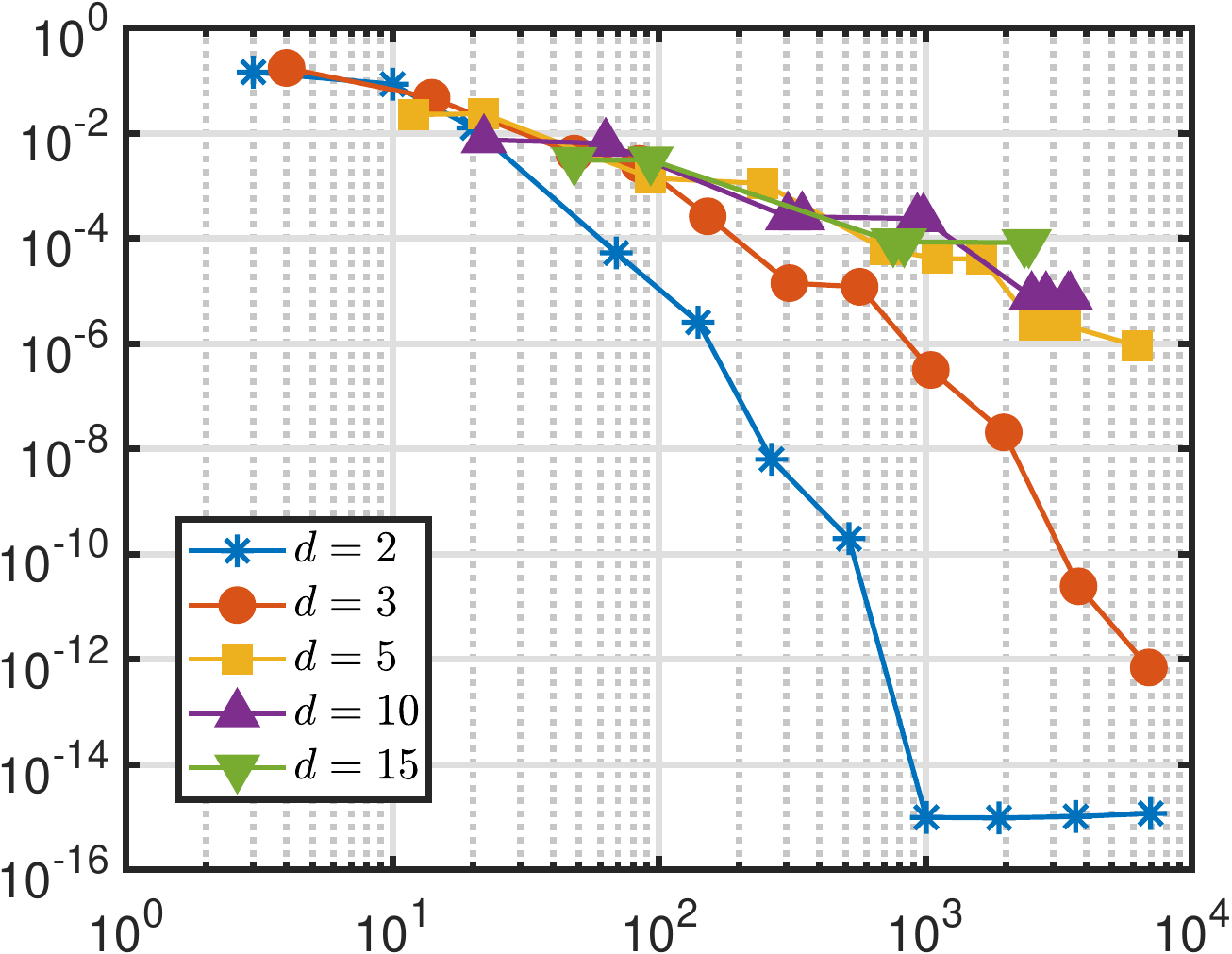} & 
\includegraphics[scale=0.38]{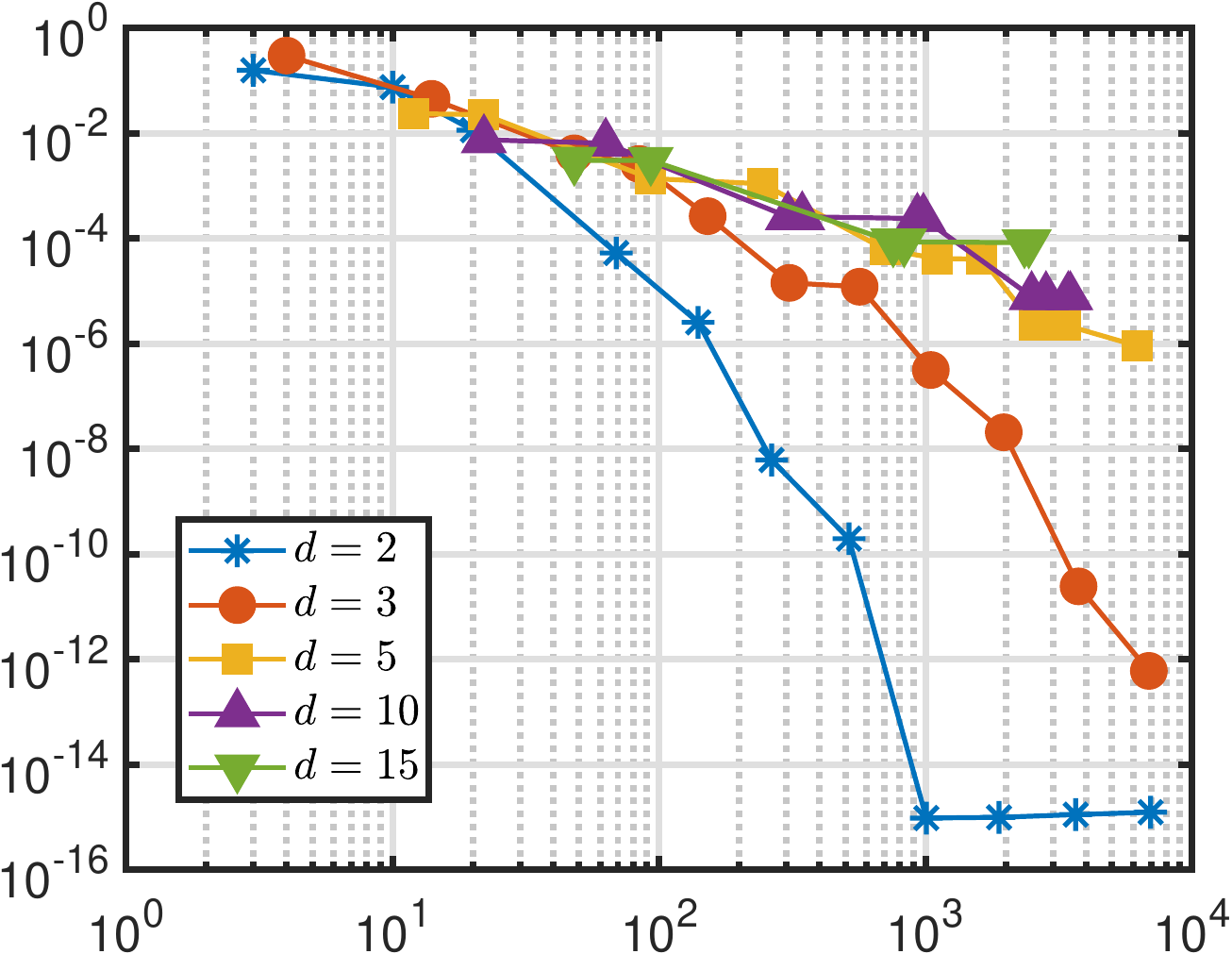}
\\  
 \includegraphics[scale=0.38]{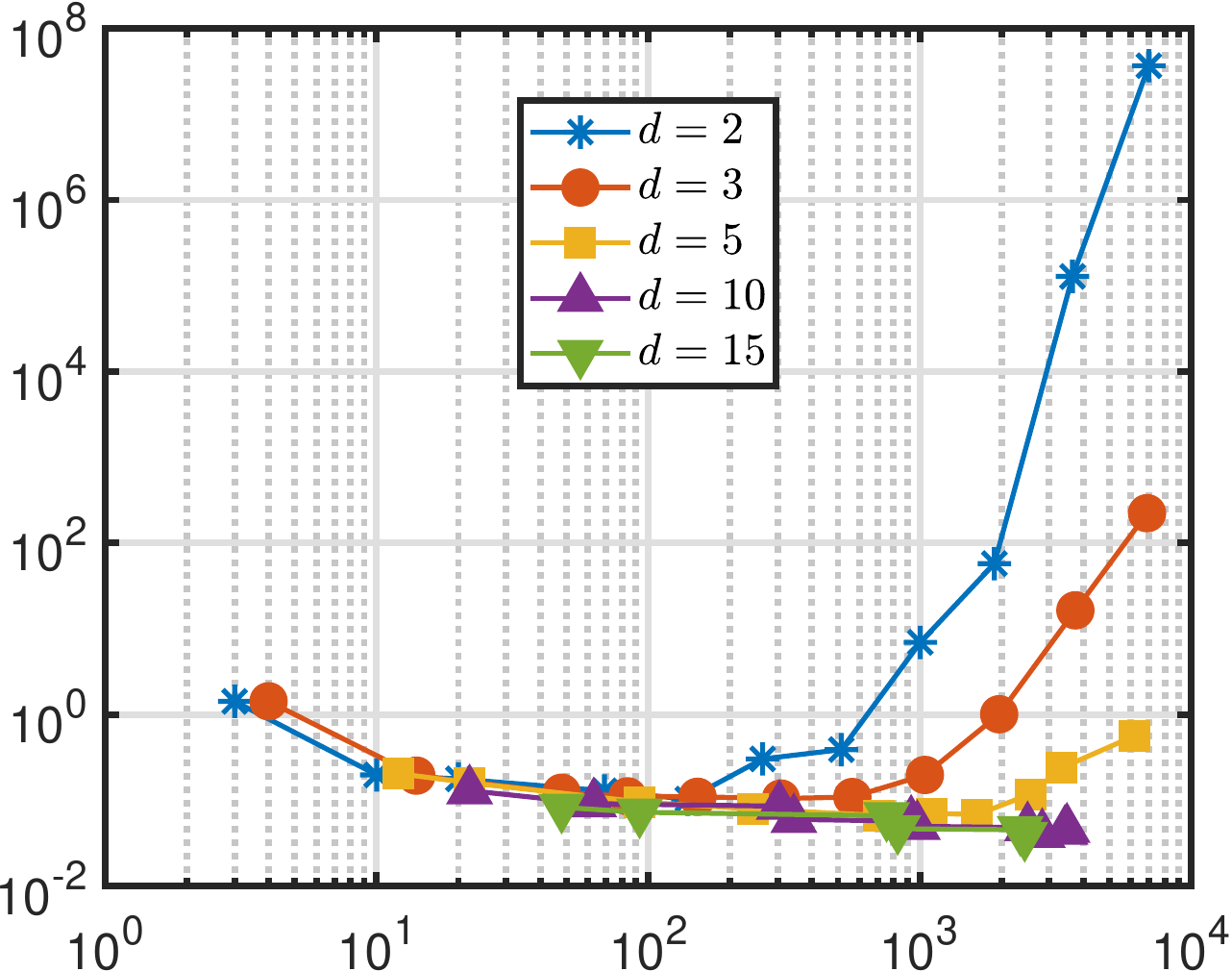} &
 \includegraphics[scale=0.38]{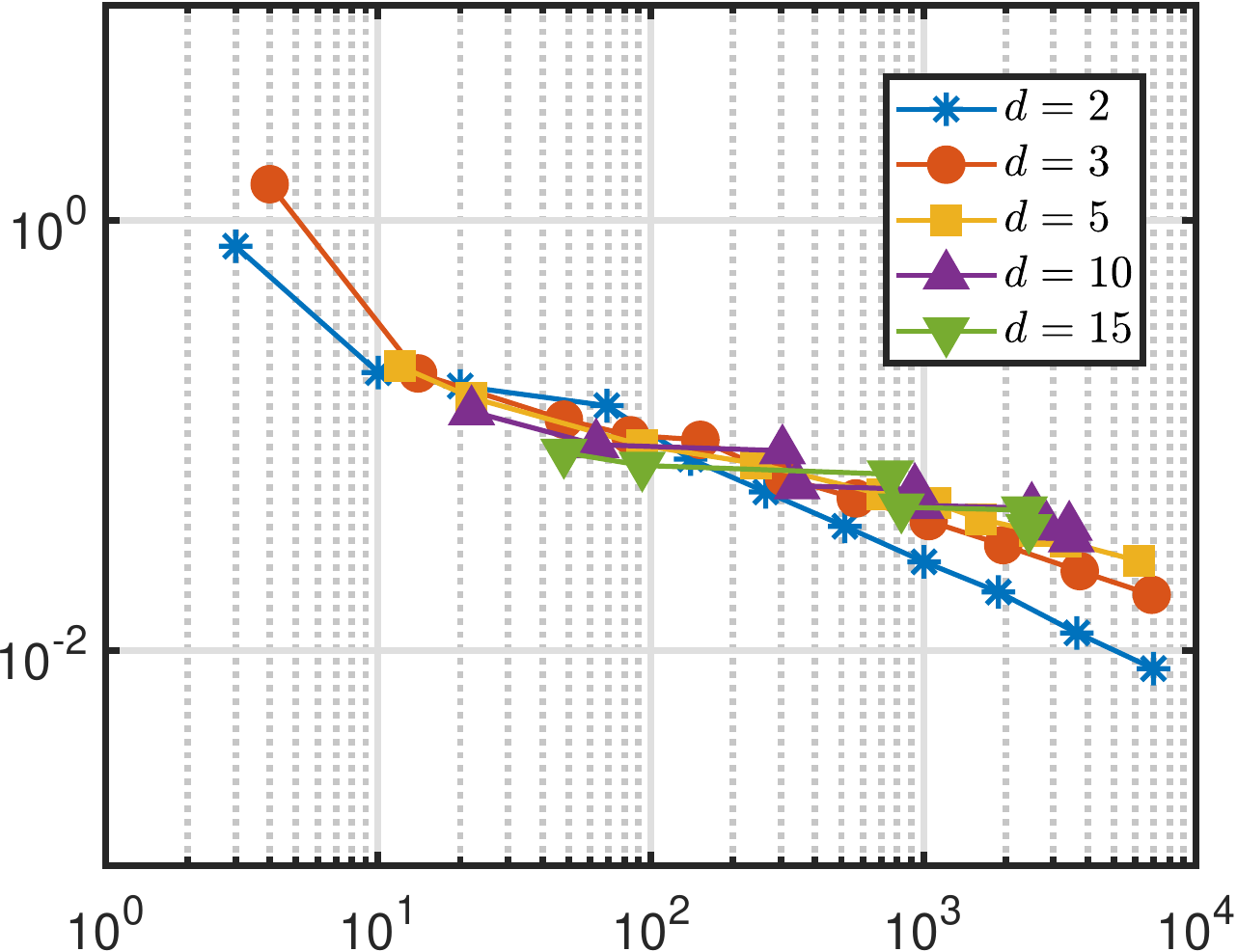} &
 \includegraphics[scale=0.38]{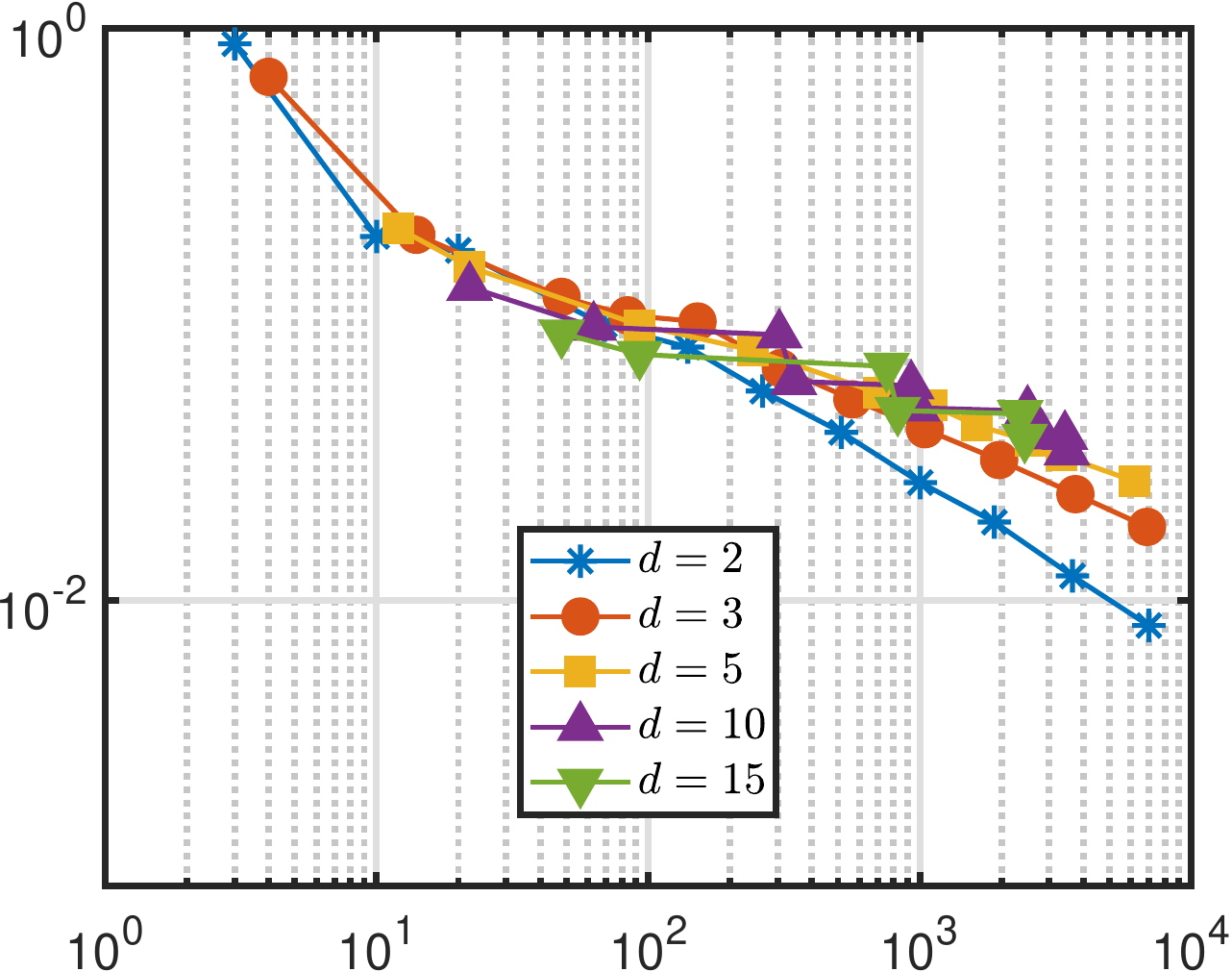}
\\
 \includegraphics[scale=0.38]{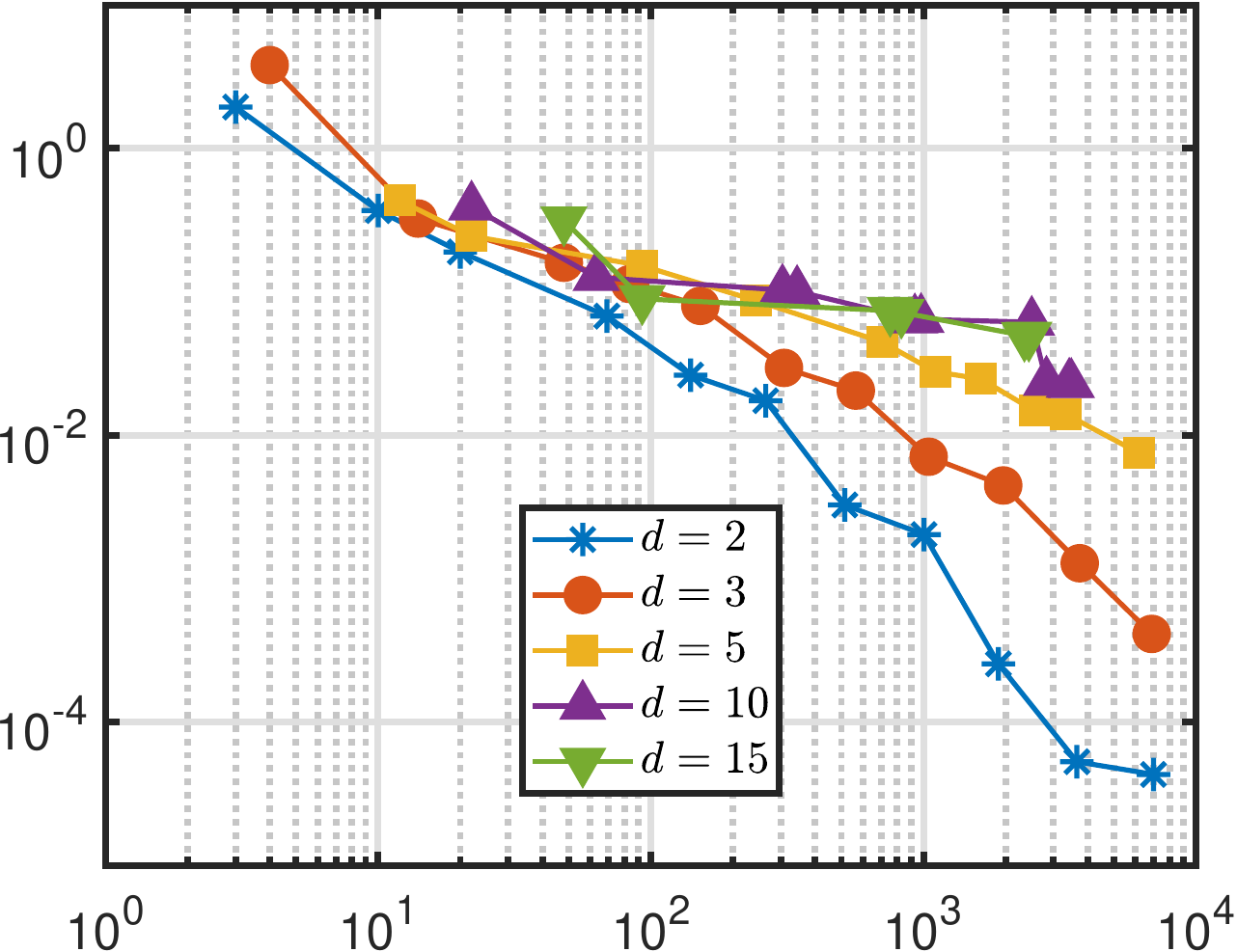} &
  \includegraphics[scale=0.38]{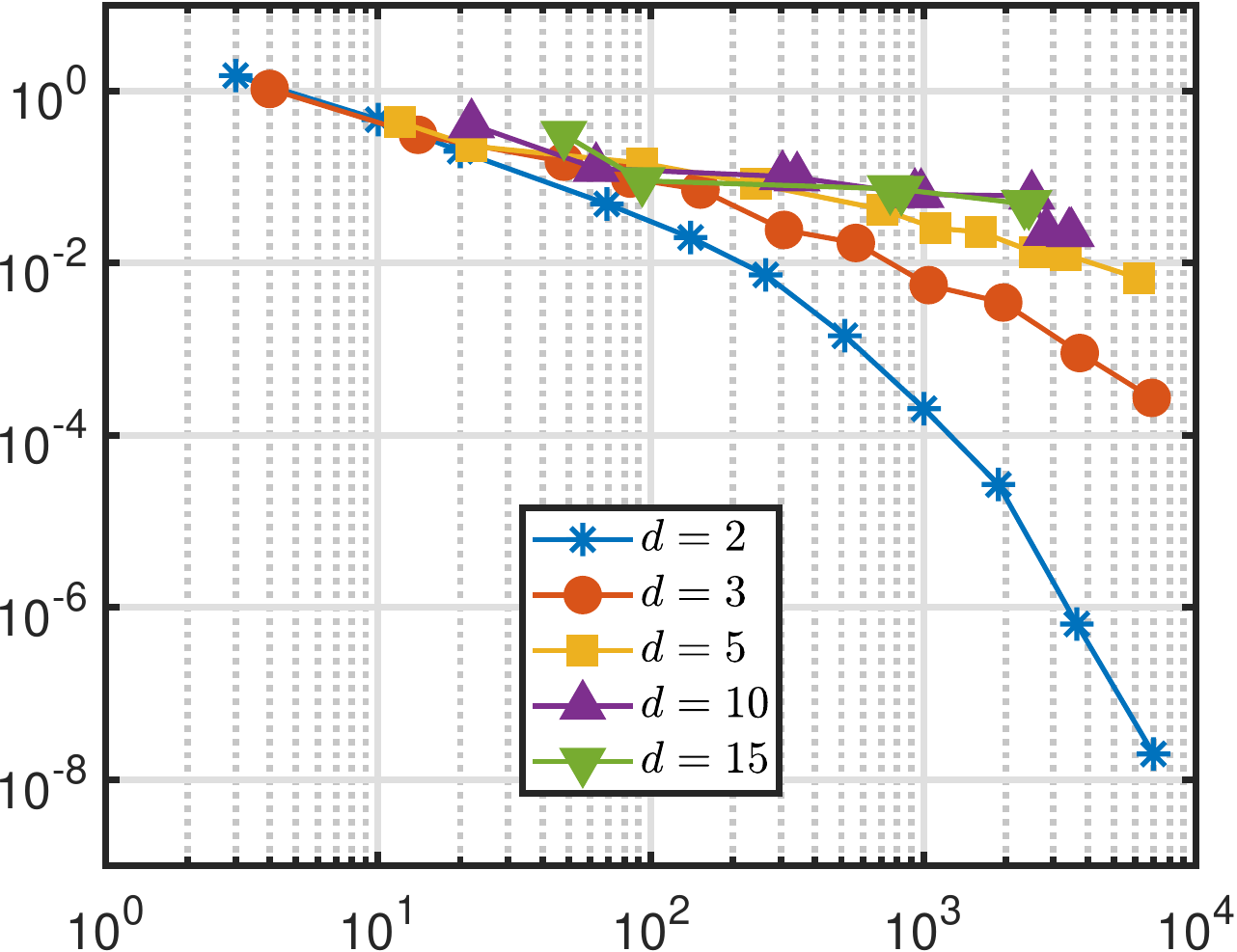} &
 \includegraphics[scale=0.38]{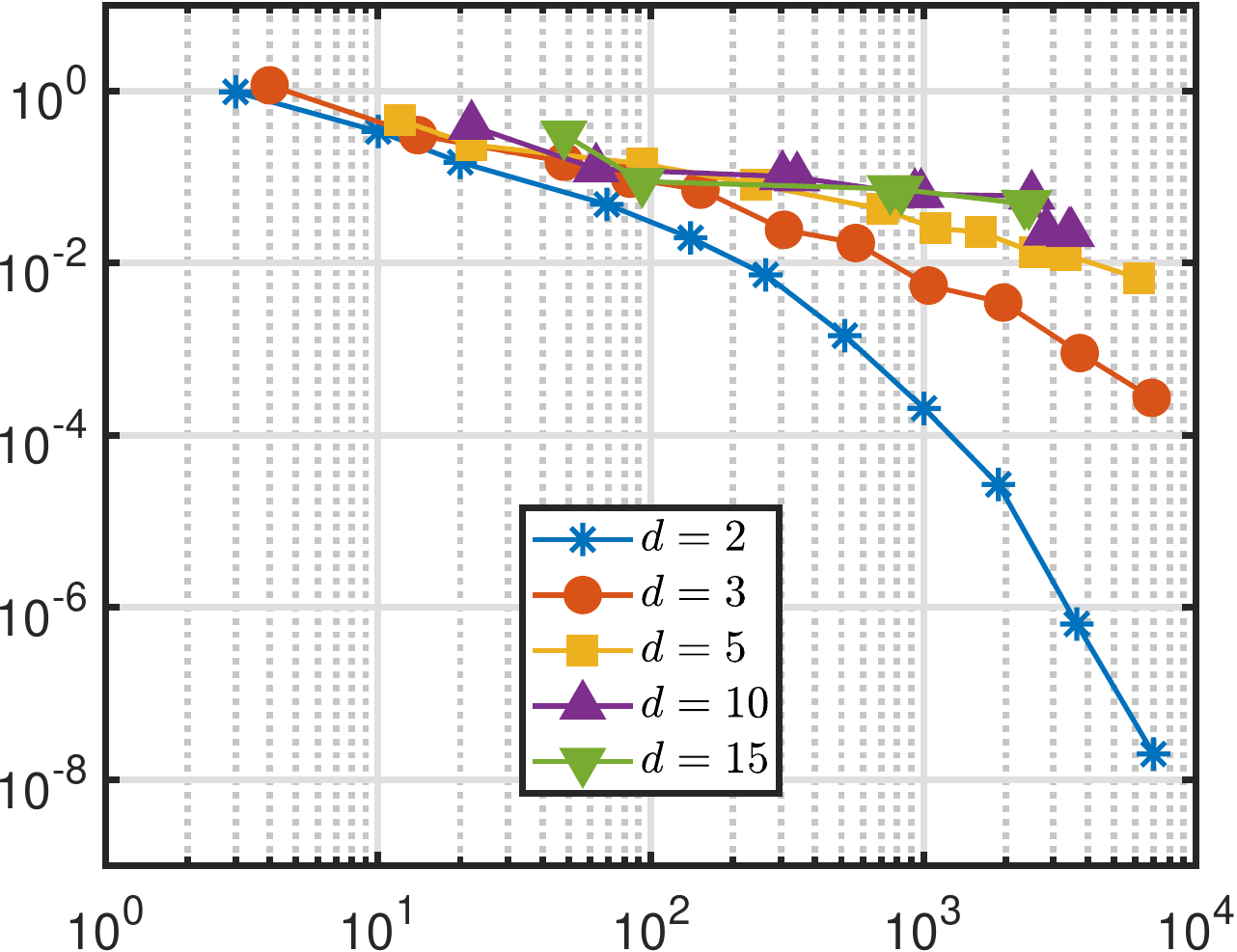}
\\  
 \includegraphics[scale=0.38]{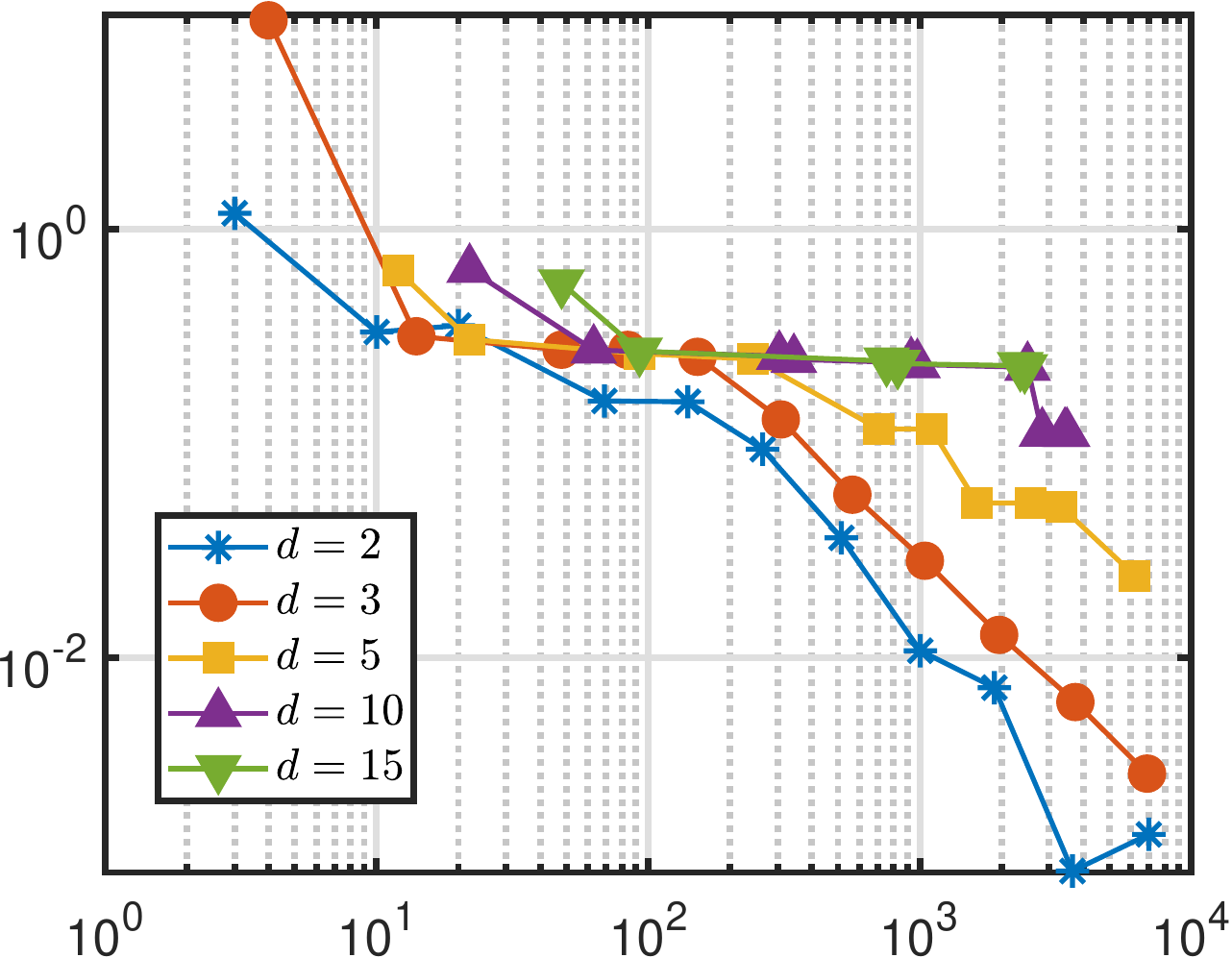}&
 \includegraphics[scale=0.38]{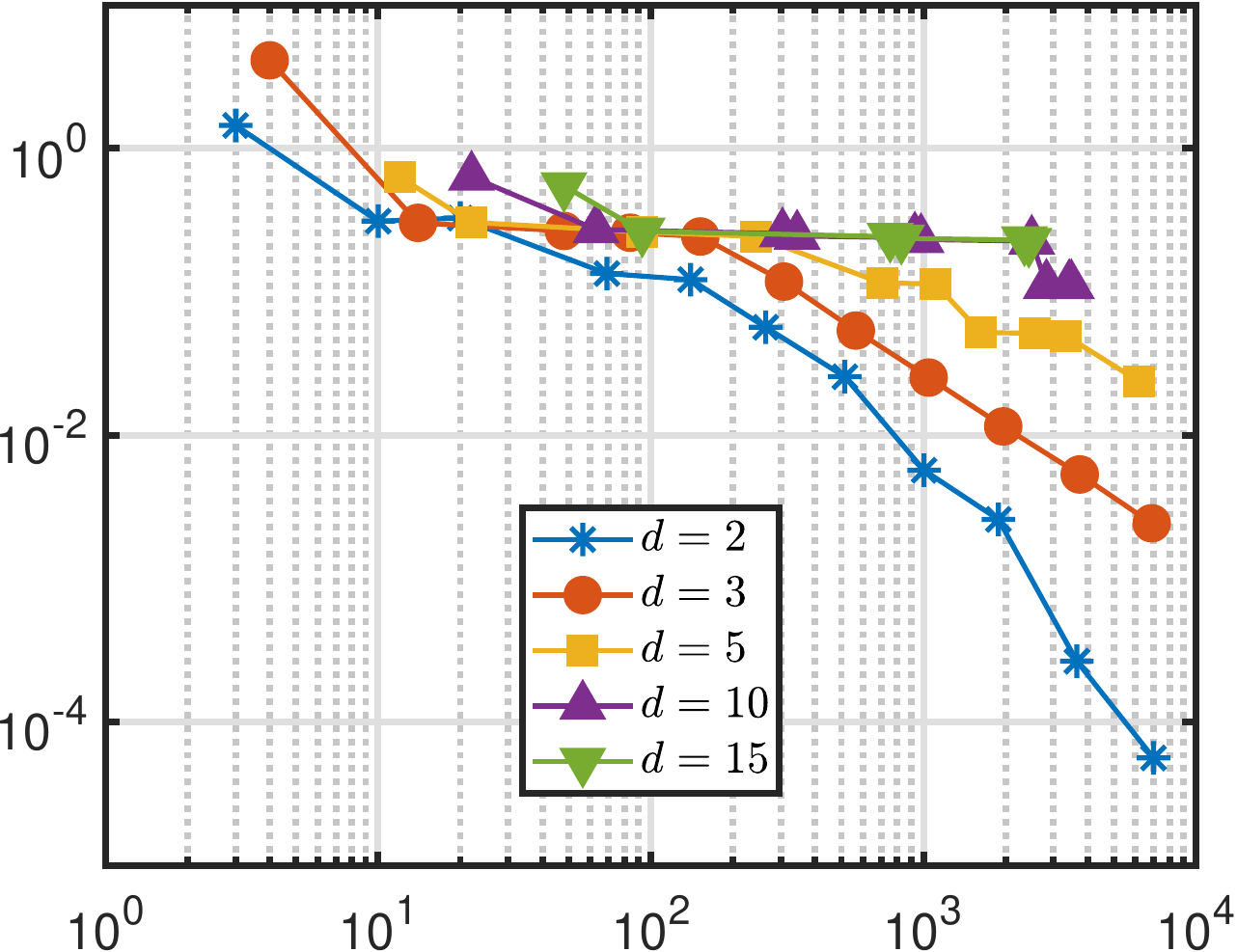}&
 \includegraphics[scale=0.38]{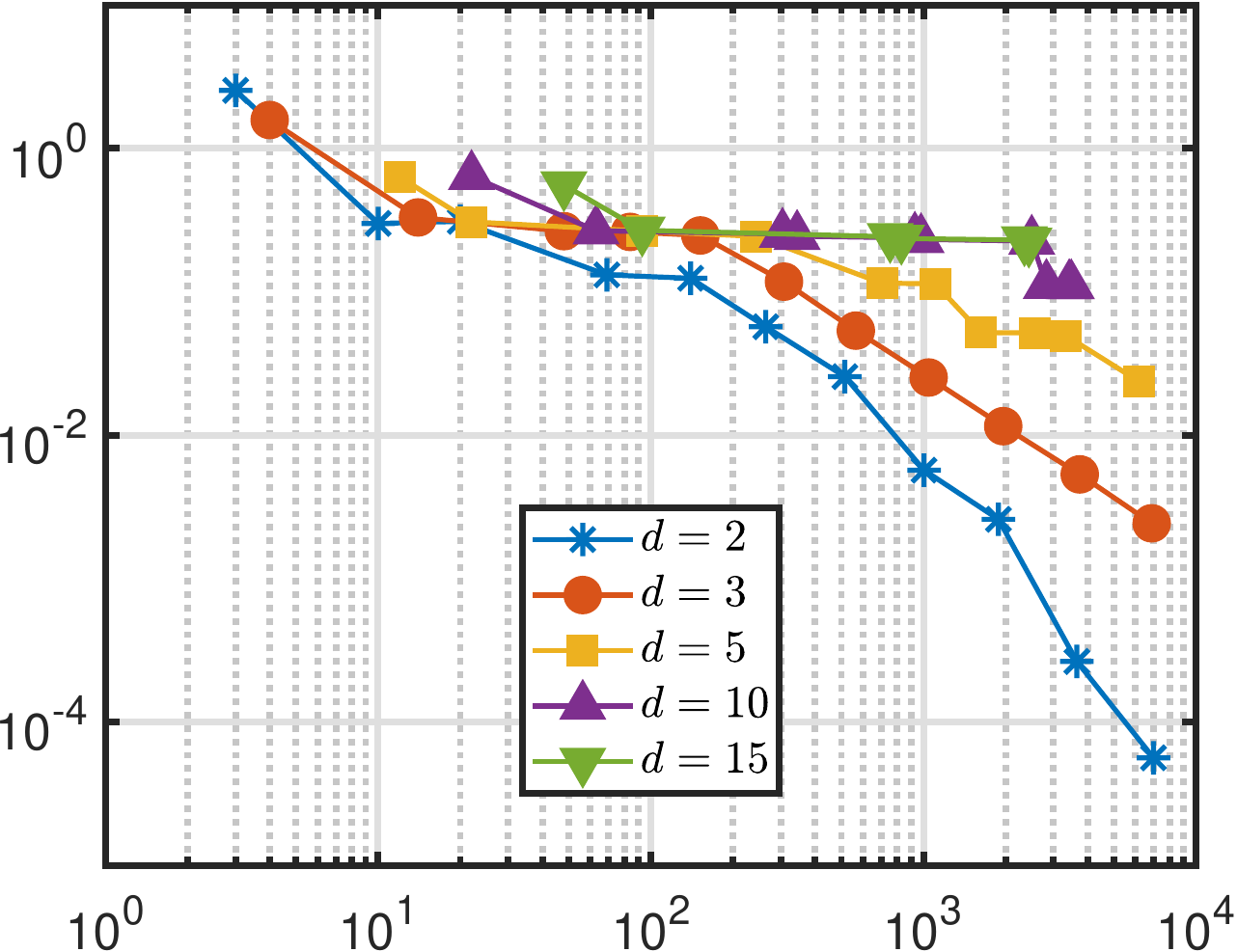}
\\  
Uniform & Method 1 & Method 2  
  \end{tabular}
  }
\end{center}
\caption{
{The error $E_{\tau}(f)$ versus $M$, with $M$ chosen as the smallest value such that $M \geq N \log(N)$.  First row: the domain $\Omega_1 = \left \{ \bm{y} : 1/4 \leq \nm{\bm{y}}_2 \leq 1 \right \}$ with $f = f_1$.  Second row: $\Omega_1$ and $f = f_2$. Third row: the domain $\Omega_2 = \left \{ \bm{y} \in (-1,1)^d : y_1 + \ldots + y_d \leq 1 \right \}$ with $f = f_3$.  Fourth row: the domain $\Omega_3 = \left \{ \bm{y} \in (-1,1)^d :   y^2_1+y^2_2\geq 1/4 \right \}$ with $f = f_4$}.  The domains are shown in Fig.\ \ref{f:Domains} for $d = 2$.  
}
\label{f:Message1}
\end{figure}

In Fig.\ \ref{f:Message1} we compare Method 1 and Method 2 with uniform random sampling (Uniform) over three domains in various different dimensions.  The domains (for $d = 2$) are shown in Fig.\ \ref{f:Domains}, along with the fine grid of $K$ points and the samples generated by Method 1 for a typical value of $M$.  In all cases, both procedures lead to an improvement over uniform sampling.  {The effect is most noticeable for lower dimensions; such an observation is not surprising, given that in higher dimensions the maximum degree $n$ of the hyperbolic cross index set $\Lambda^{\mathrm{HC}}_n$ is not as large as in lower dimensions (in our experiments, $N = |\Lambda^{\mathrm{HC}}_n| \leq 1000$}).  For Uniform the error actually increases with $N$ in some cases. This effect is most dramatic for $f_2$ which, unlike the other functions, is not infinitely smooth in $\Omega$.  The reason for this increase is because the number of samples $M \asymp N \log(N)$ in this experiment, which is asymptotically lower order than the quadratic scaling $N^2 \log(N)$ known to be sufficient (for certain domains, see \cite{adcock2018approximating}) when working with uniform random samples.

This experiment also demonstrates that Method 1 and Method 2 have similar performance in all cases.  Recall, however, that Method 1 recycles none of its samples when $N$ increases, whereas Method 2 recycles all its samples.

\begin{figure}
\begin{center}
{\small
\begin{tabular}{ccc}
\includegraphics[scale=0.29]{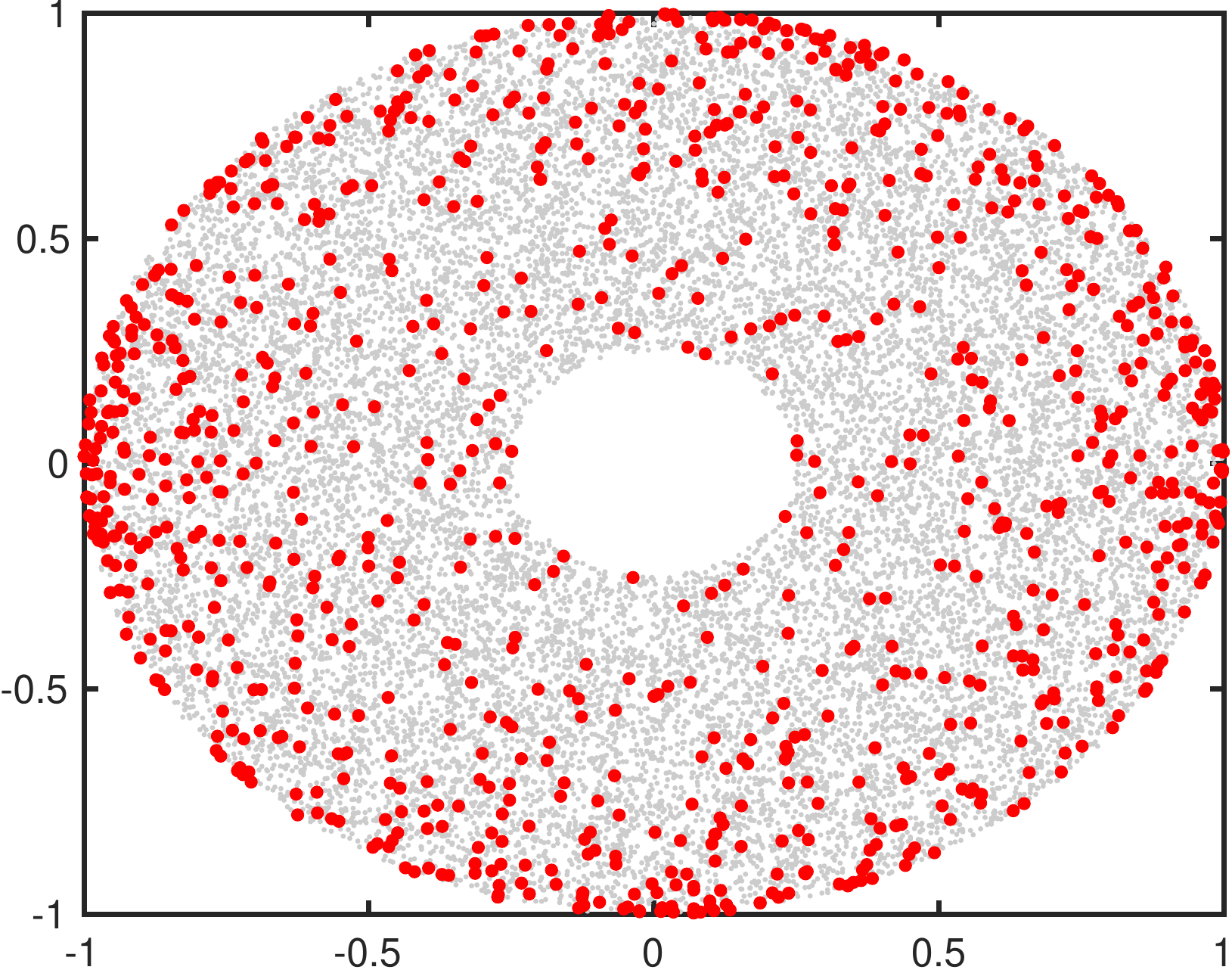} & \hspace{-0.3cm}
\includegraphics[scale=0.29]{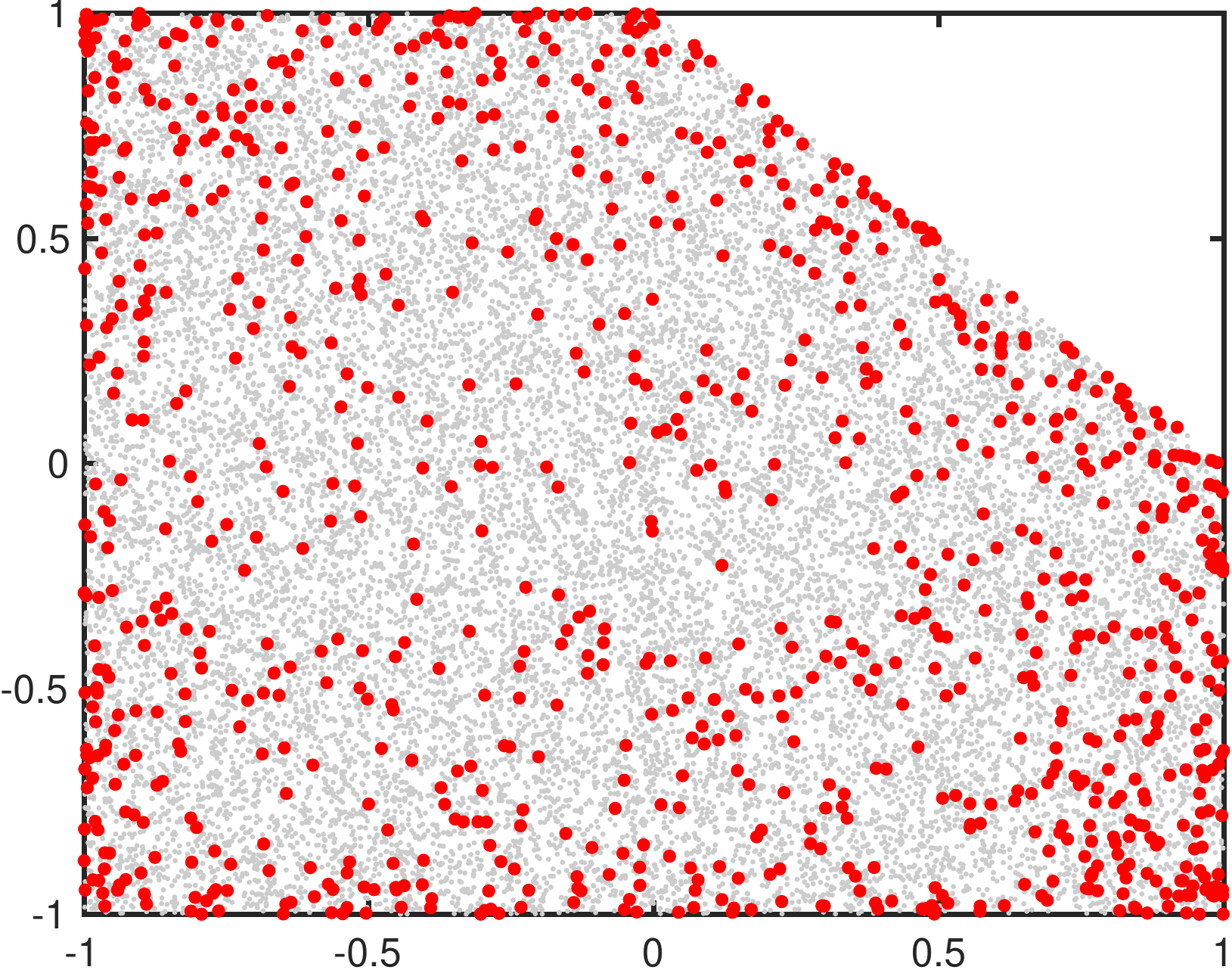}& \hspace{-0.3cm}
\includegraphics[scale=0.29]{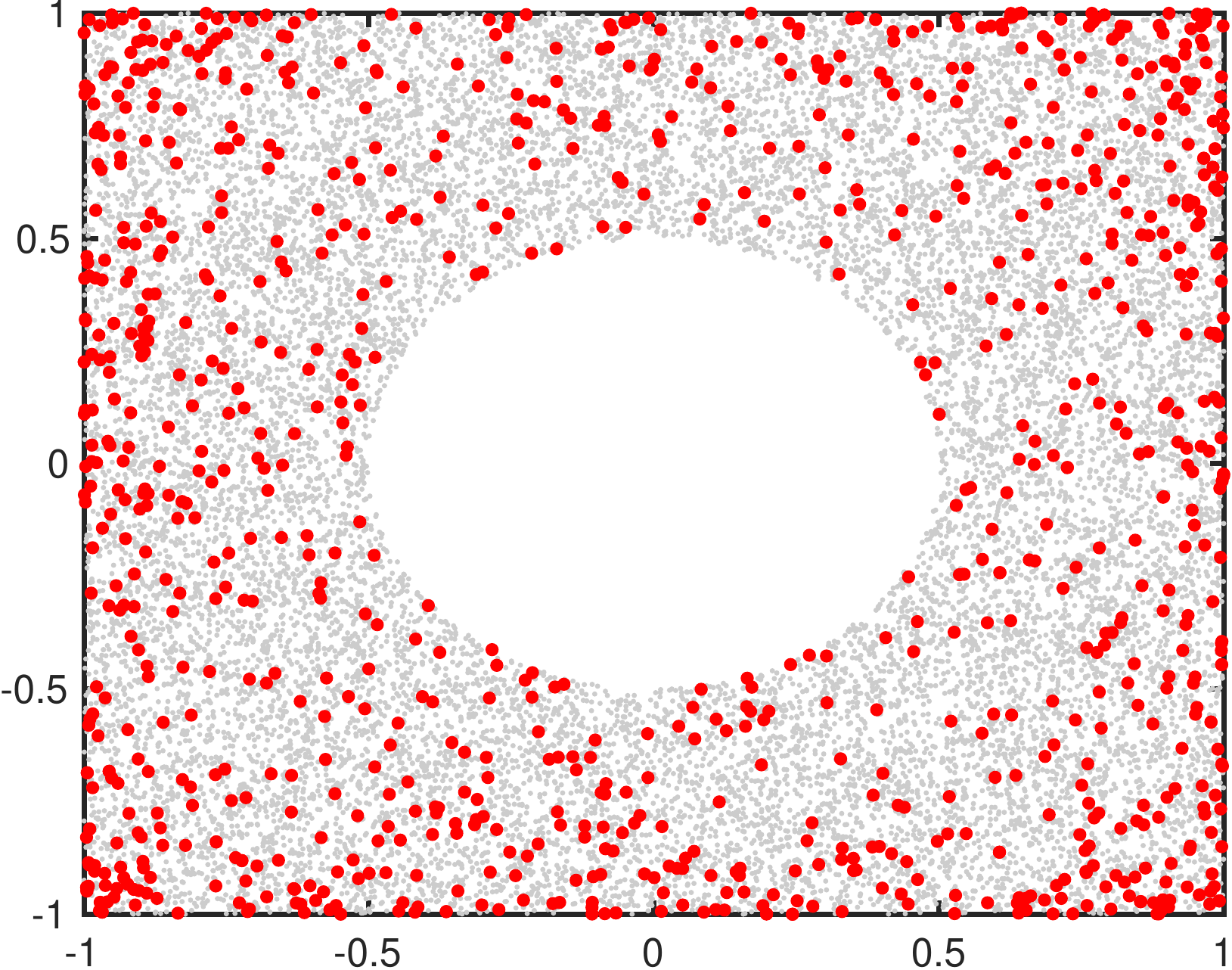}
  \end{tabular}
  }
\end{center}
\caption{
The domains $\Omega_1$, $\Omega_2$ and $\Omega_3$ (left to right) used in Fig.\ \ref{f:Message1} for $d = 2$, the fine grid with $K = 20000$ points, and the samples generated in a typical instance of Method 1 with $M = 1000$ and $N = 198$.
}
\label{f:Domains}
\end{figure}

In Fig.\ \ref{f:Message3} we examine the constant $\cC$ for the three domains and methods, and for different scalings of $M$ with $N$.  For Uniform, all scalings lead to an exponentially increasing constant $\cC$ -- a well-known phenomenon \cite{AdcockNecSamp}.  For Method 1 and Method 2, notice that any linear scaling $M = c N$ eventually leads to a growing constant $\cC$, whereas $\cC$ remains bounded for either of the two log-linear scalings $M = c N \log(N)$.  This result therefore verifies Theorem \ref{Thm_optimalsampling}.

\begin{figure}
\begin{center}
{\small
\begin{tabular}{ccc}
\includegraphics[scale=0.38]{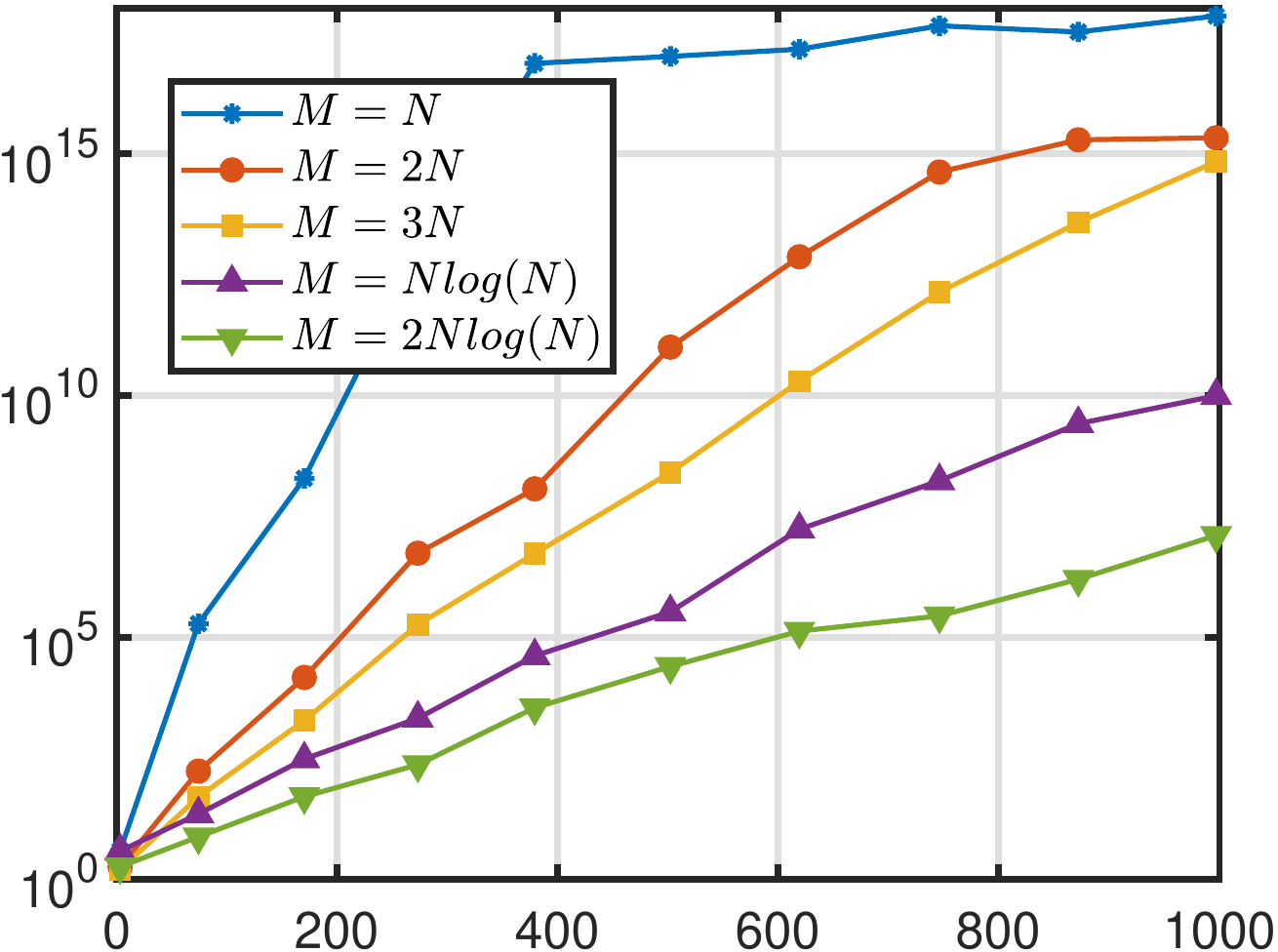} &
\includegraphics[scale=0.38]{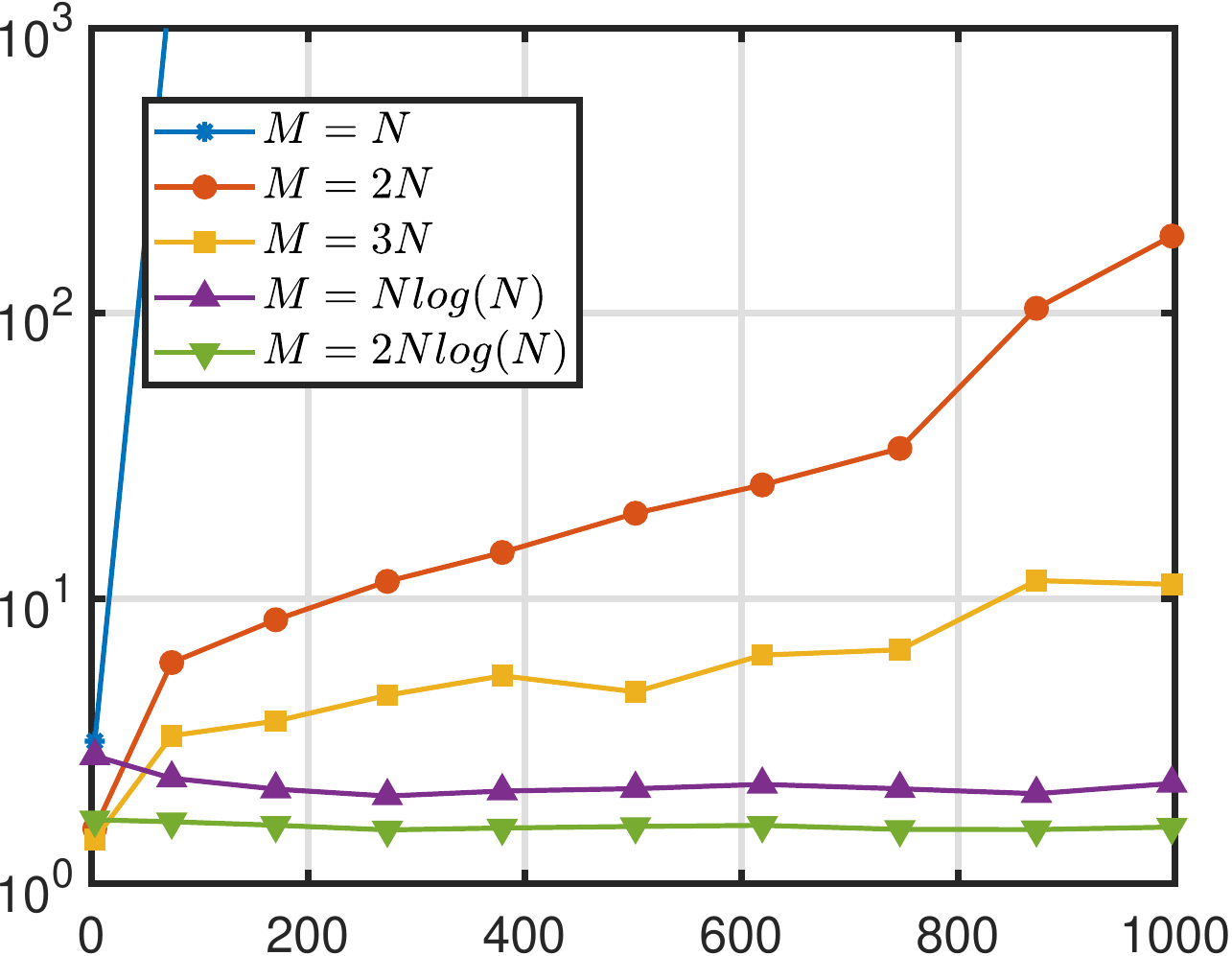} &
\includegraphics[scale=0.38]{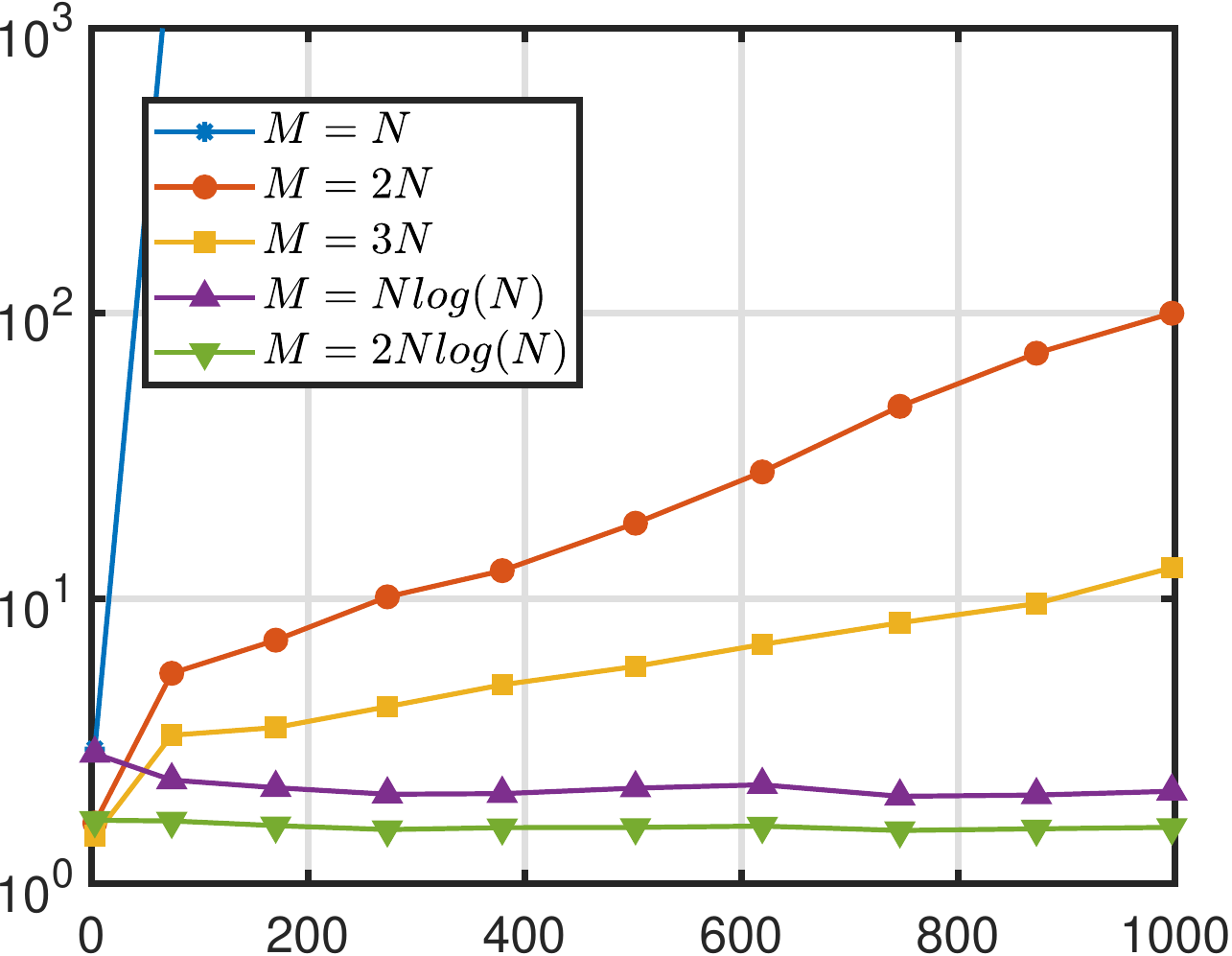} 
\\
\includegraphics[scale=0.38]{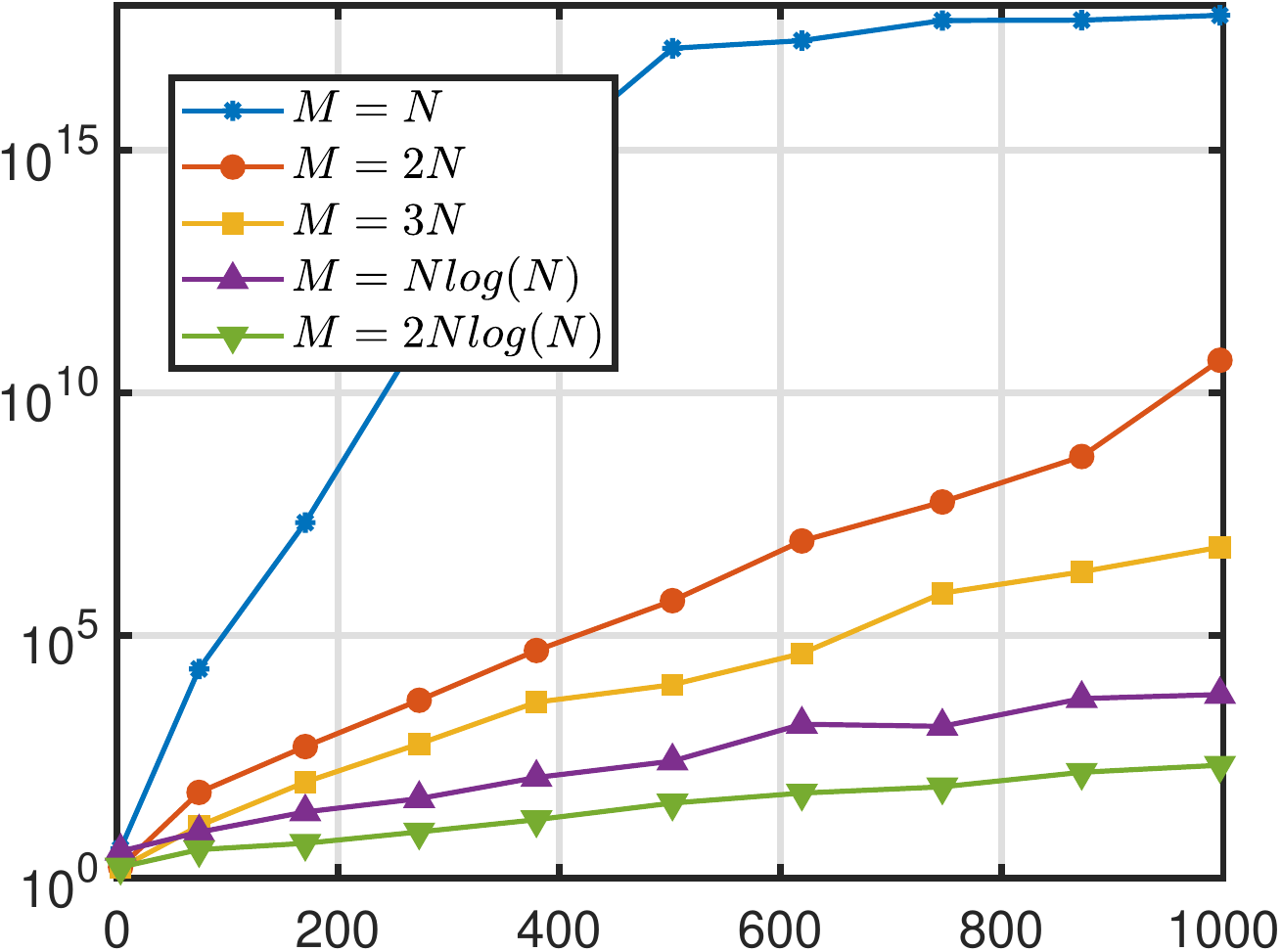} &
\includegraphics[scale=0.38]{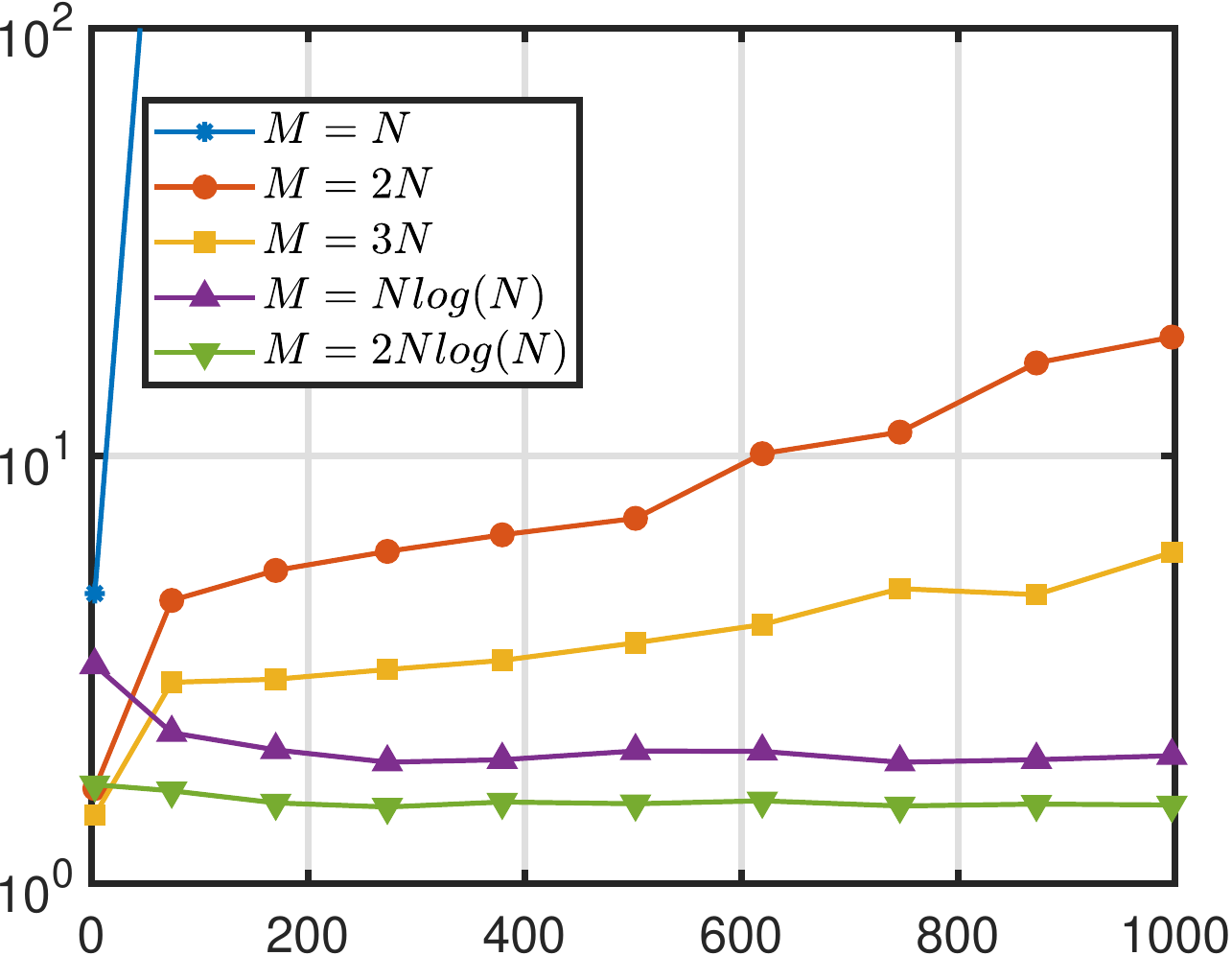} &
\includegraphics[scale=0.38]{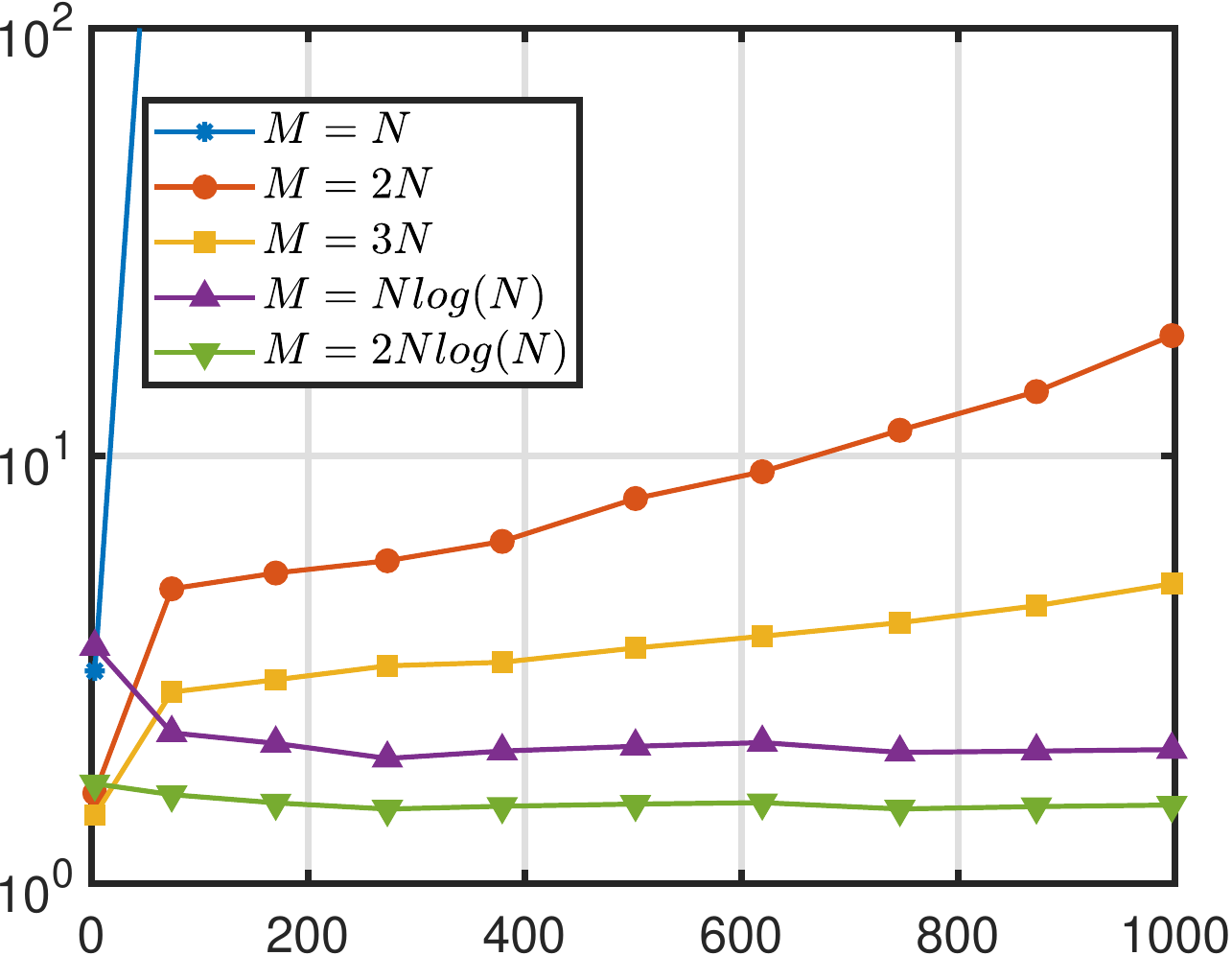}
\\
\includegraphics[scale=0.38]{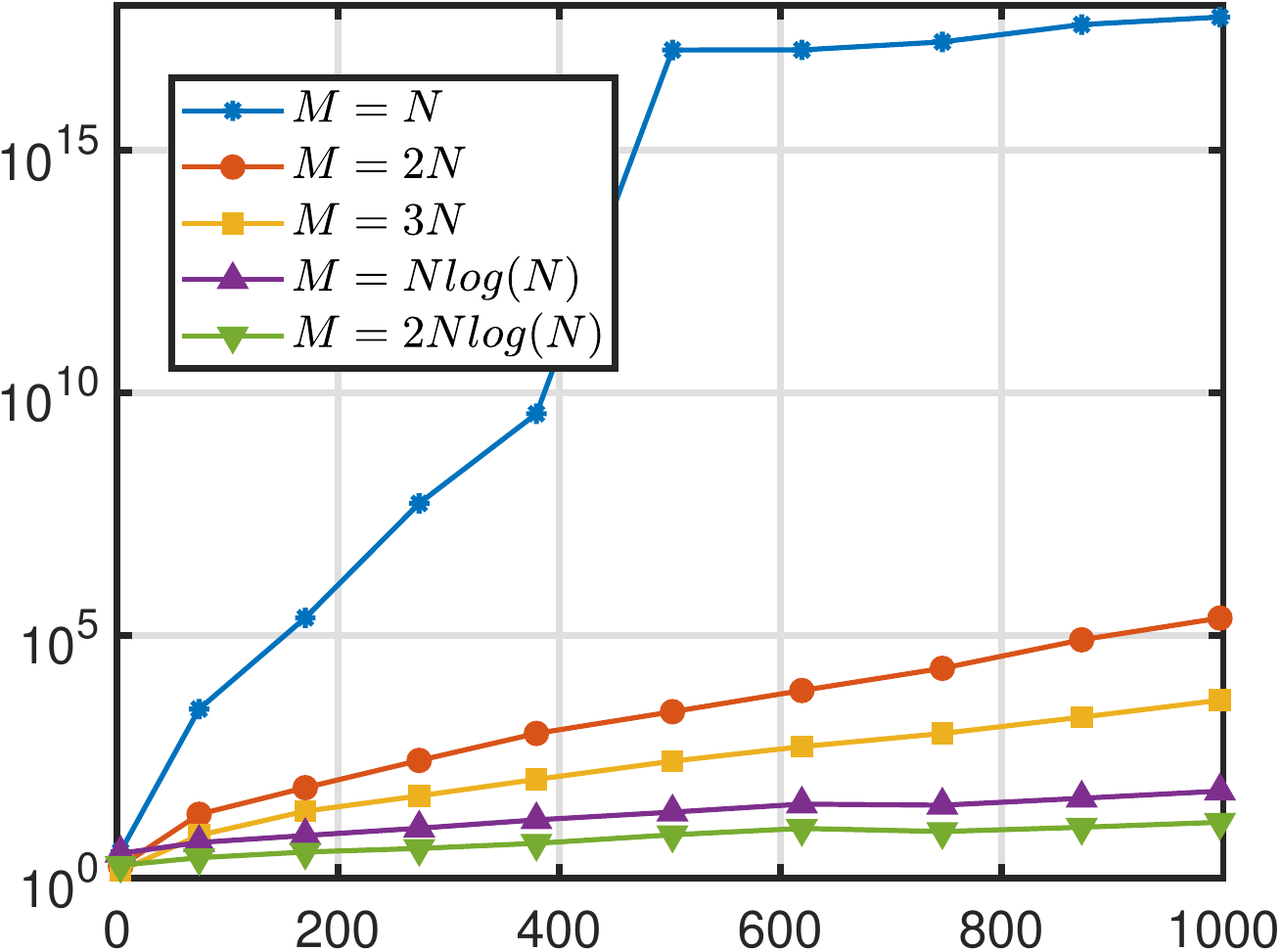} &
\includegraphics[scale=0.38]{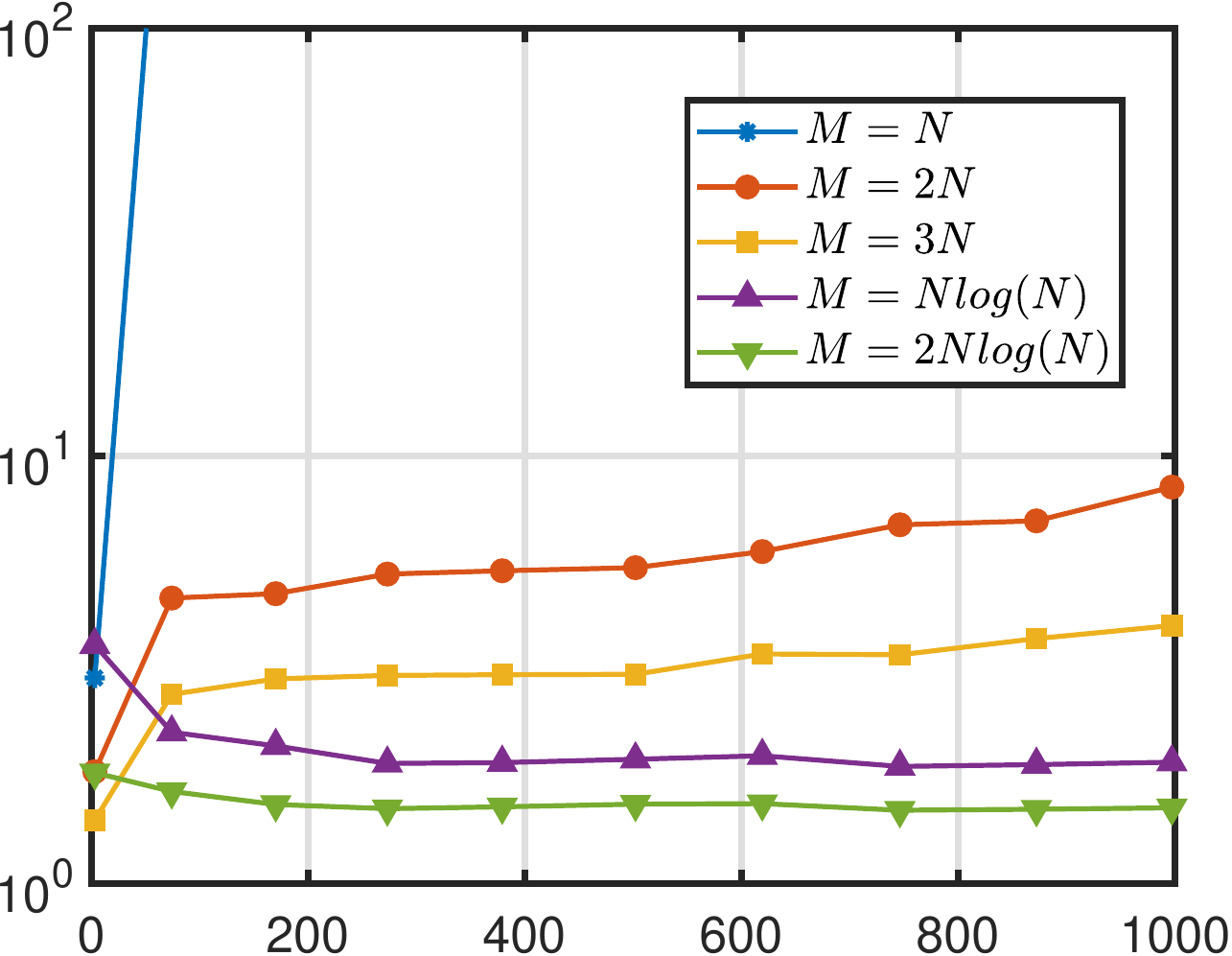} &
\includegraphics[scale=0.38]{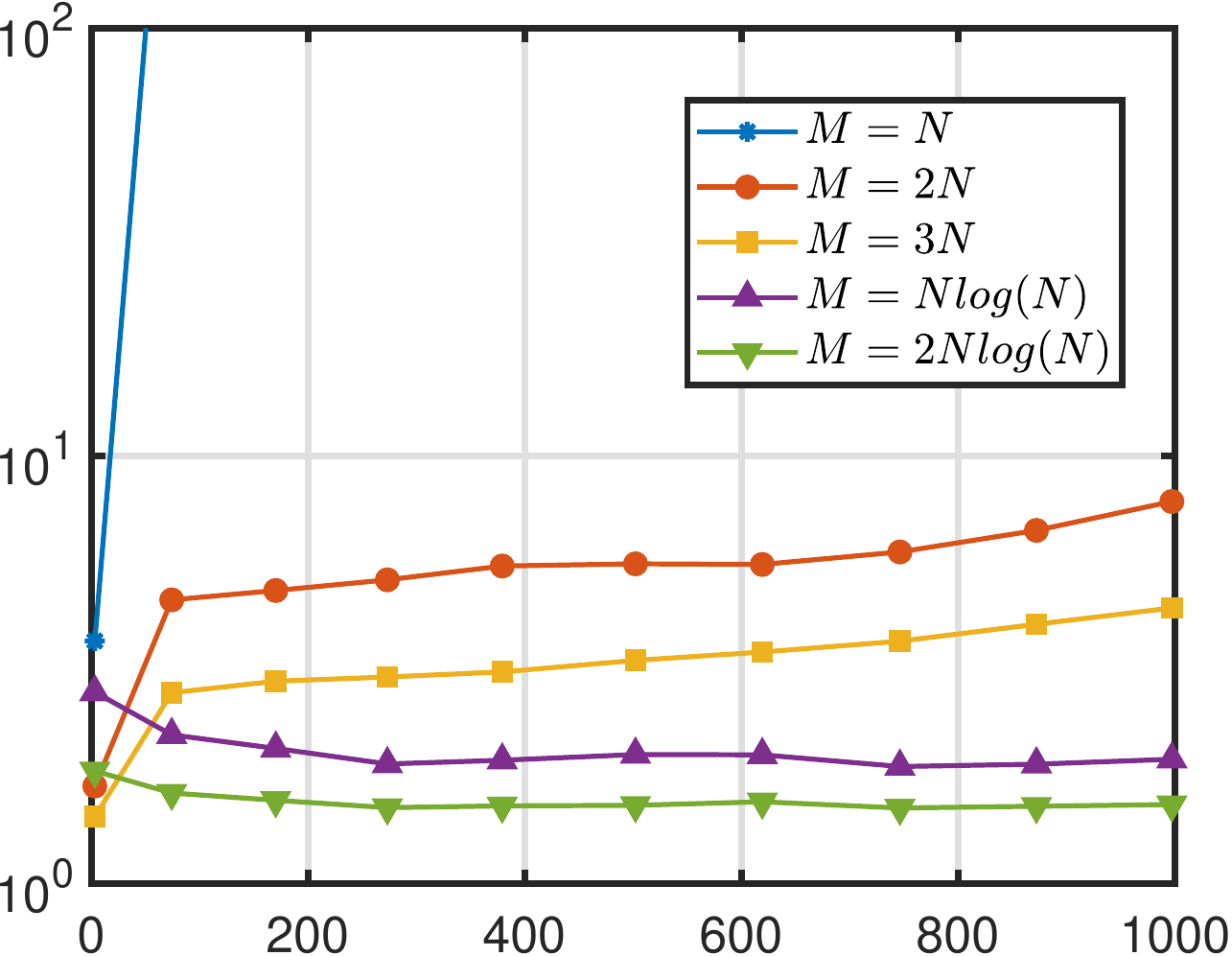}
\\  
\\  
\\  
Uniform & Method 1 & Method 2  
  \end{tabular}
  }
\end{center}
\caption{
The constant $\cC$ versus $N$ in $d = 2$ dimensions for the domains $\Omega_1$, $\Omega_2$, $\Omega_3$ (top to bottom) used in Fig.\ \ref{f:Message1}.
}
\label{f:Message3}
\end{figure}

{
\rem{
As discussed in \cite{adcock2018approximating} (see also \cite{BADHFramesPart2,BADHframespart}), the matrix $\bm{B}$ is in general ill-conditioned since the initial basis $\{ \psi_1,\ldots,\psi_N \}$ is nearly redundant for large $N$. Therein, this is tackled by using a truncated SVD. Conversely, this ill-conditioning is seemingly not an issue in Methods 1 and 2. As shown in \cite{MiglioratiIrregular}, an explanation for this is that the ill-conditioning of $\bm{B}$ does not prohibit generating a near-orthonormal basis for $P$ over the $K$-grid. Of course, if such a grid cannot be generated (see \S \ref{s:conclusion}) then one may have no choice but to work with the basis $\{ \psi_1,\ldots,\psi_N \}$ and the resulting ill-conditioning, as in \cite{adcock2018approximating}.
}
}

{
Thus far, the error has been computed via \R{Kgriderr} over the same grid of $K$ points used in the method (this corresponds to the first error bound in Theorem \ref{Thm_errorbound}). It is also important to consider the approximation's performance off this grid. In Fig.\ \ref{f:Tgriderr} we consider the error
\be{
\label{Tgriderr}
E_{\tilde{\tau}}(f) = \frac{\nmu{f - \tilde{f}}_{L^2(\Omega,\tilde{\tau})}}{\nmu{f}_{L^2(\Omega,\tilde{\tau})}},
}
where $\tilde{\tau}$ is a discrete measure over $T$ points, chosen randomly from the uniform measure on $\Omega$ and independently of the points used in the $K$-grid $\tau$. As this figure shows, when $K = 20000$, the error $E_{\tilde{\tau}}(f) $ behaves significantly worse than the error $E_{\tau}(f)$, which is shown in top row of Fig.\ \ref{f:Message1}. In view of Theorem \ref{Thm_errorbound}, this indicates that the $K$-grid is not large enough to ensure the constant $\cD$ is small.  As expected, doubling $K$ yields better behaviour for the error $E_{\tilde{\tau}}(f) $, and doubling it once more yields a further improvement.
}

\begin{figure}
\begin{center}
{\small
\begin{tabular}{ccc}
\includegraphics[scale=0.38]{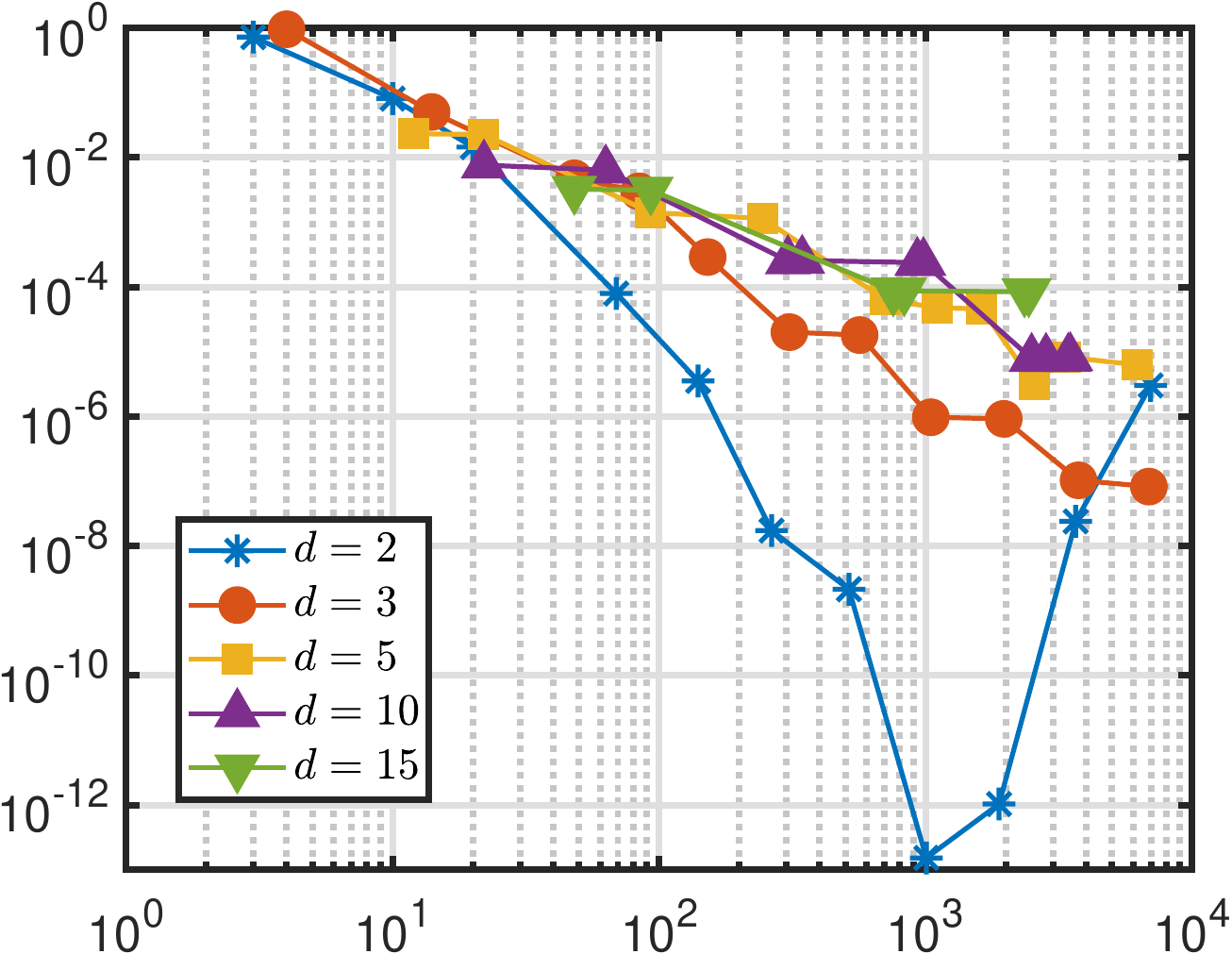}&
\includegraphics[scale=0.38]{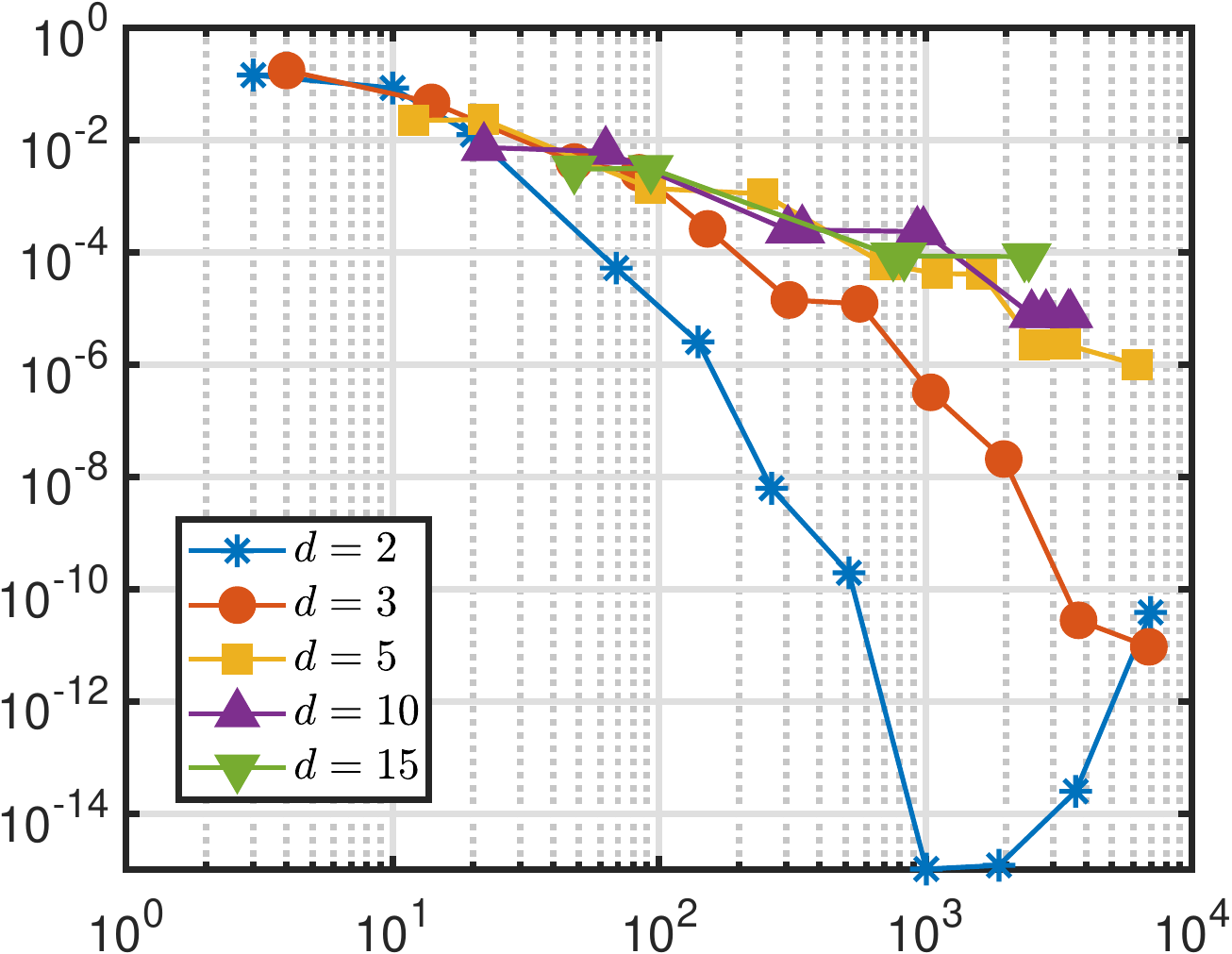} &
\includegraphics[scale=0.38]{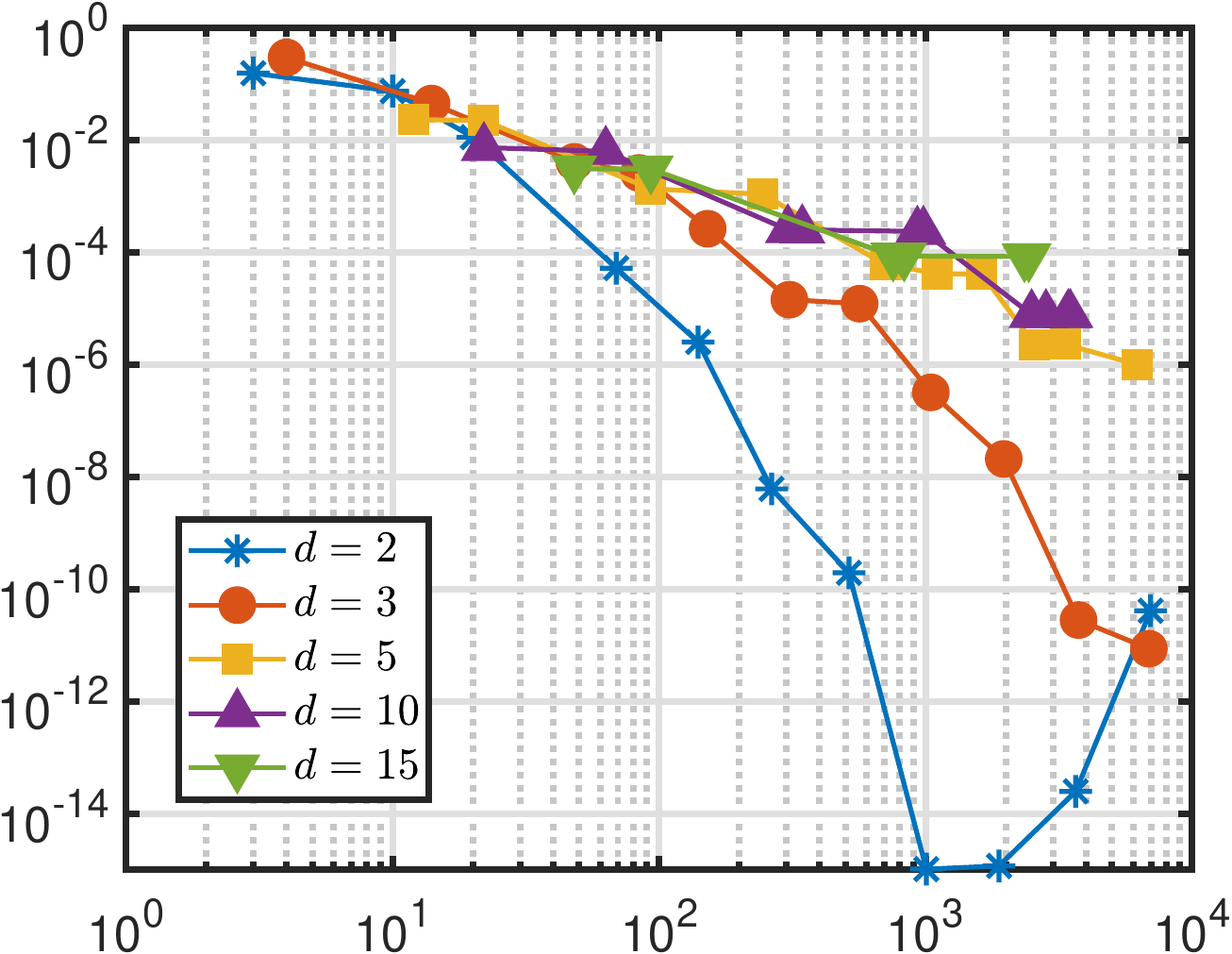} 
\\
\includegraphics[scale=0.38]{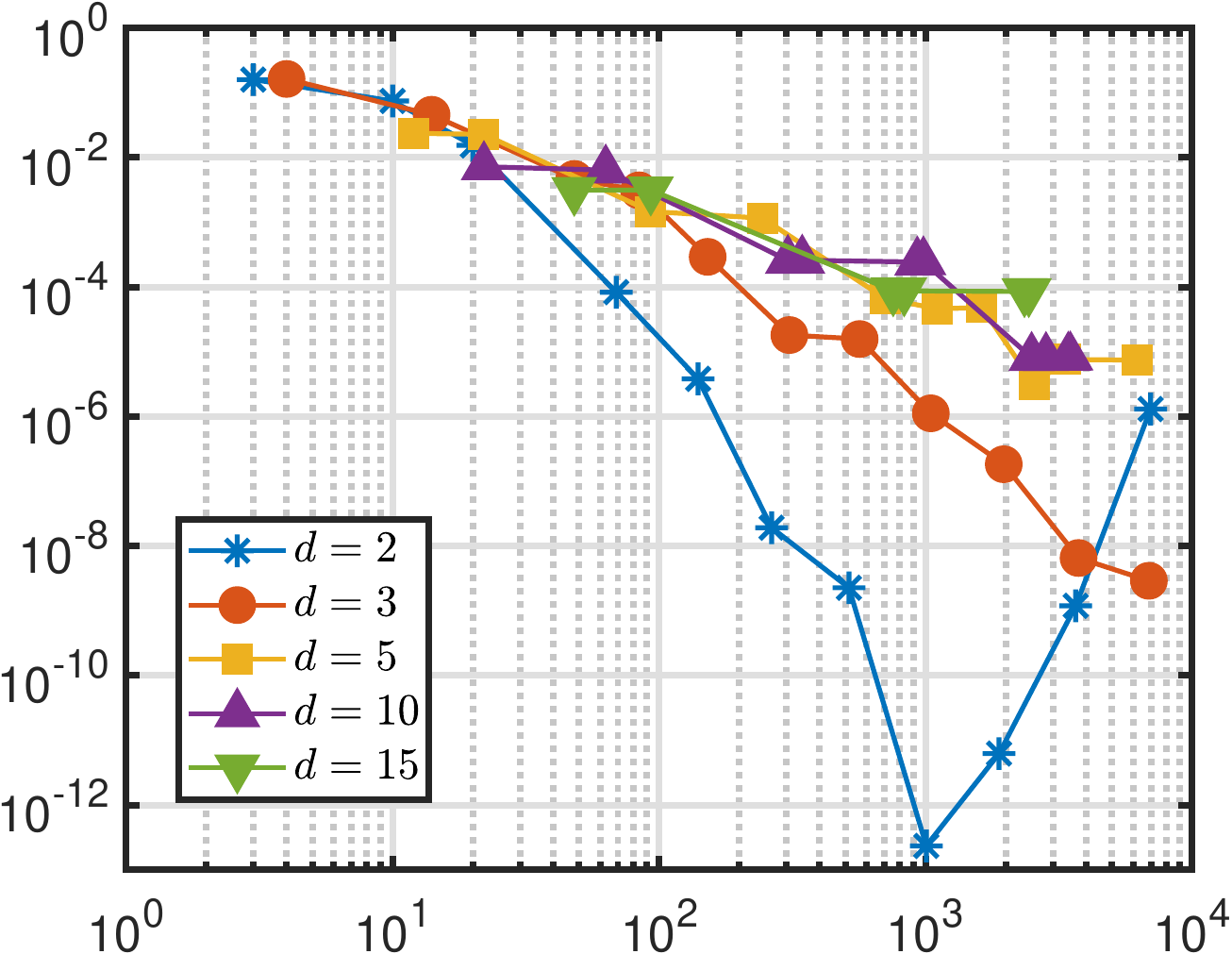}&
\includegraphics[scale=0.38]{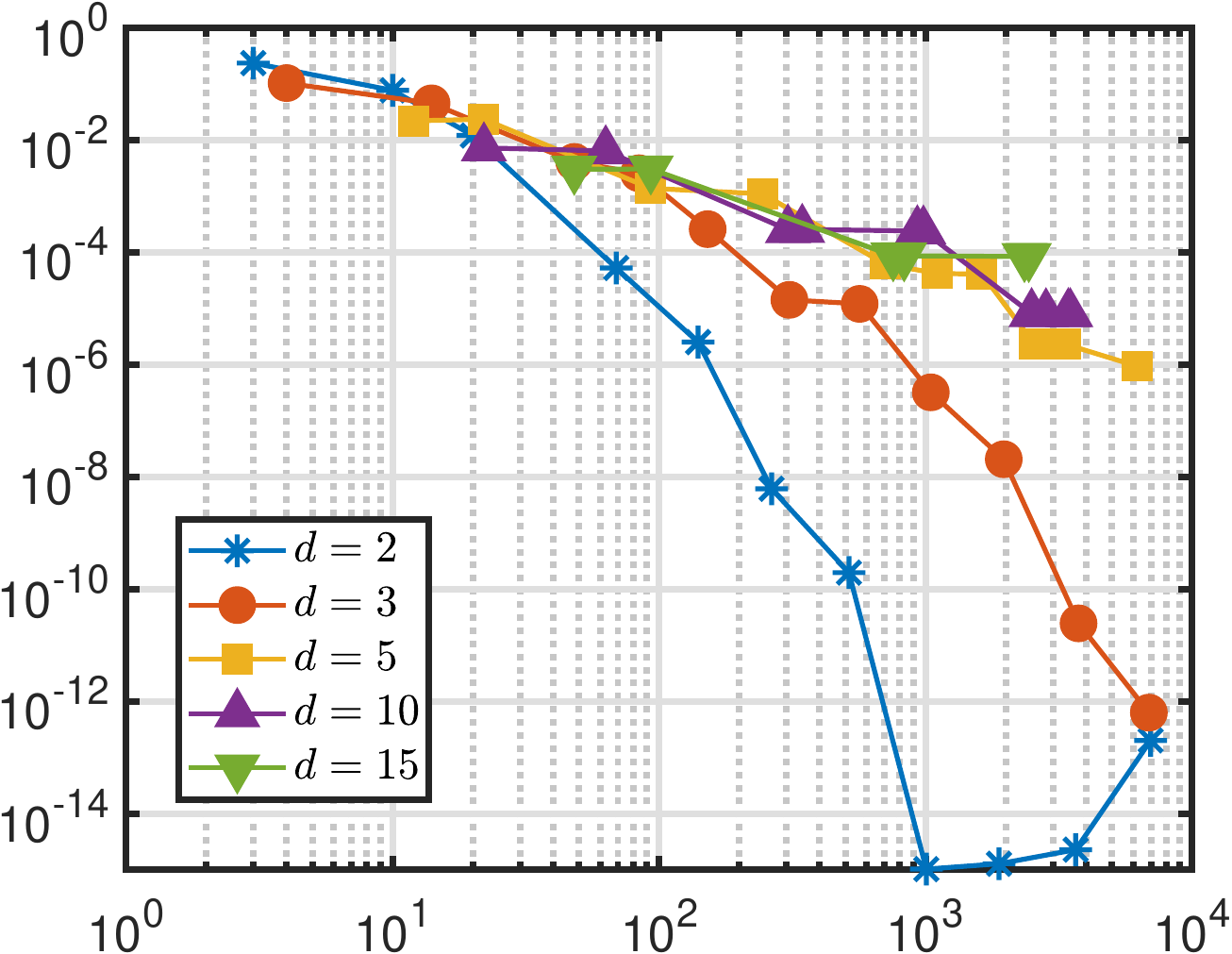} &
\includegraphics[scale=0.38]{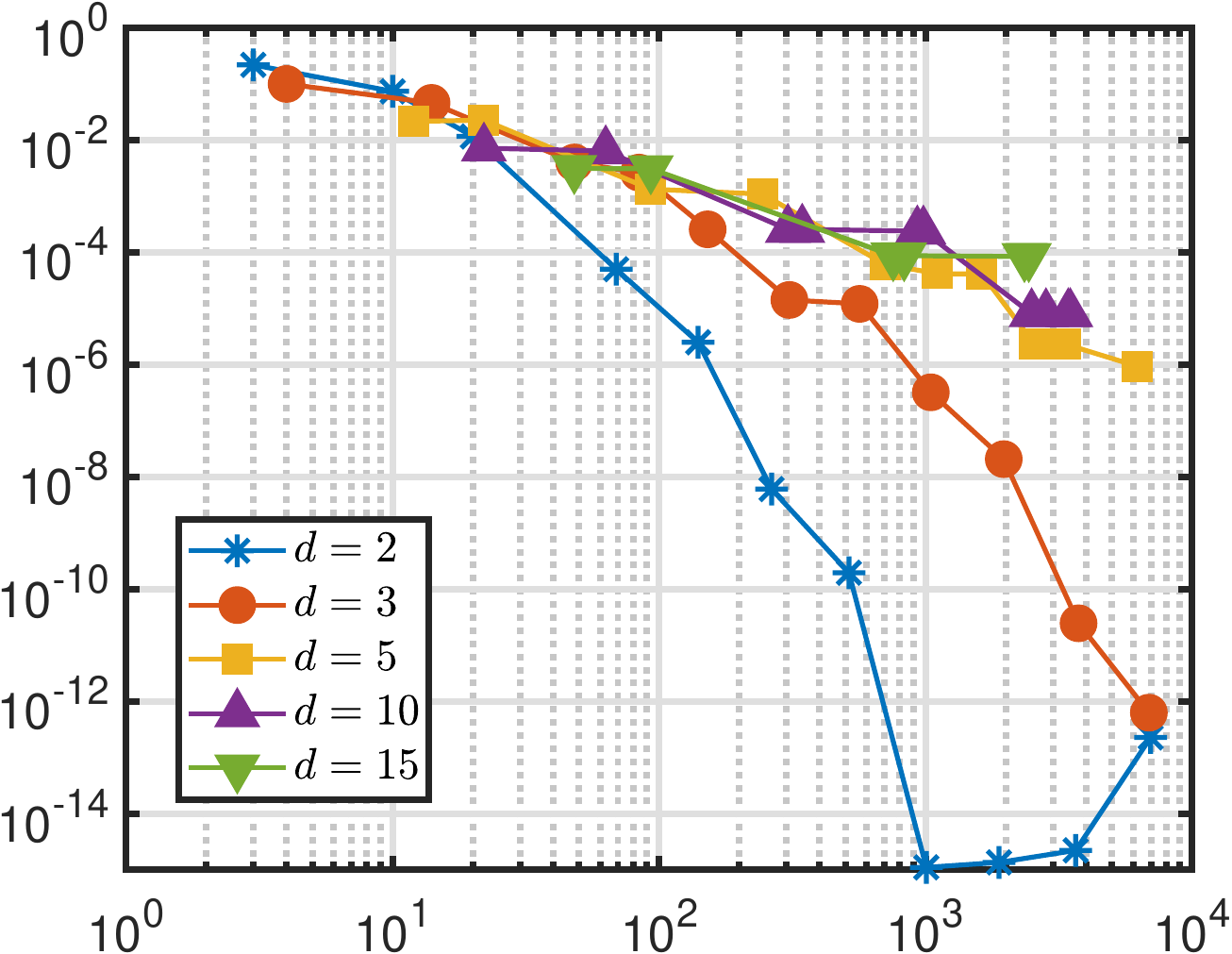} 
\\
\includegraphics[scale=0.38]{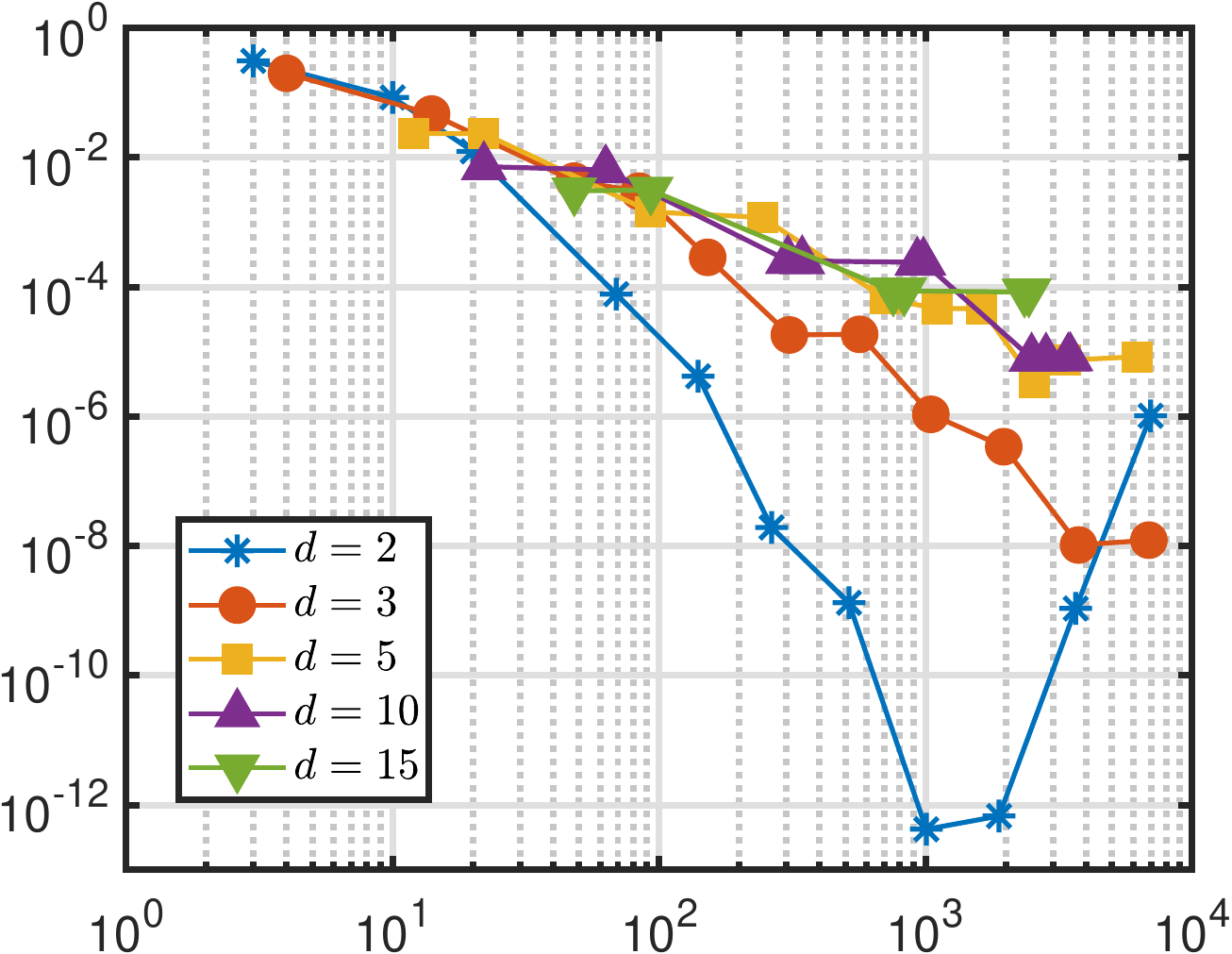}&
\includegraphics[scale=0.38]{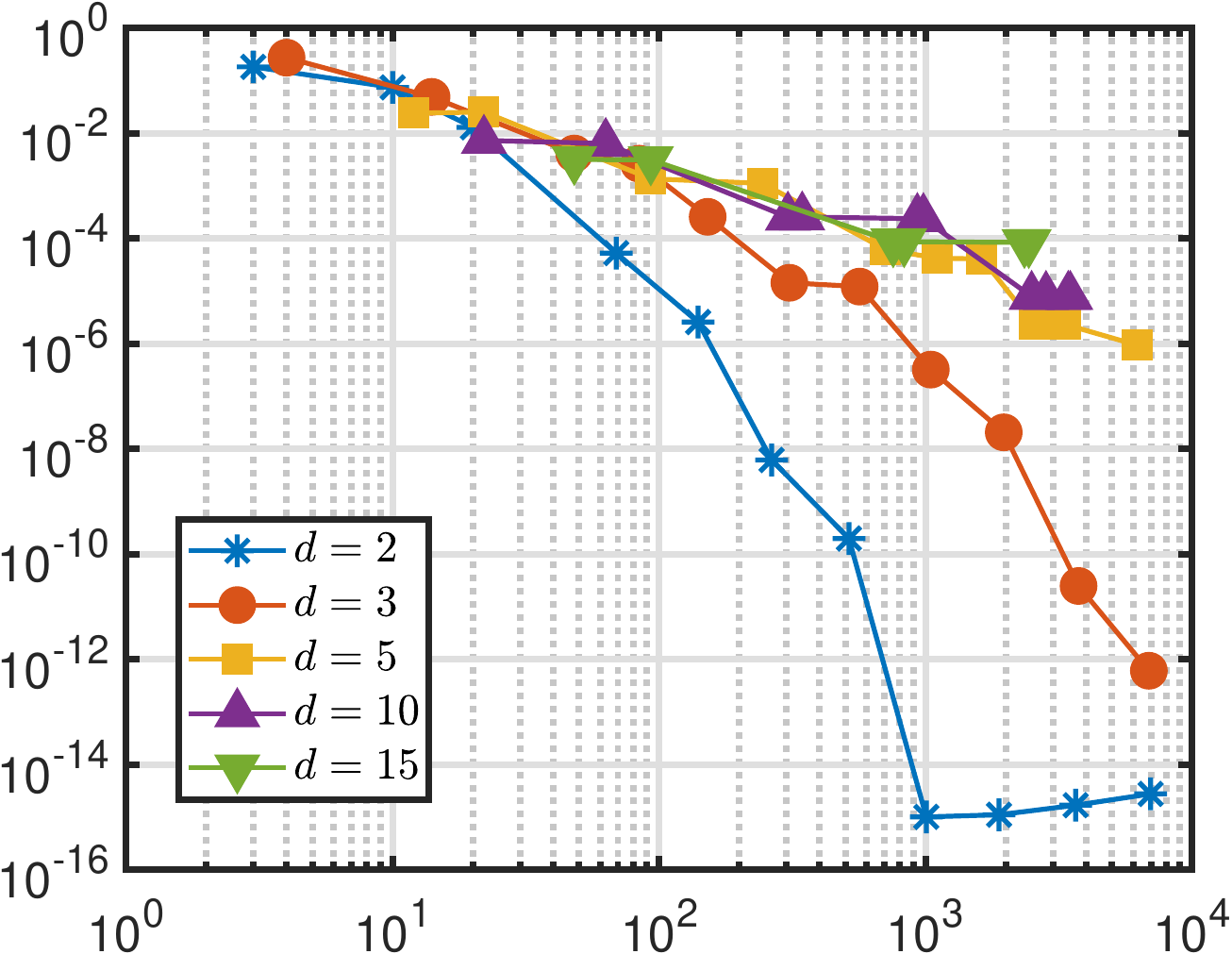} &
\includegraphics[scale=0.38]{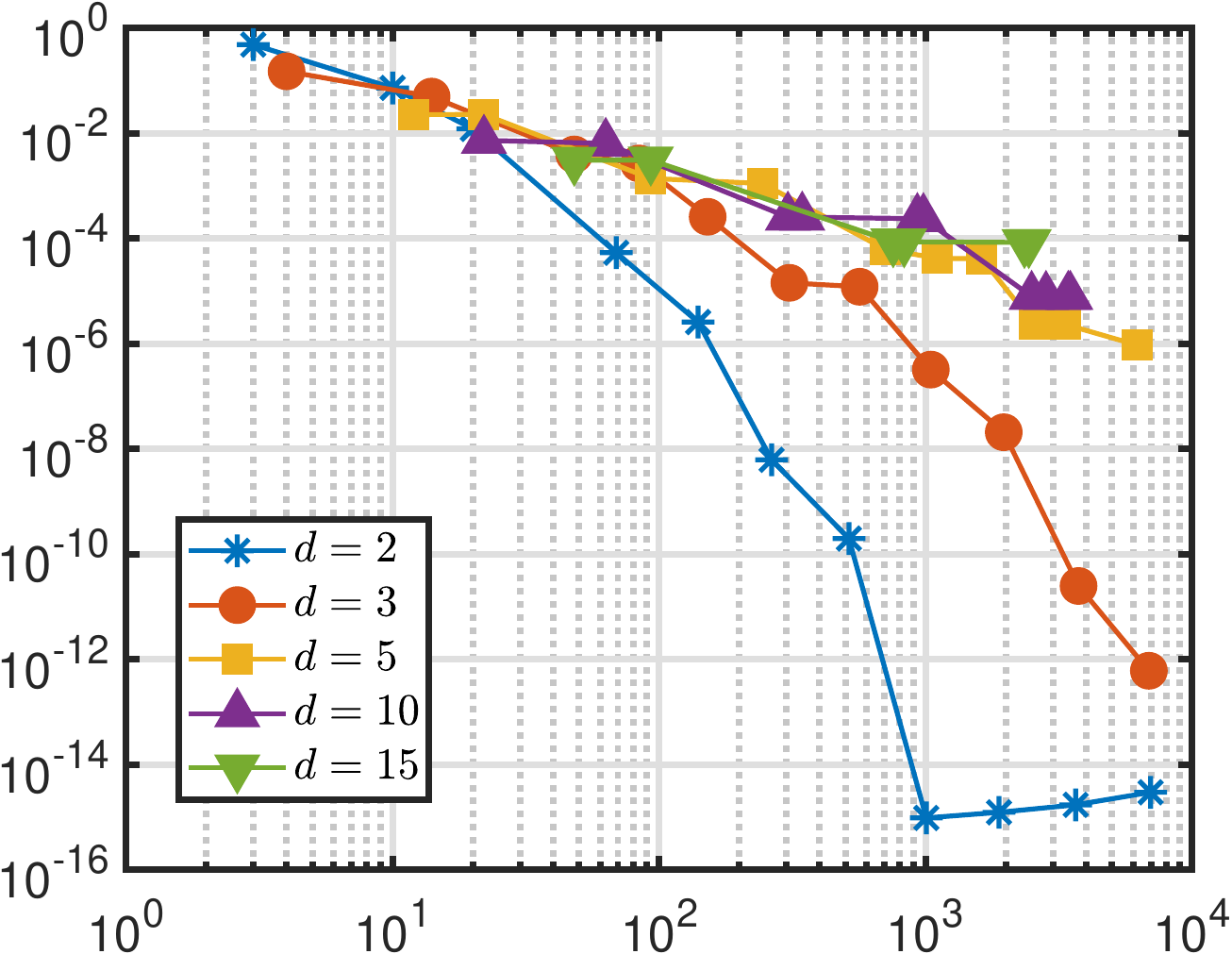}   
\\
Uniform & Method 1 & Method 2  
  \end{tabular}
  }
\end{center}
\caption{
{The error $E_{\tilde{\tau}}(f)$ versus $M$, with $M$ chosen as the smallest value such that $M \geq N \log(N)$, for $\Omega_1 = \left \{ \bm{y} : 1/4 \leq \nm{\bm{y}}_2 \leq 1 \right \}$ and $f = f_1$.  Top row: $T = K = 20000$.  Middle row: $T = K = 40000$. Bottom row: $T = K = 80000$.}
}
\label{f:Tgriderr}
\end{figure}

{This experiment demonstrates one of the challenges with approximation over general domains. Namely, while it is straightforward to check for good accuracy over the $K$-grid (indeed, one simply computes $\cC$), it is difficult to ensure \textit{a priori} good performance over the whole of $\Omega$.  Indeed, the constant $\cD$ is not straightforwardly computable. As noted, if $\Omega$ has the $\lambda$-rectangle property, then one may use Propositions \ref{p:Ksize} and \ref{p:Klambdarect} to estimate $K$ (we caution, however, this estimate may not be particularly useful if $\lambda$ is small). But for domains not satisfying this property, as is the case for the domain considered in Fig.\ \ref{f:Tgriderr},  the answer to the question of how large $K$ should be to ensure good accuracy over $\Omega$ is currently unknown.}

Finally, in Fig.\ \ref{f:Message4} we demonstrate a curious phenomenon: in certain cases, there may be far less benefit from using these methods over uniform sampling.  In this experiment we consider an annular domain with outer radius $1/2$.  In this case, quite in contrast to what was seen for the annular domain $\Omega_1$ (which has outer radius $1$) in Fig.\ \ref{f:Message1}, the approximation converges.  Furthermore, neither Method 1 nor Method 2 achieves a better rate of convergence.  Fig.\ \ref{f:Message4C} shows the constant $\cC$ for all three methods when $d = 2$.  {Unlike in the case of $\Omega_1$} (see Fig.\ \ref{f:Message3}) {for Uniform the constant $\cC$} remains bounded {when log-linear sampling is used}, although it is several orders of magnitude larger than for Methods 1 and 2 {with} the same scaling.  This phenomenon relates to the fact that the domain $\Omega$ in this case is compactly contained in $(-1,1)^d$.  Hence the Legendre polynomials on $[-1,1]^d$, when restricted to $\Omega$ constitute a frame.  See {\cite[Sec.\ 8.2]{adcock2018approximating}} for further discussion.

\begin{figure}
\begin{center}
{\small
\begin{tabular}{ccc}
\includegraphics[scale=0.38]{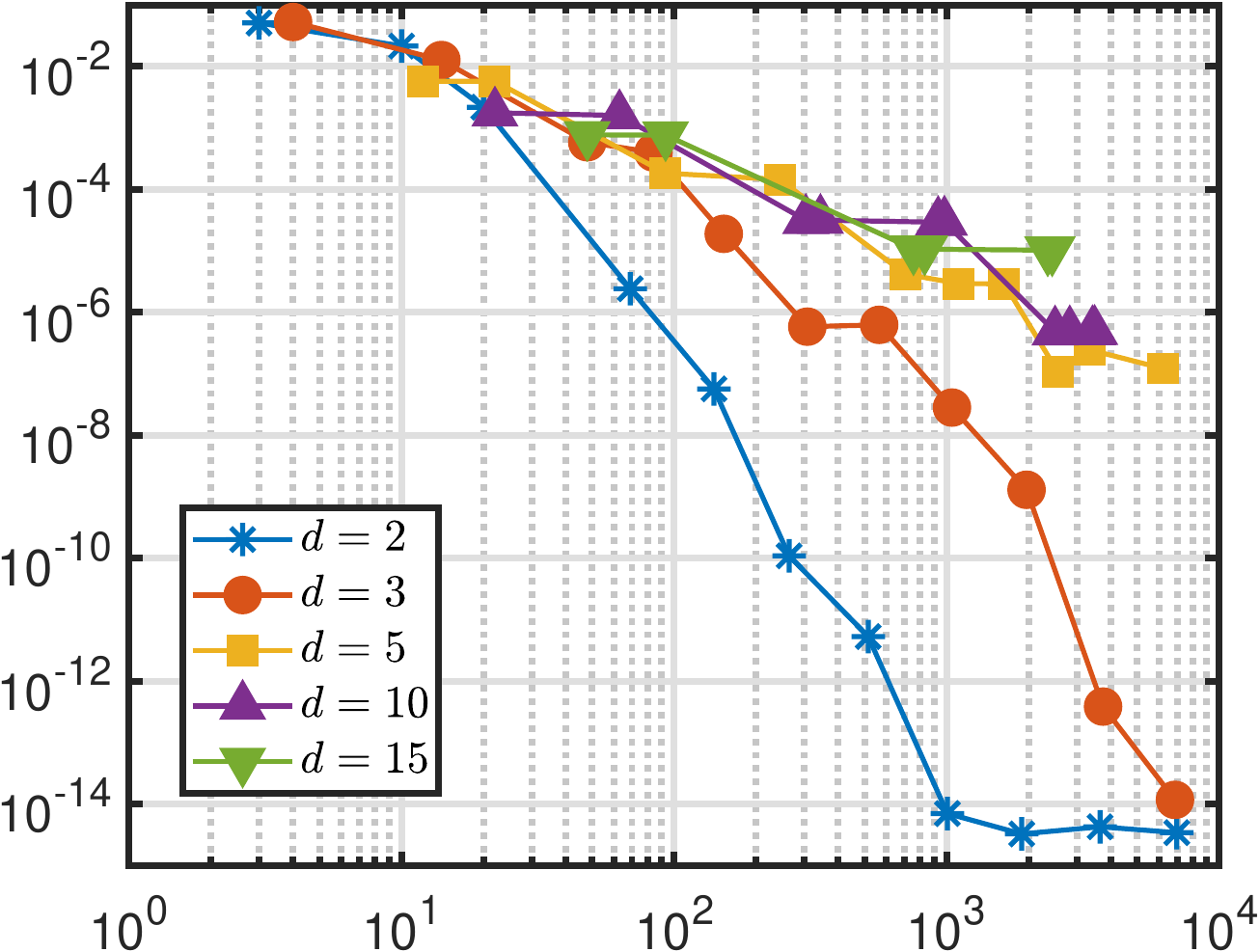} &
\includegraphics[scale=0.38]{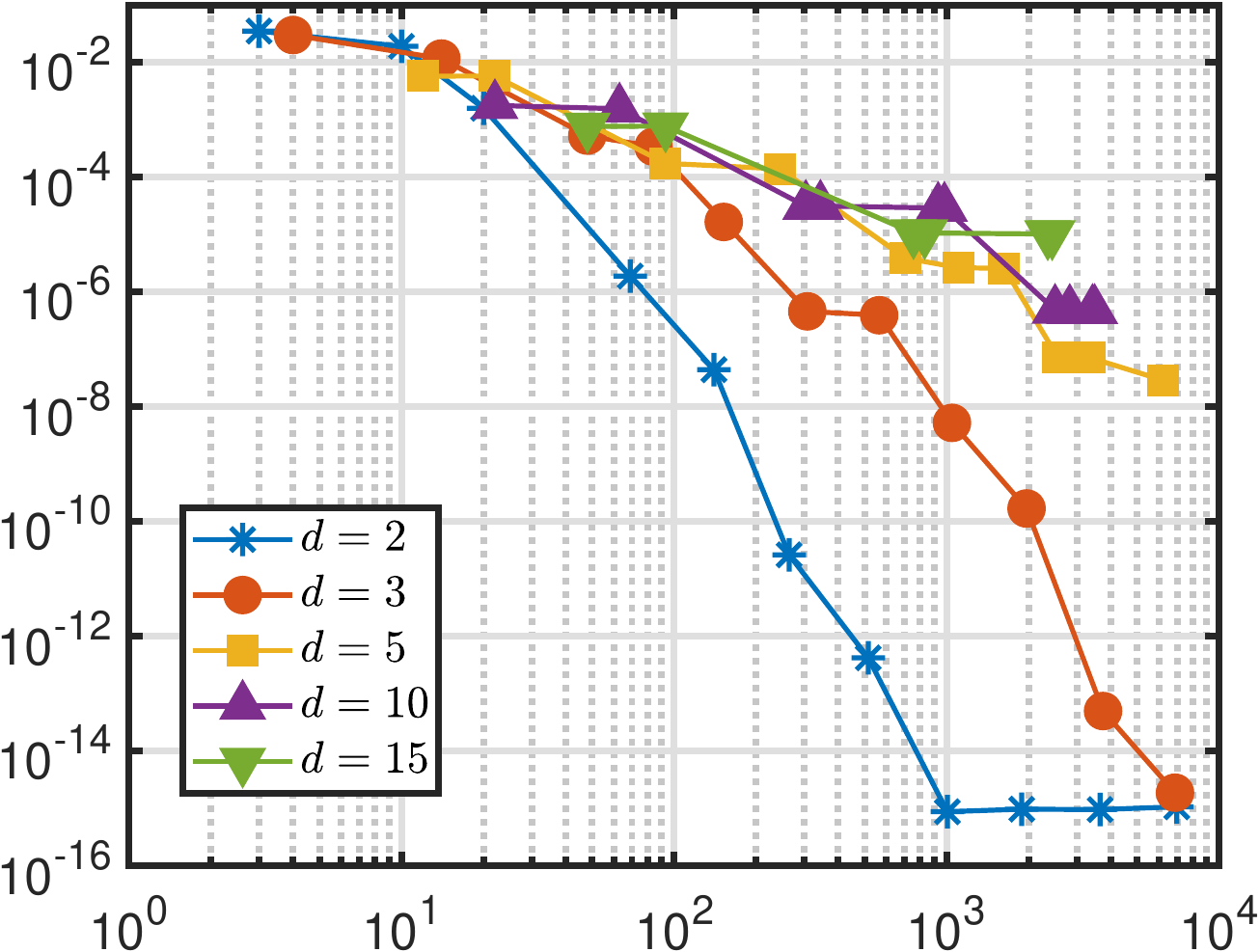} &
\includegraphics[scale=0.38]{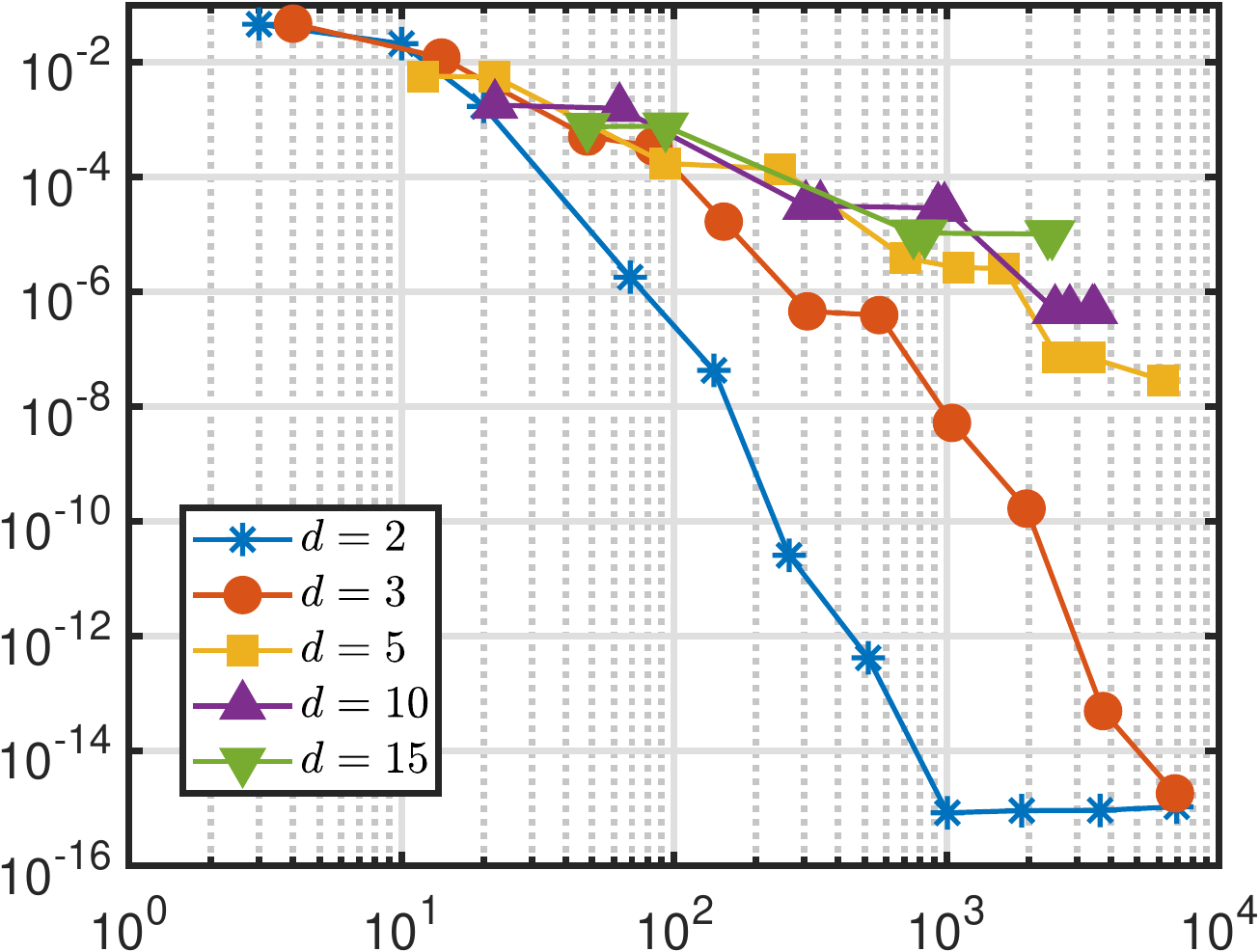}
  \\  
Uniform & Method 1 & Method 2  
  \end{tabular}
  }
\end{center}
\caption{
The error {$E_{\tau}(f)$} versus $M$ for the domain $\Omega = \left \{ \bm{y} : 1/8 \leq \nm{\bm{y}}_2 \leq 1/2 \right \}$ with $f = f_1$.
}
\label{f:Message4}
\end{figure}

\begin{figure}
\begin{center}
{\small
\begin{tabular}{ccc}
\includegraphics[scale=0.38]{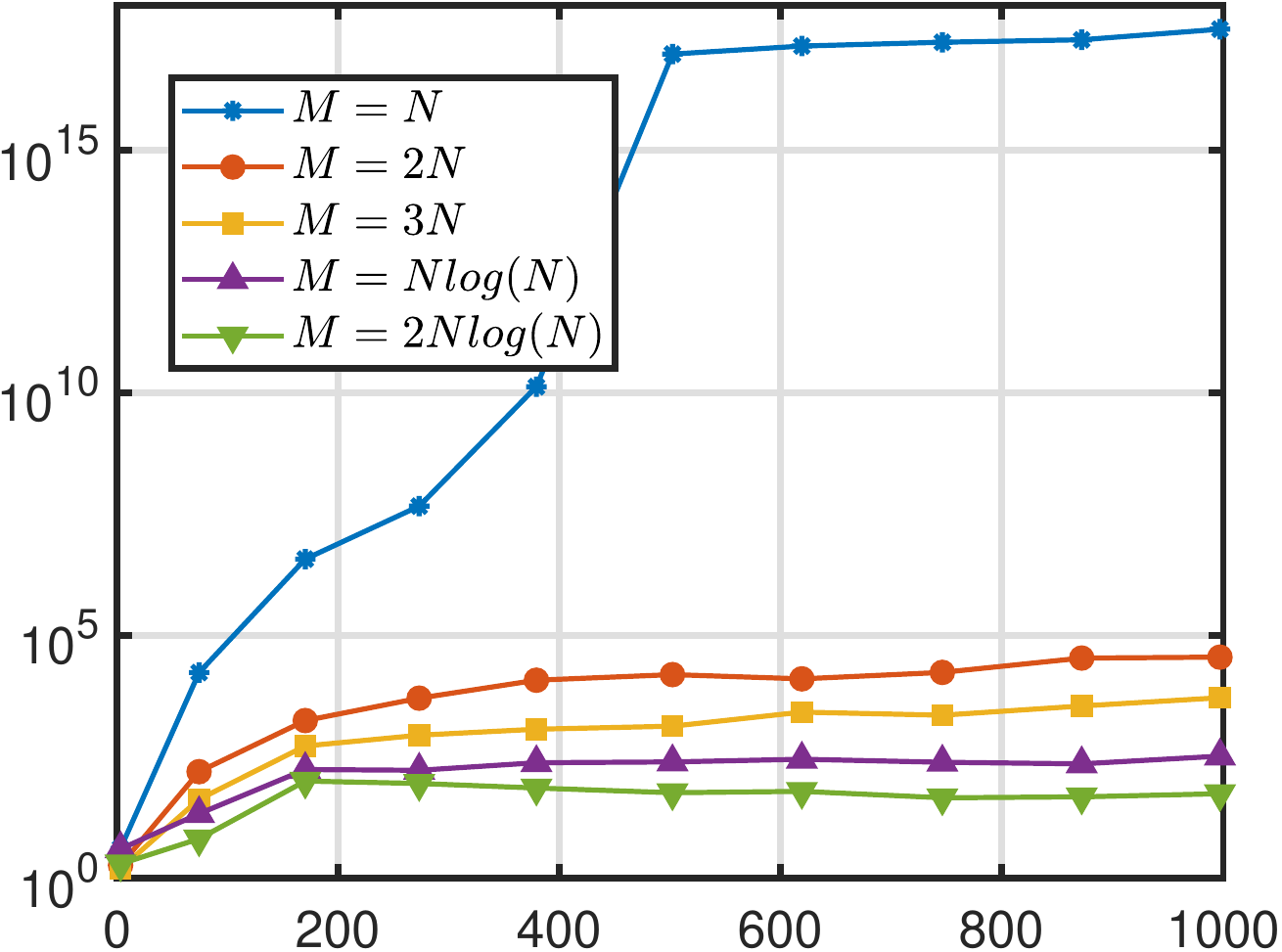} &
\includegraphics[scale=0.38]{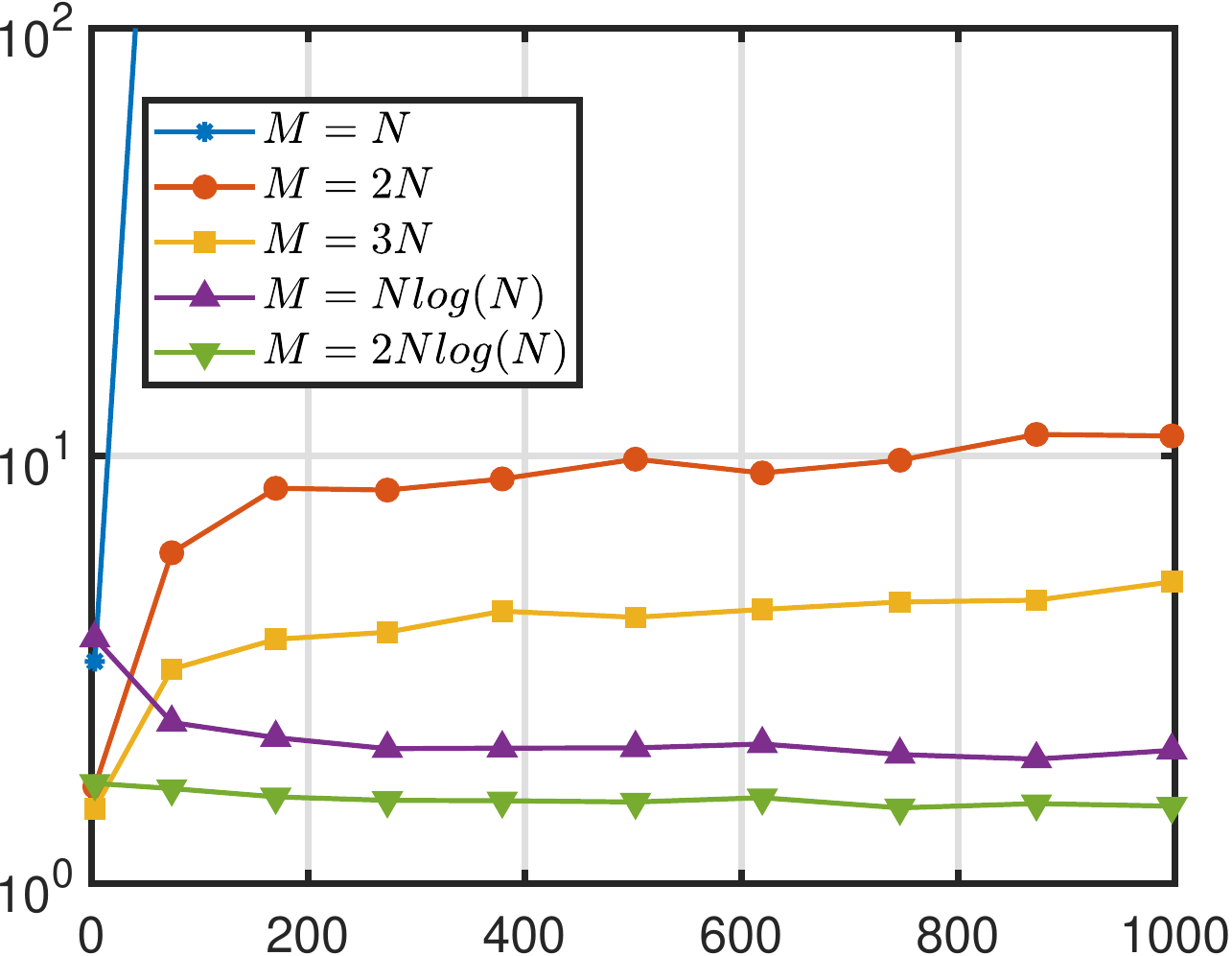} &
\includegraphics[scale=0.38]{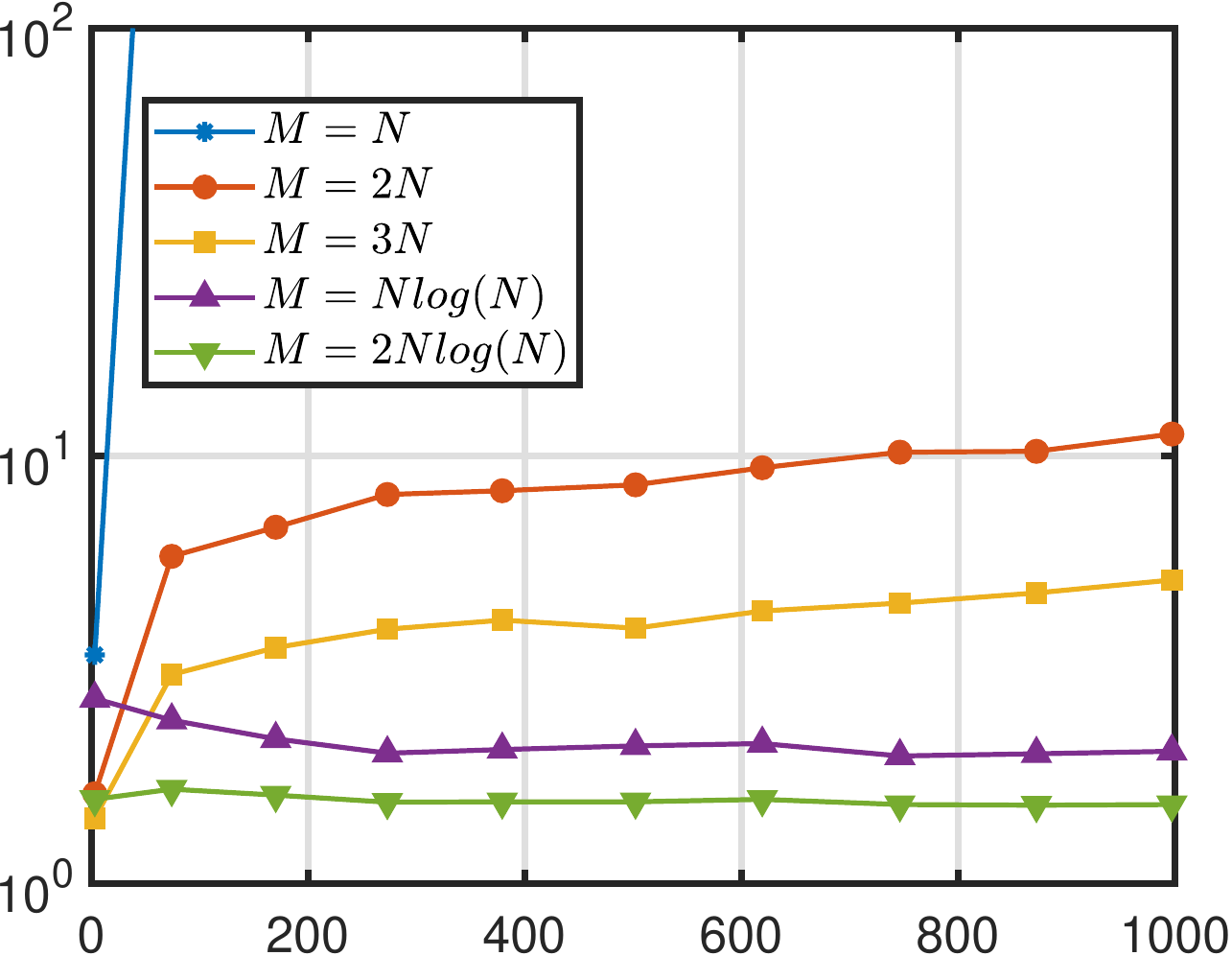}
  \\  
Uniform & Method 1 & Method 2  
  \end{tabular}
  }
\end{center}
\caption{
The constant $\cC$ versus $N$ for the domain $\Omega = \left \{ \bm{y} : 1/8 \leq \nm{\bm{y}}_2 \leq 1/2 \right \}$ with $d = 2$.
}
\label{f:Message4C}
\end{figure}

\section{Conclusions}\label{s:conclusion}

In this paper, we introduced new methods for optimal weighted least-squares approximations in general domains and arbitrary finite-dimensional spaces based on a discrete grid $Z$.  As we showed, under the log-linear sample complexity $M \gtrsim N \log(N)$, these methods are provably well conditioned and accurate over the grid. {This improves on the method introduced in \cite{adcock2018approximating}, for which the sampling complexity at best quadratic in $N$.}  To obtain accuracy over the original space, a random grid is used, whose size $K$ is related to the Nikolskii constant $\cN(P,\rho)$.  For domains possessing the so-called $\lambda$-rectangle property the quadratic scaling $K \gtrsim N^2 \log(N)$ is sufficient.  We introduce two versions of this method, with the latter recycling all its samples when the approximation $P$ is augmented.  Numerical experiments confirm the benefits of these methods over standard sampling.

{
As noted in the previous section, a shortcoming of this method is ensuring $K$ is large enough to guarantee a good approximation over $\Omega$ when $\Omega$ does not satisfy the $\lambda$-rectangle property.  Obtaining estimates for general domains is an open problem. Another limitation is that it must be possible to sample from the continuous measure $\rho$ over $\Omega$ in order to generate the $K$-grid.  In particular, and unlike in the method introduced in \cite{adcock2018approximating}, $\Omega$ must be known in advance in some suitable sense, and even then, procedures such as rejection sampling may become prohibitively expensive in high dimensions.  In practice, $\Omega$ may not be known in advance: see \cite{adcock2018approximating}, for instance, for examples motivated by uncertainty quantification where $\Omega$ can only be `learned' as the samples of $f(\bm{y})$ are taken.  An objective of future work is to investigate whether the procedure developed in this paper can be extended to certain settings where generating a fine grid beforehand is not feasible.
}

\section*{Acknowledgments}
This work was supported by the PIMS CRG in ``High-dimensional Data Analysis'' and by NSERC through grant R611675.


\bibliographystyle{abbrv}
\small
\bibliography{Optimal_SamplingRefs}

\end{document}